\documentclass[11pt, leqno]{article}
\usepackage{amsmath}
\usepackage{color}
\usepackage{amssymb}
\usepackage{latexsym}
\usepackage{enumerate}
\usepackage{amsthm}

\topmargin by -\headsep
\usepackage{fullpage}

     {\begin{list}{}%
             {\setlength{\leftmargin}{#1}}%
             \item[]%
     }
     {\end{list}}

\newtheorem{theorem}{Theorem}[section]
\newtheorem{thm}[theorem]{Theorem}
\newtheorem{definition}[theorem]{Definition}
\newtheorem{lemma}[theorem]{Lemma}

\newtheorem{abbreviatio}[theorem]{Abbreviations}

\newtheorem{observation}[theorem]{Observation}

\newtheorem{lem}[theorem]{Lemma}

 \newtheorem{deff}[theorem]{Definition}

\newtheorem{rem}[theorem]{Remark}

\newtheorem{cor}[theorem]{Corollary}

\newtheorem{remark}[theorem]{Remark}

\newcommand*{\IKP}{\textbf{IKP}}

\newcommand*{\dotin}{\mathrel{\dot{\in}}}

\newcommand{\force}[2]{\large{\Vdash}^{\negthickspace\negthinspace^{#1}}_{\negthickspace \negthinspace_{#2}}}

\newcommand{\goed}[1]{\ulcorner #1\urcorner}

\newcommand{\AC}{{\mathbf{AC}}}

\newcommand{\dbi}{\mbox{
${\mathbf\Delta}^{ 1}_{ 2}$--${\mathbf{CA}}+{\mathbf{BI}}$}}
\newcommand{\KP}{{\mathbf{KP}}}

\newcommand{\OO}{\Omega}

\newcommand{\BI}{{\mathbf{BI}}}

\newcommand{\beq}{\begin{eqnarray}}
\newcommand{\eeq}{\end{eqnarray}}

\newcommand{\RR}{\mathsf{R}}

\newcommand{\LL}{{\mathbf{L}}}

\newcommand{\prf}{{\bf Proof\/}: }

\newcommand{\Cut}{\text{(Cut)}}
\newcommand{\GA}{\Gamma}

\newcommand{\CH}{{\mathcal H}}

\newcommand{\In}{\in}

\newcommand{\lev}[1]{{\mid}\,{#1}\,{\mid}}

\newcommand{\om}{\omega}

\newcommand{\al}{\alpha}
\newcommand{\nIn}{\!\notin\!}

\newcommand{\infone}[8]
{\rule[-0.5cm]{0cm}{1.5cm}
\begin{array}{c}
    \provx {{#1}} {{#2}} {{#3}} {{#4}}\rule{0cm}{5mm}\\
    \hline \rule{0cm}{5mm}
    \provx {{#5}} {{#6}} {{#7}} {{#8}}
 \end{array}}

\newcommand{\inftwo}[9]
{\rule[-0.5cm]{0cm}{1.5cm}
 \begin{array}{c}
    \provx {{#1}} {{#2}} {{}} {{#3}}\qquad \provx {{#4}} {{#5}} {}
{{#6}} \rule{0cm}{5mm}
\\
    \hline  \rule{0cm}{5mm}
    \provx {{#7}} {{#8}} {} {{#9}}
 \end{array}}

\newcommand{\ifthree}[3]
{\rule[-0.5cm]{0cm}{1.5cm}
 \begin{array}{c}
     {{#1}} \qquad {{#2}}
 \rule{0cm}{5mm}
\\
    \hline  \rule{0cm}{5mm}
{{#3}}
 \end{array}}

\newcommand{\und}{\,\wedge\,}

\newcommand{\PA}{{\mathbf{PA}}}

\newcommand{\psioi}{\psi_{\Omega}}

\newcommand{\TI}{{\mathrm{TI}}}

\newcommand{\ZF}{{\mathbf{ZF}}}
\newcommand{\CZF}{{\mathbf{CZF}}}

\newcommand{\KPP}{{\mathbf{KP}}({\mathcal P})}
\newcommand{\Deltaop}{\Delta_0^{\mathcal P}}

\newcommand{\Sigmap}{\Sigma^{\mathcal P}}
\newcommand{\Sigmaop}{\Sigma^{\mathcal P}}

\newcommand{\PC}{{\mathcal P}}

\newcommand{\subb}{\subseteq}
\newcommand{\Vb}[1]{{\mathbb V}_{#1}}
\newcommand{\Va}{{\mathbb V}_{\alpha}}
\newcommand{\SRP}{(\Sigma^{\mathcal P}\mbox{-}Ref)}
\newcommand{\RSOP}{RS_{\Omega}^{\mathcal P}}

\newcommand{\POW}{{\mathbf{Pow}}}

\mathchardef\str='1066
\def\negprov#1#2#3{
\setbox1=\hbox{\kern1.5pt$\scriptstyle#2$} \setbox4=\hbox{$\str$}
\def\zeichen{#1}
\ifx\zeichen\empty\setbox0=\hbox to 1em{}\else\setbox0=\hbox
{\kern1.5pt$\scriptstyle#1$}\fi \dimen1=\dp0 \ifdim \dimen1=0pt
\advance \dimen1 by 1.5ex \else \advance \dimen1 by 1.2ex
\fi\dimen3=2ex\dimen4=.5ex\ifdim \wd0<\wd1 \dimen2=\wd1 \else
\dimen2=\wd0
\fi\hbox{\hskip.5em$\kern-1.9pt\raise1pt\copy4\kern-\wd4\kern1.9pt\vrule
height\dimen3 depth\dimen4\raise\dimen1\copy0\hskip-1\wd0
\lower\ht1\copy1\hskip-1\wd1\vrule width\dimen2 height.7ex
depth-.6ex \hskip3pt minus1.5pt#3\hskip2pt plus2pt minus2pt$}}

\def\mod#1#2{
\def\zeichen{#1}
\hbox{\hskip 2pt plus3pt minus 2pt\vrule width.5pt height2ex
depth.5ex \vbox{\ifx\zeichen\empty\hbox to .75em{}\else
\hbox{\kern1.5pt $\scriptstyle#1$}\fi \kern2pt \hrule \kern1.7pt
\hrule\kern1.7pt} \hskip3pt minus 2pt$#2$}\hskip2pt plus3pt
minus2pt}

\def\notmod#1#2{\hbox{\hskip 2pt plus 3pt minus 3pt\vrule width.5pt
height2ex depth.5ex \vbox{\hbox{\kern1.5pt
$\scriptstyle#1$}\kern3pt \setbox0=\hbox{\kern2pt$\scriptstyle/$}
\hrule \kern-1.7pt \copy0 \kern-\ht0 \kern 1.7pt
\hrule\kern1.7pt}n \hskip3pt minus 2pt$#2$}\hskip2pt plus3pt
minus2pt}
\def\sq{\hbox{\rlap{$\sqcap$}$\sqcup$}}
\def\qed{\ifmmode\sq\else{\unskip\nobreak\hfil\penalty50\hskip1em\null
\nobreak\hfil\sq\parfillskip=0pt\finalhyphendemerits=0\endgraf}\fi\medskip}

\def\lleq{\hbox{\hskip3pt minus3pt\kern1pt\lower4pt
\vbox{\hbox{$\scriptstyle\ll$}
\kern-7pt\hbox{\kern1pt$\scriptstyle=$}}\hskip3pt minus 3pt}}

\mathchardef\res='1152 \mathchardef\qin='1062
\mathchardef\qprec='1036 \mathchardef\qless='474
\mathchardef\dpkt='72

\def\EnableBpAbbreviations{%
    \let\AX\Axiom
    \let\AXC\AxiomC
    \let\UI\UnaryInf
    \let\UIC\UnaryInfC
    \let\BI\BinaryInf
    \let\BIC\BinaryInfC
    \let\TI\TrinaryInf
    \let\TIC\TrinaryInfC
    \let\QI\QuaternaryInf
    \let\QIC\QuaternaryInfC
    \let\QuI\QuinaryInf
    \let\QuIC\QuinaryInfC
    \let\LL\LeftLabel
    \let\RL\RightLabel
    \let\DP\DisplayProof
}

\def\ScoreOverhang{4pt}         % How much underlines extend out
\def\ScoreOverhangLeft{\ScoreOverhang}
\def\ScoreOverhangRight{\ScoreOverhang}

\def\extraVskip{2pt}            % Extra space above and below lines
\def\ruleScoreFiller{\hrule}        % Horizontal rule filler.

\def\defaultScoreFiller{\ruleScoreFiller}  % Default horizontal filler.
\def\defaultBuildScore{\buildSingleScore}  % In \singleLine mode at start.

\def\defaultHypSeparation{\hskip.2in}   % Used if \insertBetweenHyps isn't given

\def\labelSpacing{3pt}      % Horizontal space separating labels and lines

\def\proofSkipAmount{\vskip.8ex plus.8ex minus.4ex}
\ifx\fCenter\undefined
\def\fCenter{\relax}
\fi

\def\theHypSeparation{\defaultHypSeparation}
\def\alwaysScoreFiller{\defaultScoreFiller} % Horizontal filler.
\def\alwaysBuildScore{\defaultBuildScore}
\def\theScoreFiller{\alwaysScoreFiller} % Horizontal filler.
\def\buildScore{\alwaysBuildScore}   %This command builds the score.
\def\hypKernAmt{0pt}    % Initial setting for kerning the hypotheses.

\def\defaultLeftLabel{}
\def\defaultRightLabel{}

\def\myTrue{Y}
\def\bottomAlignFlag{N}
\def\centerAlignFlag{N}
\def\defaultRootAtBottomFlag{Y}
\def\rootAtBottomFlag{Y}

\expandafter\ifx\csname newenvironment\endcsname\relax%
\def\makeatletter{\catcode`\@=11\relax}
\def\makeatother{\catcode`\@=12\relax}
\makeatletter
\def\newcount{\alloc@0\count\countdef\insc@unt}
\def\newdimen{\alloc@1\dimen\dimendef\insc@unt}
\def\newskip{\alloc@2\skip\skipdef\insc@unt}
\def\newbox{\alloc@4\box\chardef\insc@unt}
\makeatother
\else
\newenvironment{prooftree}%
{\begin{center}\proofSkipAmount \leavevmode}%
{\DisplayProof \proofSkipAmount \end{center} }
\fi

\def\thecur#1{\csname#1\number\theLevel\endcsname}

\newcount\theLevel    % This counter is the height of the stack.
\global\theLevel=0      % Initialized to zero
\newcount\myMaxLevel
\global\myMaxLevel=0
\newbox\myBoxA      % Temporary storage boxes
\newbox\myBoxB
\newbox\myBoxC
\newbox\myBoxD
\newbox\myBoxLL     % Boxes for the left label and the right label.
\newbox\myBoxRL
\newdimen\thisAboveSkip     %Internal use: amount to skip above line
\newdimen\thisBelowSkip     %Internal use: amount to skip below line
\newdimen\newScoreStart     % More temporary storage.
\newdimen\newScoreEnd
\newdimen\newCenter
\newdimen\displace
\newdimen\leftLowerAmt%     Amount to lower left label
\newdimen\rightLowerAmt%    Amount to lower right label
\newdimen\scoreHeight%      Score height
\newdimen\scoreDepth%       Score Depth
\newdimen\htLbox%
\newdimen\htRbox%
\newdimen\htRRbox%
\newdimen\htRRRbox%
\newdimen\htAbox%
\newdimen\htCbox%

\setbox\myBoxLL=\hbox{\defaultLeftLabel}%
\setbox\myBoxRL=\hbox{\defaultRightLabel}%

\def\allocatemore{%
    \ifnum\theLevel>\myMaxLevel%
        \expandafter\newbox\curBox%
        \expandafter\newdimen\curScoreStart%
        \expandafter\newdimen\curCenter%
        \expandafter\newdimen\curScoreEnd%
        \global\advance\myMaxLevel by1%
    \fi%
}

\def\prepAxiom{%
    \advance\theLevel by1%
    \edef\curBox{\thecur{myBox}}%
    \edef\curScoreStart{\thecur{myScoreStart}}%
    \edef\curCenter{\thecur{myCenter}}%
    \edef\curScoreEnd{\thecur{myScoreEnd}}%
    \allocatemore%
}

\def\Axiom$#1\fCenter#2${%
    \prepAxiom%
    \setbox\myBoxA=\hbox{$\mathord{#1}\fCenter\mathord{\relax}$}%
    \setbox\myBoxB=\hbox{$#2$}%
    \global\setbox\curBox=%
         \hbox{\hskip\ScoreOverhangLeft\relax%
        \unhcopy\myBoxA\unhcopy\myBoxB\hskip\ScoreOverhangRight\relax}%
    \global\curScoreStart=0pt \relax
    \global\curScoreEnd=\wd\curBox \relax
    \global\curCenter=\wd\myBoxA \relax
    \global\advance \curCenter by \ScoreOverhangLeft%
    \ignorespaces
}

\def\AxiomC#1{      % Note argument not in math mode
    \prepAxiom%
    \setbox\myBoxA=\hbox{#1}%
    \global\setbox\curBox =%
        \hbox{\hskip\ScoreOverhangLeft\relax%
                        \unhcopy\myBoxA\hskip\ScoreOverhangRight\relax}%
        \global\curScoreStart=0pt \relax
        \global\curScoreEnd=\wd\curBox \relax
        \global\curCenter=.5\wd\curBox \relax
        \global\advance \curCenter by \ScoreOverhangLeft%
    \ignorespaces
}

\def\prepUnary{%
    \ifnum \theLevel<1
        \errmessage{Hypotheses missing!}
    \fi%
    \edef\curBox{\thecur{myBox}}%
    \edef\curScoreStart{\thecur{myScoreStart}}%
    \edef\curCenter{\thecur{myCenter}}%
    \edef\curScoreEnd{\thecur{myScoreEnd}}%
}

\def\UnaryInf$#1\fCenter#2${%
    \prepUnary%
    \buildConclusion{#1}{#2}%
    \joinUnary%
    \resetInferenceDefaults%
    \ignorespaces%
}

\def\UnaryInfC#1{
    \prepUnary%
    \buildConclusionC{#1}%
    \joinUnary%
    \resetInferenceDefaults%
    \ignorespaces%
}

\def\prepBinary{%
    \ifnum\theLevel<2
        \errmessage{Hypotheses missing!}
    \fi%
    \edef\rcurBox{\thecur{myBox}}%   Set up names of right hypothesis
    \edef\rcurScoreStart{\thecur{myScoreStart}}%
    \edef\rcurCenter{\thecur{myCenter}}%
    \edef\rcurScoreEnd{\thecur{myScoreEnd}}%
    \advance\theLevel by-1%
    \edef\lcurBox{\thecur{myBox}}% Set up names of left hypothesis
    \edef\lcurScoreStart{\thecur{myScoreStart}}%
    \edef\lcurCenter{\thecur{myCenter}}%
    \edef\lcurScoreEnd{\thecur{myScoreEnd}}%
}

\def\BinaryInf$#1\fCenter#2${%
    \prepBinary%
    \buildConclusion{#1}{#2}%
    \joinBinary%
    \resetInferenceDefaults%
    \ignorespaces%
}

\def\BinaryInfC#1{%
    \prepBinary%
    \buildConclusionC{#1}%
    \joinBinary%
    \resetInferenceDefaults%
    \ignorespaces%
}

\def\prepTrinary{%
    \ifnum\theLevel<3
        \errmessage{Hypotheses missing!}
    \fi%
    \edef\rcurBox{\thecur{myBox}}%   Set up names of right hypothesis
    \edef\rcurScoreStart{\thecur{myScoreStart}}%
    \edef\rcurCenter{\thecur{myCenter}}%
    \edef\rcurScoreEnd{\thecur{myScoreEnd}}%
    \advance\theLevel by-1%
    \edef\ccurBox{\thecur{myBox}}% Set up names of center hypothesis
    \edef\ccurScoreStart{\thecur{myScoreStart}}%
    \edef\ccurCenter{\thecur{myCenter}}%
    \edef\ccurScoreEnd{\thecur{myScoreEnd}}%
    \advance\theLevel by-1%
    \edef\lcurBox{\thecur{myBox}}% Set up names of left hypothesis
    \edef\lcurScoreStart{\thecur{myScoreStart}}%
    \edef\lcurCenter{\thecur{myCenter}}%
    \edef\lcurScoreEnd{\thecur{myScoreEnd}}%
}

\def\TrinaryInf$#1\fCenter#2${%
    \prepTrinary%
    \buildConclusion{#1}{#2}%
    \joinTrinary%
    \resetInferenceDefaults%
    \ignorespaces%
}

\def\TrinaryInfC#1{%
    \prepTrinary%
    \buildConclusionC{#1}%
    \joinTrinary%
    \resetInferenceDefaults%
    \ignorespaces%
}

\def\prepQuaternary{%
    \ifnum\theLevel<4
        \errmessage{Hypotheses missing!}
    \fi%
    \edef\rrcurBox{\thecur{myBox}}%   Set up names of very right hypothesis
    \edef\rrcurScoreStart{\thecur{myScoreStart}}%
    \edef\rrcurCenter{\thecur{myCenter}}%
    \edef\rrcurScoreEnd{\thecur{myScoreEnd}}%
    \advance\theLevel by-1%
    \edef\rcurBox{\thecur{myBox}}%   Set up names of right hypothesis
    \edef\rcurScoreStart{\thecur{myScoreStart}}%
    \edef\rcurCenter{\thecur{myCenter}}%
    \edef\rcurScoreEnd{\thecur{myScoreEnd}}%
    \advance\theLevel by-1%
    \edef\ccurBox{\thecur{myBox}}% Set up names of center hypothesis
    \edef\ccurScoreStart{\thecur{myScoreStart}}%
    \edef\ccurCenter{\thecur{myCenter}}%
    \edef\ccurScoreEnd{\thecur{myScoreEnd}}%
    \advance\theLevel by-1%
    \edef\lcurBox{\thecur{myBox}}% Set up names of left hypothesis
    \edef\lcurScoreStart{\thecur{myScoreStart}}%
    \edef\lcurCenter{\thecur{myCenter}}%
    \edef\lcurScoreEnd{\thecur{myScoreEnd}}%
}

\def\QuaternaryInf$#1\fCenter#2${%
    \prepQuaternary%
    \buildConclusion{#1}{#2}%
    \joinQuaternary%
    \resetInferenceDefaults%
    \ignorespaces%
}

\def\QuaternaryInfC#1{%
    \prepQuaternary%
    \buildConclusionC{#1}%
    \joinQuaternary%
    \resetInferenceDefaults%
    \ignorespaces%
}

\def\joinQuaternary{% Construct the quarterary inference into a vbox.
    \setbox\myBoxA=\hbox{\theHypSeparation}%
    \lcurScoreEnd=\rrcurScoreEnd%
    \advance\lcurScoreEnd by\wd\rcurBox%
    \advance\lcurScoreEnd by\wd\lcurBox%
    \advance\lcurScoreEnd by\wd\ccurBox%
    \advance\lcurScoreEnd by3\wd\myBoxA%
    \displace=\lcurScoreEnd%
    \advance\displace by -\lcurScoreStart%
    \lcurCenter=.5\displace%
    \advance\lcurCenter by\lcurScoreStart%
    \ifx\rootAtBottomFlag\myTrue%
        \setbox\lcurBox=%
            \hbox{\box\lcurBox\unhcopy\myBoxA\box\ccurBox%
                      \unhcopy\myBoxA\box\rcurBox
                      \unhcopy\myBoxA\box\rrcurBox}%
    \else%
        \htLbox = \ht\lcurBox%
        \htAbox = \ht\myBoxA%
        \htCbox = \ht\ccurBox%
        \htRbox = \ht\rcurBox%
        \htRRbox = \ht\rrcurBox%
        \setbox\lcurBox=%
            \hbox{\lower\htLbox\box\lcurBox%
                  \lower\htAbox\copy\myBoxA\lower\htCbox\box\ccurBox%
                  \lower\htAbox\copy\myBoxA\lower\htRbox\box\rcurBox%
                  \lower\htAbox\copy\myBoxA\lower\htRRbox\box\rrcurBox}%
    \fi%
    \displace=\newCenter%
    \advance\displace by -.5\newScoreStart%
    \advance\displace by -.5\newScoreEnd%
    \advance\lcurCenter by \displace%
    \edef\curBox{\lcurBox}%
    \edef\curScoreStart{\lcurScoreStart}%
    \edef\curScoreEnd{\lcurScoreEnd}%
    \edef\curCenter{\lcurCenter}%
    \joinUnary%
}

\def\prepQuinary{%
    \ifnum\theLevel<5
        \errmessage{Hypotheses missing!}
    \fi%
    \edef\rrrcurBox{\thecur{myBox}}%   Set up names of very very right hypothesis
    \edef\rrrcurScoreStart{\thecur{myScoreStart}}%
    \edef\rrrcurCenter{\thecur{myCenter}}%
    \edef\rrrcurScoreEnd{\thecur{myScoreEnd}}%
    \advance\theLevel by-1%
    \edef\rrcurBox{\thecur{myBox}}%   Set up names of very right hypothesis
    \edef\rrcurScoreStart{\thecur{myScoreStart}}%
    \edef\rrcurCenter{\thecur{myCenter}}%
    \edef\rrcurScoreEnd{\thecur{myScoreEnd}}%
    \advance\theLevel by-1%
    \edef\rcurBox{\thecur{myBox}}%   Set up names of right hypothesis
    \edef\rcurScoreStart{\thecur{myScoreStart}}%
    \edef\rcurCenter{\thecur{myCenter}}%
    \edef\rcurScoreEnd{\thecur{myScoreEnd}}%
    \advance\theLevel by-1%
    \edef\ccurBox{\thecur{myBox}}% Set up names of center hypothesis
    \edef\ccurScoreStart{\thecur{myScoreStart}}%
    \edef\ccurCenter{\thecur{myCenter}}%
    \edef\ccurScoreEnd{\thecur{myScoreEnd}}%
    \advance\theLevel by-1%
    \edef\lcurBox{\thecur{myBox}}% Set up names of left hypothesis
    \edef\lcurScoreStart{\thecur{myScoreStart}}%
    \edef\lcurCenter{\thecur{myCenter}}%
    \edef\lcurScoreEnd{\thecur{myScoreEnd}}%
}

\def\QuinaryInf$#1\fCenter#2${%
    \prepQuinary%
    \buildConclusion{#1}{#2}%
    \joinQuinary%
    \resetInferenceDefaults%
    \ignorespaces%
}

\def\QuinaryInfC#1{%
    \prepQuinary%
    \buildConclusionC{#1}%
    \joinQuinary%
    \resetInferenceDefaults%
    \ignorespaces%
}

\def\joinQuinary{% Construct the quinary inference into a vbox.
    \setbox\myBoxA=\hbox{\theHypSeparation}%
    \lcurScoreEnd=\rrrcurScoreEnd%
    \advance\lcurScoreEnd by\wd\rrcurBox%
    \advance\lcurScoreEnd by\wd\rcurBox%
    \advance\lcurScoreEnd by\wd\lcurBox%
    \advance\lcurScoreEnd by\wd\ccurBox%
    \advance\lcurScoreEnd by4\wd\myBoxA%
    \displace=\lcurScoreEnd%
    \advance\displace by -\lcurScoreStart%
    \lcurCenter=.5\displace%
    \advance\lcurCenter by\lcurScoreStart%
    \ifx\rootAtBottomFlag\myTrue%
        \setbox\lcurBox=%
            \hbox{\box\lcurBox\unhcopy\myBoxA\box\ccurBox%
                      \unhcopy\myBoxA\box\rcurBox
                      \unhcopy\myBoxA\box\rrcurBox
                      \unhcopy\myBoxA\box\rrrcurBox}%
    \else%
        \htLbox = \ht\lcurBox%
        \htAbox = \ht\myBoxA%
        \htCbox = \ht\ccurBox%
        \htRbox = \ht\rcurBox%
        \htRRbox = \ht\rrcurBox%
        \htRRRbox = \ht\rrrcurBox%
        \setbox\lcurBox=%
            \hbox{\lower\htLbox\box\lcurBox%
                  \lower\htAbox\copy\myBoxA\lower\htCbox\box\ccurBox%
                  \lower\htAbox\copy\myBoxA\lower\htRbox\box\rcurBox%
                  \lower\htAbox\copy\myBoxA\lower\htRRbox\box\rrcurBox%
                  \lower\htAbox\copy\myBoxA\lower\htRRRbox\box\rrrcurBox}%
    \fi%
    \displace=\newCenter%
    \advance\displace by -.5\newScoreStart%
    \advance\displace by -.5\newScoreEnd%
    \advance\lcurCenter by \displace%
    \edef\curBox{\lcurBox}%
    \edef\curScoreStart{\lcurScoreStart}%
    \edef\curScoreEnd{\lcurScoreEnd}%
    \edef\curCenter{\lcurCenter}%
    \joinUnary%
}

\def\buildConclusion#1#2{% Build lower sequent w/ center at \fCenter position.
        \setbox\myBoxA=\hbox{$\mathord{#1}\fCenter\mathord{\relax}$}%
        \setbox\myBoxB=\hbox{$#2$}%
    \setbox\myBoxC =%
          \hbox{\hskip\ScoreOverhangLeft\relax%
        \unhcopy\myBoxA\unhcopy\myBoxB\hskip\ScoreOverhangRight\relax}%
    \newScoreStart=0pt \relax%
    \newCenter=\wd\myBoxA \relax%
    \advance \newCenter by \ScoreOverhangLeft%
    \newScoreEnd=\wd\myBoxC%
}

\def\buildConclusionC#1{% Build lower sequent w/o \fCenter present.
    \setbox\myBoxA=\hbox{#1}%
    \setbox\myBoxC =%
        \hbox{\hbox{\hskip\ScoreOverhangLeft\relax%
                        \unhcopy\myBoxA\hskip\ScoreOverhangRight\relax}}%
    \newScoreStart=0pt \relax%
    \newCenter=.5\wd\myBoxC \relax%
    \newScoreEnd=\wd\myBoxC%
        \advance \newCenter by \ScoreOverhangLeft%
}

\def\joinUnary{%Align and join \curBox and \myBoxC into a single vbox
    \global\advance\curCenter by -\hypKernAmt%
    \ifnum\curCenter<\newCenter%
        \displace=\newCenter%
        \advance \displace by -\curCenter%
        \kernUpperBox%
    \else%
        \displace=\curCenter%
        \advance \displace by -\newCenter%
        \kernLowerBox%
    \fi%
        \ifnum \newScoreStart < \curScoreStart %
        \global \curScoreStart = \newScoreStart \fi%
    \ifnum \curScoreEnd < \newScoreEnd %
        \global \curScoreEnd = \newScoreEnd \fi%
    \ifnum \curScoreStart<\wd\myBoxLL%
        \global\displace = \wd\myBoxLL%
        \global\advance\displace by -\curScoreStart%
        \kernUpperBox%
        \kernLowerBox%
    \fi%
    \buildScore%
    \buildScoreLabels%
    \ifx\rootAtBottomFlag\myTrue%
        \buildRootBottom%
    \else%
        \buildRootTop%
    \fi%
    \global \curScoreStart=\newScoreStart%
    \global \curScoreEnd=\newScoreEnd%
    \global \curCenter=\newCenter%
}

\def\buildRootBottom{%
    \global \setbox \curBox =%
        \vbox{\box\curBox%
            \vskip\thisAboveSkip \relax%
            \nointerlineskip\box\myBoxD%
            \vskip\thisBelowSkip \relax%
            \nointerlineskip\box\myBoxC}%
}

\def\buildRootTop{%
    \global \setbox \curBox =%
        \vbox{\box\myBoxC%
            \vskip\thisAboveSkip \relax%
            \nointerlineskip\box\myBoxD%
            \vskip\thisBelowSkip \relax%
            \nointerlineskip\box\curBox}%
}

\def\kernUpperBox{%
        \global\setbox\curBox =%
            \hbox{\hskip\displace\box\curBox}%
        \global\advance \curScoreStart by \displace%
        \global\advance \curScoreEnd by \displace%
        \global\advance\curCenter by \displace%
}

\def\kernLowerBox{%
        \global\setbox\myBoxC =%
            \hbox{\hskip\displace\unhbox\myBoxC}%
        \global\advance \newScoreStart by \displace%
        \global\advance \newScoreEnd by \displace%
        \global\advance\newCenter by \displace%
}

\def\joinBinary{% Construct the binary inference into a vbox.
    \setbox\myBoxA=\hbox{\theHypSeparation}%
    \lcurScoreEnd=\rcurScoreEnd%
    \advance\lcurScoreEnd by\wd\lcurBox%
    \advance\lcurScoreEnd by\wd\myBoxA%
    \displace=\lcurScoreEnd%
    \advance\displace by -\lcurScoreStart%
    \lcurCenter=.5\displace%
    \advance\lcurCenter by\lcurScoreStart%
    \ifx\rootAtBottomFlag\myTrue%
        \setbox\lcurBox=%
            \hbox{\box\lcurBox\unhcopy\myBoxA\box\rcurBox}%
    \else%
        \htLbox = \ht\lcurBox%
        \htAbox = \ht\myBoxA%
        \htRbox = \ht\rcurBox%
        \setbox\lcurBox=%
            \hbox{\lower\htLbox\box\lcurBox%
                  \lower\htAbox\box\myBoxA\lower\htRbox\box\rcurBox}%
    \fi%
    \displace=\newCenter%
    \advance\displace by -.5\newScoreStart%
    \advance\displace by -.5\newScoreEnd%
    \advance\lcurCenter by \displace%
    \edef\curBox{\lcurBox}%
    \edef\curScoreStart{\lcurScoreStart}%
    \edef\curScoreEnd{\lcurScoreEnd}%
    \edef\curCenter{\lcurCenter}%
    \joinUnary%
}

\def\joinTrinary{% Construct the trinary inference into a vbox.
    \setbox\myBoxA=\hbox{\theHypSeparation}%
    \lcurScoreEnd=\rcurScoreEnd%
    \advance\lcurScoreEnd by\wd\lcurBox%
    \advance\lcurScoreEnd by\wd\ccurBox%
    \advance\lcurScoreEnd by2\wd\myBoxA%
    \displace=\lcurScoreEnd%
    \advance\displace by -\lcurScoreStart%
    \lcurCenter=.5\displace%
    \advance\lcurCenter by\lcurScoreStart%
    \ifx\rootAtBottomFlag\myTrue%
        \setbox\lcurBox=%
            \hbox{\box\lcurBox\unhcopy\myBoxA\box\ccurBox%
                      \unhcopy\myBoxA\box\rcurBox}%
    \else%
        \htLbox = \ht\lcurBox%
        \htAbox = \ht\myBoxA%
        \htCbox = \ht\ccurBox%
        \htRbox = \ht\rcurBox%
        \setbox\lcurBox=%
            \hbox{\lower\htLbox\box\lcurBox%
                  \lower\htAbox\copy\myBoxA\lower\htCbox\box\ccurBox%
                  \lower\htAbox\copy\myBoxA\lower\htRbox\box\rcurBox}%
    \fi%
    \displace=\newCenter%
    \advance\displace by -.5\newScoreStart%
    \advance\displace by -.5\newScoreEnd%
    \advance\lcurCenter by \displace%
    \edef\curBox{\lcurBox}%
    \edef\curScoreStart{\lcurScoreStart}%
    \edef\curScoreEnd{\lcurScoreEnd}%
    \edef\curCenter{\lcurCenter}%
    \joinUnary%
}

\def\DisplayProof{%
    \ifnum \theLevel=1 \relax \else%x
        \errmessage{Proof tree badly specified.}%
    \fi%
    \edef\curBox{\thecur{myBox}}%
    \ifx\bottomAlignFlag\myTrue%
        \displace=0pt%
    \else%
        \displace=.5\ht\curBox%
        \ifx\centerAlignFlag\myTrue\relax
        \else%
                \advance\displace by -3pt%
        \fi%
    \fi%
    \leavevmode%
    \lower\displace\hbox{\copy\curBox}%
    \global\theLevel=0%
    \global\def\alwaysBuildScore{\defaultBuildScore}% Restore "always"
    \global\def\alwaysScoreFiller{\defaultScoreFiller}% Restore "always"
    \global\def\bottomAlignFlag{N}%
    \global\def\centerAlignFlag{N}%
    \resetRootPosition
    \resetInferenceDefaults%
    \ignorespaces
}

\def\buildSingleScore{% Make an hbox with a single score.
    \displace=\curScoreEnd%
    \advance \displace by -\curScoreStart%
    \global\setbox \myBoxD =%
        \hbox to \displace{\expandafter\xleaders\theScoreFiller\hfill}%
}

\def\buildDoubleScore{% Make an hbox with a double score.
    \buildSingleScore%
    \global\setbox\myBoxD=%
        \hbox{\hbox to0pt{\copy\myBoxD\hss}\raise2pt\copy\myBoxD}%
}

\def\buildNoScore{% Make an hbox with no score (raise a little anyway)
    \global\setbox\myBoxD=\hbox{\vbox{\vskip1pt}}%
}

\def\noLine{%
    \gdef\buildScore{\buildNoScore}% Set next score to this type
    \ignorespaces
}

\def\LeftLabel#1{%
    \global\setbox\myBoxLL=\hbox{{#1}\hskip\labelSpacing}%
    \ignorespaces
}
\def\RightLabel#1{%
    \global\setbox\myBoxRL=\hbox{\hskip\labelSpacing #1}%
    \ignorespaces
}

\def\buildScoreLabels{%
    \scoreHeight = \ht\myBoxD%
    \scoreDepth = \dp\myBoxD%
    \leftLowerAmt=\ht\myBoxLL%
    \advance \leftLowerAmt by -\dp\myBoxLL%
    \advance \leftLowerAmt by -\scoreHeight%
    \advance \leftLowerAmt by \scoreDepth%
    \leftLowerAmt=.5\leftLowerAmt%
    \rightLowerAmt=\ht\myBoxRL%
    \advance \rightLowerAmt by -\dp\myBoxRL%
    \advance \rightLowerAmt by -\scoreHeight%
    \advance \rightLowerAmt by \scoreDepth%
    \rightLowerAmt=.5\rightLowerAmt%
    \displace = \curScoreStart%
    \advance\displace by -\wd\myBoxLL%
    \global\setbox\myBoxD =%
        \hbox{\hskip\displace%
            \lower\leftLowerAmt\copy\myBoxLL%
            \box\myBoxD%
            \lower\rightLowerAmt\copy\myBoxRL}%
    \global\thisAboveSkip = \ht\myBoxLL%
    \global\advance \thisAboveSkip by -\leftLowerAmt%
    \global\advance \thisAboveSkip by -\scoreHeight%
    \ifnum \thisAboveSkip<0 %
        \global\thisAboveSkip=0pt%
    \fi%
    \displace = \ht\myBoxRL%
    \advance \displace by -\rightLowerAmt%
    \advance \displace by -\scoreHeight%
    \ifnum \displace<0 %
        \displace=0pt%
    \fi%
    \ifnum \displace>\thisAboveSkip %
        \global\thisAboveSkip=\displace%
    \fi%
    \global\thisBelowSkip = \dp\myBoxLL%
    \global\advance\thisBelowSkip by \leftLowerAmt%
    \global\advance\thisBelowSkip by -\scoreDepth%
    \ifnum\thisBelowSkip<0 %
        \global\thisBelowSkip = 0pt%
    \fi%
    \displace = \dp\myBoxRL%
    \advance\displace by \rightLowerAmt%
    \advance\displace by -\scoreDepth%
    \ifnum\displace<0 %
        \displace = 0pt%
    \fi%
    \ifnum\displace>\thisBelowSkip%
        \global\thisBelowSkip = \displace%
    \fi%
    \global\thisAboveSkip = -\thisAboveSkip%
    \global\thisBelowSkip = -\thisBelowSkip%
    \global\advance\thisAboveSkip by\extraVskip% Extra space above line
    \global\advance\thisBelowSkip by\extraVskip% Extra space below line
}

\def\resetInferenceDefaults{%
    \global\def\theHypSeparation{\defaultHypSeparation}%
    \global\setbox\myBoxLL=\hbox{\defaultLeftLabel}%
    \global\setbox\myBoxRL=\hbox{\defaultRightLabel}%
    \global\def\buildScore{\alwaysBuildScore}%
    \global\def\theScoreFiller{\alwaysScoreFiller}%
    \gdef\hypKernAmt{0pt}% Restore to zero kerning.
}

\def\resetRootPosition{%
    \global\edef\rootAtBottomFlag{\defaultRootAtBottomFlag}
}

\def\provx#1#2#3#4{
\setbox1=\hbox{\kern1.5pt$\scriptstyle#3$}
\def\zeichen{#2}
\ifx\zeichen\empty\setbox0=\hbox to .75em{}\else\setbox0=\hbox
{\kern1.5pt$\scriptstyle#2$}\fi
\dimen1=\dp0 \ifdim \dimen1=0pt
\advance \dimen1 by 1.5ex \else \advance \dimen1 by 1.2ex
\fi\dimen3=2ex\dimen4=.5ex\ifdim \wd0<\wd1 \dimen2=\wd1 \else \dimen2=\wd0
\fi\hbox{$#1\hskip 5pt minus5pt\vrule height\dimen3
depth\dimen4\raise\dimen1\copy0\hskip-1\wd0 \lower\ht1
\copy1\hskip-1\wd1\vrule width\dimen2 height.7ex depth-.6ex\hskip3pt
minus1.5pt#4\hskip2pt plus2pt minus2pt$}}

\def\prov#1#2#3{
\setbox1=\hbox{\kern1.5pt$\scriptstyle#2$}
\def\zeichen{#1}
\ifx\zeichen\empty\setbox0=\hbox to .75em{}\else\setbox0=\hbox
{\kern1.5pt$\scriptstyle#1$}\fi
\dimen1=\dp0
\ifdim \dimen1=0pt
\advance \dimen1 by 1.5ex \else \advance \dimen1 by 1.2ex
\fi\dimen3=2ex\dimen4=.5ex\ifdim \wd0<\wd1 \dimen2=\wd1 \else \dimen2=\wd0
\fi\hbox{\hskip0pt plus 4pt
$\vrule height\dimen3
depth\dimen4\raise\dimen1\copy0\hskip-1\wd0
\lower\ht1\copy1\hskip-1\wd1\vrule width\dimen2 height.7ex depth-.6ex
\hskip3pt minus1.5pt#3\hskip2pt plus2pt minus2pt$}}

\def\prv#1#2{
\setbox1=\hbox{\kern1.5pt$\scriptstyle#2$}
\ifx\zeichen\empty\setbox0=\hbox to .75em{}\else\setbox0=\hbox
{\kern1.5pt$\scriptstyle#1$}\fi
\dimen1=\dp0 \ifdim \dimen1=0pt
\advance \dimen1 by 1.5ex \else \advance \dimen1 by 1.2ex
\fi\dimen3=2ex\dimen4=.5ex\ifdim \wd0<\wd1 \dimen2=\wd1 \else \dimen2=\wd0
\fi\hbox{\hskip.5em$\vrule height\dimen3
depth\dimen4\raise\dimen1\copy0\hskip-1\wd0
\lower\ht1\copy1\hskip-1\wd1\vrule width\dimen2 height.7ex depth-.6ex
\hskip3pt minus1.5pt$}}

\mathchardef\str='1066
\def\negprov#1#2#3{
\setbox1=\hbox{\kern1.5pt$\scriptstyle#2$}
\setbox4=\hbox{$\str$}
\def\zeichen{#1}
\ifx\zeichen\empty\setbox0=\hbox to 1em{}\else\setbox0=\hbox
{\kern1.5pt$\scriptstyle#1$}\fi
\dimen1=\dp0
\ifdim \dimen1=0pt
\advance \dimen1 by 1.5ex \else \advance \dimen1 by 1.2ex
\fi\dimen3=2ex\dimen4=.5ex\ifdim \wd0<\wd1 \dimen2=\wd1 \else \dimen2=\wd0
\fi\hbox{\hskip.5em$\kern-1.9pt\raise1pt\copy4\kern-\wd4\kern1.9pt\vrule height\dimen3
depth\dimen4\raise\dimen1\copy0\hskip-1\wd0
\lower\ht1\copy1\hskip-1\wd1\vrule width\dimen2 height.7ex depth-.6ex
\hskip3pt minus1.5pt#3\hskip2pt plus2pt minus2pt$}}

\def\goed#1{\setbox5=\hbox{$#1$}\dimen1=.25em \dimen2=\dimen1 \advance \dimen2
by -1pt\hbox{\raise.65\ht5 \hbox{\vrule height.5\ht5 depth0pt width.4pt\vrule
height.5\ht5 width\dimen1 depth-.48\ht5}\kern-\dimen2\copy5\kern-\dimen2
\raise.65\ht5 \hbox{\vrule height .5\ht5 width\dimen1 depth-.48\ht5\vrule
height.5\ht5 depth 0pt width.4pt}\hskip4pt plus2pt minus2pt}}

\def\mod#1#2{
\def\zeichen{#1}
\hbox{\hskip 2pt plus3pt minus 2pt\vrule width.5pt height2ex depth.5ex
\vbox{\ifx\zeichen\empty\hbox to .75em{}\else
\hbox{\kern1.5pt $\scriptstyle#1$}\fi
\kern2pt
\hrule
\kern1.7pt
\hrule\kern1.7pt}
\hskip3pt minus 2pt$#2$}\hskip2pt
plus3pt minus2pt}

\def\notmod#1#2{\hbox{\hskip 2pt plus 3pt minus 3pt\vrule width.5pt
height2ex depth.5ex
\vbox{\hbox{\kern1.5pt $\scriptstyle#1$}\kern3pt
\setbox0=\hbox{\kern2pt$\scriptstyle/$}
\hrule
\kern-1.7pt
\copy0
\kern-\ht0
\kern 1.7pt
\hrule\kern1.7pt}n
\hskip3pt minus 2pt$#2$}\hskip2pt
plus3pt minus2pt}
\def\sq{\hbox{\rlap{$\sqcap$}$\sqcup$}}
\def\qed{\ifmmode\sq\else{\unskip\nobreak\hfil\penalty50\hskip1em\null
\nobreak\hfil\sq\parfillskip=0pt\finalhyphendemerits=0\endgraf}\fi\medskip}

\def\lleq{\hbox{\hskip3pt minus3pt\kern1pt\lower4pt
\vbox{\hbox{$\scriptstyle\ll$}
\kern-7pt\hbox{\kern1pt$\scriptstyle=$}}\hskip3pt minus 3pt}}

\mathchardef\res='1152
\mathchardef\qin='1062
\mathchardef\qprec='1036
\mathchardef\qless='474
\mathchardef\dpkt='72

\newcommand{\RSX}{{\mathbf{RS}}_{\Omega}(X)}
\newcommand{\SRX}{(\Sigma\mbox{-}\mathrm{Ref}_{\Omega}(X))}
\newcommand{\emt}{\bar{\emptyset}}

\newcommand{\RSOPX}{{\mathbf{RS}}_{\Omega}^{\mathcal P}(X)}
\newcommand{\RSOPXR}{{\mathbf{RS}}_{\Omega}^{\mathcal P}(\RR,X)}
\newcommand{\kld}{\,\dot{<}\,}
\newcommand{\kldg}{\,\dot{\leq}\,}
\newcommand{\su}{\ell}
\newcommand{\GAC}{\mathbf{AC}_{\!\mbox{\it\tiny global}}}

\begin{document}
\title{Classifying the provably total set functions of $\KP$ and $\KPP$}
\author{Jacob Cook and Michael Rathjen\\
{\small Department of Pure Mathematics, University of Leeds}\\
{\small Leeds LS2 9JT, United Kingdom}\\
 {\tt \scriptsize jacob$\underline{\!\phantom{a}}\,$knows@hotmail.co.uk $\;\;$
rathjen@maths.leeds.ac.uk}
}

\date{}
\maketitle
\begin{abstract} This article is concerned with classifying the provably total set-functions of Kripke-Platek set theory, $\textbf{KP}$,
and Power Kripke-Platek set theory, $\KPP$, as well as proving several (partial) conservativity results.
The main technical tool used in this paper is a relativisation technique where ordinal analysis is carried out relative to an arbitrary but fixed set $x$.

A classic result from ordinal analysis is the characterisation of the provably recursive functions of Peano Arithmetic, \textbf{PA}, by means of the fast growing hierarchy \cite{stan-PA}. Whilst it is possible to formulate the natural numbers within $\KP$, the theory speaks primarily about sets. For this reason it is desirable to obtain a characterisation of its provably total set functions. We will show that $\KP$ proves the totality of a set function precisely when it falls within a hierarchy of set functions based upon a relativised constructible hierarchy stretching up in length to any ordinal below the Bachmann-Howard ordinal.  As a consequence of this result we obtain that $\IKP+\forall x\forall y\,(x\in y\,\vee\,x\notin y)$ is conservative over $\KP$ for $\Pi_2$-formulae, where $\IKP$
stands for intuitionistic Kripke-Platek set theory.

 In a similar vein, utilising \cite{powerKP}, it is shown that $\KPP$ proves the totality of a set function precisely when it falls within a hierarchy of set functions based upon a relativised von Neumann hierarchy of the same length. The relativisation technique applied to $\KPP$ with
 the global axiom of choice, $\GAC$, also yields a parameterised extension of a result in \cite{powerKPAC}, showing that
 $\KPP+\GAC$ is conservative over $\KPP+\AC$ and $\CZF+\AC$ for $\Pi_2^{\mathcal P}$ statements. Here $\AC$ stands for the ordinary axiom of choice and $\CZF$ refers to constructive Zermelo-Fraenkel set theory.

\end{abstract}

\section{Introduction}
A major application of the techniques of ordinal analysis has been the classification of the provably total recursive functions of a theory. Usually the theories to which this methodology has been applied have been arithmetic theories, in that context it makes most sense to speak about arithmetic functions.  The concept of a recursive function on natural numbers can be extended to a more general recursion theory on arbitrary sets. For more details see \cite{mosc}, \cite{normann} and \cite{sacks}. Since $\KP$ speaks primarily about sets, it is perhaps desirable to obtain a classification of its provably total recursive set functions.\\

\noindent To provide some context we first state a classic result from proof theory, the classification of the {\em provably total recursive functions} of $\PA$. A classification can be gleaned from Gentzen's 1938 \cite{gentzen38} and 1943 \cite{gentzen43} papers.
The first explicit characterization of these functions as those definable by recursions on ordinals less than $\varepsilon_0$ was given by Kreisel \cite{kreisel1,kreisel2} in the early 1950s. Many people re-proved or provided variants of this classification result (see \cite[Chap. 4]{schwicht-wainer} for the history). As to techniques for extracting numerical bounds from infinite proofs, Schwichtenberg's \cite{schwichtenberg-PA} and the considerably more elegant approach by Buchholz and Wainer in \cite{stan-PA} and its generalization and simplification by Weiermann in \cite{weiermann-PA} are worth mentioning.  For the following definitions, suppose we have an ordinal representation system for ordinals below $\varepsilon_0$, together with an assignment of fundamental sequences to the limit ordinal terms. For an ordinal term $\al$, let $\al_n$ denote the $n$-th element of the fundamental sequence for $\al$, ie. $\al_{n+1}<\al_n$ and $\text{sup}_{n<\omega}(\al_n)=\al$. There are certain technical properties that such an assignment must satisfy, these will not be gone into here, for a detailed presentation see \cite{stan-PA}.
\begin{definition}{\em For each $\al<\varepsilon_0$ we define the function $F_\al:\omega\rightarrow\omega$ by transfinite recursion as follows
\begin{align*}
F_0(n)&:=n+1\\
F_{\al+1}(n)&:=F^{n+1}_\al(n)\;(:=\overbrace{F_\al\circ \ldots\circ F_\al}^{n+1}(n))\\
F_\al(n)&:=F_{\al_n}(n)\quad\text{if $\al$ is a limit.}
\end{align*}
}\end{definition}
\noindent This hierarchy is known as the {\em fast growing hierarchy}. Given unary functions on the natural numbers $f$ and $g$, we say that $f$ majorises $g$ if there is some $n$ such that $(\forall m>n)(g(m)<f(m))$. For a recursive function $f$ let $A_f(n,m)$ be the $\Sigma$ formula expressing that on input $n$ the Turing machine for computing $f$ outputs $m$, to avoid frustrating counter examples let us suppose $A_f$ does this in some `natural' way.
\begin{theorem}\label{PA-fun}{\em Suppose $f:\omega\rightarrow\omega$ is a recursive function. Then
\begin{description}
\item[i)] If $\PA\vdash\forall x\exists! y A_f(x,y)$ then $f$ is majorised by $F_\al$ for some $\al<\varepsilon_0$.
\item[ii)] $\PA\vdash\forall x\exists! yA_{F_\al}(x,y)$ for every $\al<\varepsilon_0$.
\end{description}
}\end{theorem}
\begin{proof}
This classic result is proved in full in \cite{stan-PA}.\end{proof}
\noindent This chapter will be focused on obtaining a similar result for the provably total set functions of $\KP$.\footnote{There are many papers concerned with the provably recursive {\bf number-theoretic} functions of $\KP$ and much stronger theories. The basic idea consists in adding another layer of control to the ordinal analysis that allows one to extract bounds for numerical witnesses.
These techniques were initially engineered
by Buchholz, Wainer \cite{stan-PA} and Weiermann \cite{weiermann-PA} and then got extended by Blankertz, Weiermann \cite{blankertz-w,blankertz,blankertz-w2}, Michelbrink \cite{michelbrink}, Pohlers and Stegert \cite{po-ste} to ever stronger theories.
 Another route for obtaining classifications of provably numerical functions proceeds as follows. The ordinal analysis of a set theory $T$ shows that the arithmetic part of $T$ can be reduced to $\PA$ plus transfinite induction for every ordinal below the proof-theoretic ordinal of $T$. Thus it suffices to characterize the provably numerical functions of the latter system. This leads to the descent recursive functions in the sense of \cite{fs}. That this  method is perfectly general was first sketched in \cite{fs} and then proved rigorously in \cite{buch-not}. The latter approach has the advantage that the ordinal analysis of $T$ needn't be burdened with the extra task of controlling numerical witnesses. }
 A similar role to the fast growing hierarchy in Theorem \ref{PA-fun} will be played by the {\em relativised constructible hierarchy}.
\begin{definition}{\em
Let $X$ be any set. We may relativise the constructible hierarchy to $X$ as follows:
\begin{align*}
L_0(X)&:=TC(\{X\})\quad\text{the \textit{transitive closure} of $\{X\}$}\\
L_{\alpha+1}(X)&:=\{B\subseteq L_\alpha (X)\::\:B\text{ is definable over }\left\langle L_\alpha(X), \in\right\rangle\}\\
L_\theta (X) &:=\displaystyle\bigcup_{\xi<\theta} L_\xi (X)\quad\text{when $\theta$ is a limit.}
\end{align*}
}\end{definition}
\noindent In section 2 we build an ordinal notation system relativised to an arbitrary set $X$, this will be used for the rest of the article. In section 3 we define the infinitary system $\RSX$, based on the relativised constructible hierarchy and show that we can eliminate cuts for derivations of $\Sigma$ formulae. In section 4 we embed $\KP$ into $\RSX$, allowing us to obtain cut free infinitary derivations of $\KP$ provable $\Sigma$ formulae. Technically we use Buchholz' operator controlled derivations (see \cite{bu93}) which are also used in \cite{po}. In section 5 we give a well ordering proof in $\KP$ for the ordinal notation system given in section 2. Finally we combine the results of this chapter to give a classification of the provably total set functions of $\KP$ in section 6. This result, whilst perhaps known to those who have thought hard about these things, has not appeared in the literature to date.
Section 7 contains applications to semi-intuitionistic Kripke-Platek set theory.
Section 8 carries out a relativised ordinal analysis of Power Kripke-Platek set theory, $\KPP$, from which ensues a classification of its provable set functions. This closely follows the treatment in \cite{powerKP}. In section 9, a further ingredient is added to the infinitary system by incorporating a global choice relation. Due to the relativisation one gets partial conservativity results for $\KPP+\GAC$ over $\KPP+\AC$ and $\CZF+\AC$
that provide  improvements on \cite[Theorem 3.3]{powerKPAC} and \cite[Corollary 5.2]{powerKPAC}. These theories can also be added  to the list of theories \cite[Theorem 15.1]{ml} with the same proof-theoretic strength.

\section{A relativised ordinal notation system}
The aim of this section is to relativise the construction of the Bachmann-Howard ordinal to contain an arbitrary set $X$ or rather its rank $\theta$. We will construct an ordinal representation system that will be set primitive recursive given access to an oracle for $X$. Here the notion of recursive and primitive recursive is extended to arbitrary sets, see \cite{normann} or \cite{sacks} for more detail. The construction of an ordinal representation system for the Bachmann-Howard ordinal is now fairly standard in proof theory, carried out for example in \cite{bu86}. Intuitively our system will appear similar, only the ordering $W$ will be inserted as an initial segment before new ordinals start being `named'  via the collapsing function.\\

\noindent Before defining the formal terms and the procedure for computing their ordering, it is informative to give definitions for the corresponding ordinals and ordinal functions themselves. To this end we will begin working in \textbf{ZFC}, later it will become clear that the necessary ordinals can be expressed as formal terms and comparisons between these terms can be made primitive recursively relative to $W\!$.\\

\noindent In what follows $\textbf{ON}$ will denote the class of all ordinals. First we require some information about the $\varphi$ function on ordinals. These definitions and results are well known, see \cite{sch77}.

\begin{definition}{\em For each $\al\in\textbf{ON}$ we define a class of ordinals $Cr(\al)\subseteq\textbf{ON}$ and a class function
\begin{equation*}
\varphi_\al:\textbf{ON}\rightarrow\textbf{ON}
\end{equation*} by transfinite recursion.
\begin{description}
\item[i)] $Cr(0):=\{\omega^\beta\;|\;\beta\in\textbf{ON}\}$ and $\varphi_0(\beta):=\omega^\beta$.
\item[ii)] For $\al>0$ $Cr(\al):=\{\beta\;|\;(\forall\gamma<\al)(\varphi_\gamma(\beta)=\beta)\}$.
\item[iii)] For each $\al\in\textbf{ON}$ $\varphi_\al(\cdot)$ is the function enumerating $Cr(\al)$.
\end{description}
}\end{definition}
\noindent The convention is to write $\varphi\al\beta$ instead of $\varphi_\al(\beta)$. An ordinal $\beta\in Cr(0)$ is often referred to as {\em additive principal}, since for all $\beta_1,\beta_2<\beta$ we have $\beta_1+\beta_2<\beta$.
\begin{theorem}\label{phifunction}{\em $\phantom{a}$
\begin{description}
\item[i)] $\varphi\al_1\beta_1=\varphi\al_2\beta_2$ if and only if $\left\{\begin{array}{ll}& \al_1<\al_2\quad\text{and}\quad\beta_1=\varphi\al_2\beta_2\\
\text{or }&\al_1=\al_2\quad\text{and}\quad\beta_1=\beta_2\\
\text{or }&\al_2<\al_1\quad\text{and}\quad\varphi\al_1\beta_1=\beta_2.
\end{array}\right.$
\item[ii)] $\varphi\al_1\beta_1<\varphi\al_2\beta_2$ if and only if $\left\{\begin{array}{ll}& \al_1<\al_2\quad\text{and}\quad\beta_1<\varphi\al_2\beta_2\\
\text{or }&\al_1=\al_2\quad\text{and}\quad\beta_1<\beta_2\\
\text{or }&\al_2<\al_1\quad\text{and}\quad\varphi\al_1\beta_1<\beta_2.
\end{array}\right.$
\item[iii)] For any additive principal $\beta$ there are unique ordinals $\beta_1\leq\beta$ and $\beta_2<\beta$ such that $\beta=\varphi\beta_1\beta_2$.
\end{description}
}\end{theorem}
\begin{proof}
This result is proved in full in \cite{sch77}.
\end{proof}
\begin{definition}{\em
We define $\Gamma_{(\cdot)}:\textbf{ON}\rightarrow\textbf{ON}$ to be the class function enumerating the ordinals $\beta$ such that for all $\beta_1,\beta_2<\beta$ we have $\varphi\beta_1\beta_2<\beta$. Ordinals of the form $\Gamma_\beta$ will be referred to as {\em strongly critical}.
}\end{definition}
\noindent Now let $\theta\in\textbf{ON}$ be the unique ordinal that is the set-theoretic   rank of $X$.

\begin{definition}\label{theta}{\em
Let $\Omega$ be the least uncountable cardinal greater than $\theta\!$. The sets $B_\theta(\alpha)\subseteq\textbf{ON}$ and ordinals $\psi_\theta (\alpha)$ are defined by transfinite recursion on $\alpha$ as follows:
\begin{align*}
B_\theta (\alpha)&:=\text{Closure of }\{0,\Omega\}\cup\{\Gamma_\beta\::\:\beta\leq\theta\}\text{ under $+$, $\varphi$ and $\psi_\theta |_\alpha$}\\
\psi_\theta (\alpha)&:=\text{min}\{\beta\::\:\beta\notin B_\theta(\alpha)\}
\end{align*}
}\end{definition}
\noindent For the remainder of this section, since $\theta$ remains fixed, the subscripts will be dropped from %$\Omega_\theta$,
 $B_\theta$ and $\psi_\theta$ to improve readability. At first glance it may appear strange having the elements from $\theta$ inserted into the $\Gamma$-numbers. Ultimately we aim to have $+$ and $\varphi$ as primitive symbols in our notation system, simply having $\theta$ as an initial segment here would cause problems with unique representation. Some ordinals could get a name directly from $\theta$ and other names by applying $+$ and $\varphi$ to smaller elements.

\begin{lemma}\label{card}{\em For each $\alpha\in\textbf{ON}$:
\begin{description}
\item[i)] The cardinality of $B(\alpha)$ is max$\{\aleph_0,|\theta|\}$, where $|\theta|$ denotes the cardinality of $\theta$.
\item[ii)] $\psi\alpha<\Omega$.
\end{description}
}\end{lemma}
\begin{proof}
i) Let
\begin{align*}
B^0(\alpha):=&\{0,\Omega\}\cup\{\Gamma_\beta\::\:\beta\leq\theta\}\\
B^{n+1} (\alpha):=&B^n (\alpha)\cup\{\xi+\eta\::\:\xi,\eta\in B^n(\alpha)\}\\
&\cup\{\varphi\xi\eta\::\:\xi,\eta\in B^n (\alpha)\}\\
&\cup\{\psi\xi\::\:\xi\in B^n(\alpha)\cap\alpha\}.
\end{align*}
Observe that $B(\alpha)=\displaystyle\cup_{n<\omega} B^n (\alpha)$, this can be proved by a straightforward induction on $n$.\\

\noindent If $\theta$ is finite then, again by induction on $n$, we can show that each $B^n(\alpha)$ is also finite. Since $B(\alpha)$ is a countable union of finite sets and $\omega\subseteq B(\alpha)$ it follows that it must have cardinality $\aleph_0$.\\

\noindent Now suppose $\theta$ is infinite, so $B(\alpha)$ is the countable union of sets of cardinality $|\theta|$ and thus also has cardinality $|\theta|$.\\

\noindent ii) If $\psi\alpha\geq\Omega$ then $\Omega\subset B(\alpha)$ contradicting i).
\end{proof}
\begin{lemma}\label{collapseprop}$\phantom{a}${\em
\begin{description}
\item[i)] If $\gamma\leq\delta$ then $B(\gamma)\subseteq B(\delta)$ and $\psi\gamma\leq\psi\delta$.
\item[ii)] If $\gamma\in B(\delta)\cap\delta$ then $\psi\gamma<\psi\delta$.
\item[iii)] If $\gamma\leq\delta$ and $[\gamma,\delta)\cap B(\gamma)=\emptyset$ then $B(\gamma)=B(\delta)$.
\item[iv)] If $\xi$ is a limit then $B(\xi)=\cup_{\eta<\xi} B(\eta)$.
\item[v)] $\psi\gamma$ is a strongly critical and $\psi\gamma\geq\Gamma_{\theta+1}$.
\item[vi)] $B(\gamma)\cap\Omega=\psi\gamma$.
\item[vii)] If $\xi$ is a limit then $\psi\xi=\text{sup}_{\eta<\xi}\psi\eta$.
\item[viii)] $\psi(\gamma+1)\leq(\psi\gamma)^\Gamma$, where $\delta^\Gamma$ denotes the smallest strongly critical ordinal above $\delta$.
\item[ix)] If $\alpha\in B(\alpha)$ then $\psi(\alpha+1)=(\psi\alpha)^\Gamma$.
\item[x)] If $\alpha\notin B(\alpha)$ then $\psi(\alpha+1)=\psi\alpha$ and $B(\alpha+1)=B(\alpha)$.
\item[xi)] If $\gamma\in B(\gamma)$ and $\delta\in B(\delta)$ then $[\gamma<\delta$ if and only if $\psi\gamma<\psi\delta]$.
\end{description}
}\end{lemma}
\begin{proof}
i) Suppose $\gamma\leq\delta$, now note that $B(\delta)$ is closed under $\psi|_\delta$ which includes $\psi|_\gamma$ so $B(\gamma)\subseteq B(\delta)$. From this it immediately follows from the definition that $\psi\gamma\leq\psi\delta$.\\

\noindent ii) From $\gamma\in B(\delta)\cap\delta$ we get $\psi\gamma\in B(\delta)$, thus $\psi\gamma<\psi\delta$ b the definition of $\psi\delta$.\\

\noindent iii) It is enough to show that $B(\gamma)$ is closed under $\psi|_\delta$. Let $\beta\in B(\gamma)$ and $\beta<\delta$, then by assumption $\beta<\gamma$, thus $\psi\beta\in B(\gamma)$.\\

\noindent iv) By i) we have $\cup_{\eta<\xi} B(\eta)\subseteq B(\xi)$. It remains to verify that $Y:=\cup_{\eta<\xi}B(\eta)$ is closed under $\psi|_{\xi}$. So let $\delta\in Y\cap\xi$, since $\xi$ is a limit there is some $\xi_0<\xi$ such that $\delta\in Y\cap\xi_0$ and there is some $\xi_1<\xi$ such that $\delta\in B(\xi_1)$. Therefore $\delta\in B(\xi^*)\cap\xi^\star$ where $\xi^*=\text{max}\{\xi_0,\xi_1\}$, thus $\psi\delta\in B(\xi^*)\subseteq Y$.\\

\noindent v) We may write the ordinal $\psi\alpha$ in Cantor normal form, so that $\psi\alpha=\omega^{\alpha_1}+\ldots+\omega^{\alpha_n}$ with $\alpha_1\geq\ldots\geq\alpha_n$. If $n>1$ then $\alpha_1,\ldots,\alpha_n<\psi\alpha$ which implies by the definition of $\psi\alpha$ that $\alpha_1,\ldots,\alpha_n\in B(\alpha)$. But by closure of $B(\alpha)$ under $+$ and $\varphi$ we get $\varphi 0\alpha_1+\ldots+\varphi 0\alpha_n=\omega^{\alpha_1}+\ldots+\omega^{\alpha_n}\in B(\alpha)$ contradicting $\psi\alpha\notin B(\alpha)$. Thus $\psi\alpha$ is additive principal and it follows from Theroem \ref{phifunction}iii) that we may find ordinals $\gamma\leq\psi\al$ and $\delta<\psi\alpha$ such that $\psi\alpha=\varphi\gamma\delta$. If $\delta>0$ then $\gamma<\psi\alpha$ since $\gamma\leq\varphi\gamma 0<\varphi\gamma\delta$, but if $\delta,\gamma<\psi\alpha$ then we have $\delta,\gamma\in B(\alpha)$ and hence $\varphi\gamma\delta\in B(\alpha)$ contradicting $\psi\alpha\notin B(\alpha)$. Thus $\psi\alpha=\varphi\gamma 0$, but if $\gamma<\psi\alpha$ then again we get $\varphi\gamma 0\in B(\alpha)$; a contradiction. So it must be the case that $\psi\alpha=\gamma$, ie. $\psi\al$ is strongly critical.\\

\noindent For the second part note that $\psi\alpha\neq\Gamma_\beta$ for any $\beta\leq\theta$ since by definition each such $\Gamma_\beta\in B(\alpha)$.\\

\noindent vi) By \ref{card}ii) and the definition of $\psi$ it is clear that $\psi\alpha\subseteq B(\alpha)\cap\Omega$. Now let
\begin{equation*}
Y:=\psi\alpha\cup\{\delta\geq\Omega\:|\:\delta\in B(\alpha)\}
\end{equation*}
by v) $Y$ contains $0,\OO$ and $\Gamma_\beta$ for $\beta\leq\theta$, moreover it is closed under $+$ and $\varphi$. It remains to show that $Y$ is closed under $\psi|_\alpha$, this follows immediately from ii).\\

\noindent vii) Let $\xi$ be a limit ordinal. Using parts vi), iv) and i) we have
\begin{equation*}
\psi\xi=B(\xi)\cap\Omega=(\cup_{\eta<\xi} B(\eta))\cap\Omega=\cup_{\eta<\xi}(B(\eta)\cap\Omega)=\cup_{\eta<\xi}\psi\eta=\text{sup}_{\eta<\xi}\psi\eta.
\end{equation*}
viii) Let
\begin{equation*}
Y:=(\psi\alpha)^\Gamma\cup\{\delta\geq\Omega\:|\:\delta\in B(\alpha)\}.
\end{equation*}
$Y$ is closed under $+$ and $\varphi$, also it contans $\Gamma_\beta$ for any $\beta\leq\theta$ by v). Moreover it contains $\psi\gamma$ for any $\gamma\leq\alpha$ by i), so it is closed under $\psi|_{(\alpha+1)}$. Therefore $Y$ must contain $B(\alpha+1)$, and so $\psi(\alpha+1)\leq(\psi\alpha)^\Gamma$.\\

\noindent ix) From $\alpha\in B(\alpha)$ we get $\alpha\in B(\alpha+1)$, it then follows from ii) that $\psi\alpha<\psi(\alpha+1)$. Thus $\psi(\alpha+1)\leq(\psi\alpha)^\Gamma$ by viii) and $\psi(\alpha+1)\geq(\psi\alpha)^\Gamma$ from v), so it must be the case that $\psi(\alpha+1)=(\psi\alpha)^\Gamma$.\\

\noindent x) Suppose $\alpha\notin B(\alpha)$, then $[\alpha,\alpha+1)\cap B(\alpha)=\emptyset$ so we may apply iii) to give $B(\alpha+1)=B(\alpha)$ from which $\psi(\alpha+1)=\psi\alpha$ follows immediately.\\

\noindent xi) Suppose $\gamma\in B(\gamma)$ and $\delta\in B(\delta)$. If $\gamma<\delta$ then from ix) we get $\psi(\gamma+1)=(\psi\gamma)^\Gamma>\psi\gamma$, but by i) $\psi(\gamma+1)\leq\psi\delta$.\\

\noindent Now if $\psi\gamma<\psi\delta$ then from the contraposition of i) we get $\gamma<\delta$.
\end{proof}
\begin{definition}{\em
We write
\begin{description}
\item[i)] $\alpha=_{NF}\alpha_1+\ldots+\alpha_n$ if $\alpha=\alpha_1+\ldots+\alpha_n$, $n>1$, $\alpha_1,\ldots,\alpha_n$ are additive principal numbers and $\alpha_1\geq\ldots\geq\alpha_n$.
\item[ii)] $\alpha=_{NF}\varphi\gamma\delta$ if $\alpha=\varphi\gamma\delta$ and $\gamma,\delta<\varphi\gamma\delta$.
\item[iii)] $\alpha=_{NF}\psi\gamma$ if $\alpha=\psi\gamma$ and $\gamma\in B(\gamma)$
\end{description}
}\end{definition}
\begin{lemma}\label{!NF}{\em $\phantom{a}$
\begin{description}
\item[i)] If $\alpha=_{NF}\alpha_1+\ldots+\alpha_n$ then for any $\eta\in\textbf{ON}$
\begin{equation*}
\alpha\in B(\eta)\quad\text{if and only if}\quad\alpha_1,\ldots,\alpha_n\in B(\eta).
\end{equation*}
\item[ii)] If $\alpha=_{NF}\varphi\gamma\delta$ then for any $\eta\in\textbf{ON}$
\begin{equation*}
\alpha\in B(\eta)\quad\text{if and only if}\quad\gamma,\delta\in B(\eta).
\end{equation*}
\item[iii)] If $\alpha=_{NF}\psi\gamma$ then for any $\eta\in\textbf{ON}$
\begin{equation*}
\alpha\in B(\eta)\quad\text{if and only if}\quad\gamma\in B(\eta)\cap\eta.
\end{equation*}
\end{description}
}\end{lemma}
\begin{proof}
i) Suppose $\alpha=_{NF}\alpha_1+\ldots+\alpha_n$, the $\Leftarrow$ direction is clear from the closure of $B(\eta)$ under $+$. For the other direction let
\begin{equation*}
AP(\alpha):=\begin{cases}
\emptyset\quad\text{if $\alpha=0$}\\
\{\alpha\}\quad\text{if $\alpha$ is additive principal}\\
\{\alpha_1,\ldots,\alpha_n\}\quad\text{if $\alpha=_{NF}\alpha_1+\ldots+\alpha_n$}
\end{cases}
\end{equation*}
$AP(\alpha)$ stands for the \textit{additive predecessors} of $\alpha$. Now let
\begin{equation*}
Y:=\{\gamma\in B(\eta)\:|\:AP(\gamma)\subseteq B(\eta)\}.
\end{equation*}
Observe that $0,\Omega\in Y$ and $\{\Gamma_\beta\:|\:\beta\leq\theta\}\subseteq Y$. Now choose any $\gamma,\delta\in Y$, we have $AP(\gamma+\delta)\subseteq AP(\gamma)\cup AP(\delta)\subseteq B(\eta)$, thus $Y$ is closed under $+$. Now $AP(\varphi\gamma\delta)=\{\varphi\gamma\delta\}$ since the range of $\varphi$ is the additive principal numbers thus $Y$ is closed under $\varphi$. Finally $AP(\psi\gamma)=\{\psi\gamma\}$ for any $\gamma\in Y\cap\eta$ so $Y$ is closed under $\psi|_{\eta}$. It follows that $B(\eta)\subseteq Y$ and thus the other direction is proved.\\

\noindent ii) Again the $\Leftarrow$ direction follows immediately from the closure of $B(\eta)$ under $\varphi$. For the other direction we let
\begin{equation*}
PP(\alpha):=\begin{cases}
\emptyset\quad\text{if $\al=0$}\\
\{\al\}\quad\text{if $\al$ is strongly critical}\\
\{\gamma,\delta\}\quad\text{if $\alpha=_{NF}\varphi\gamma\delta$}\\
\{\al_1,\ldots,\al_n\}\quad\text{if $\al=_{NF}\al_1+\ldots+\al_n$.}
\end{cases}
\end{equation*}
for want of a better phrase $PP(\alpha)$ stands for the \textit{predicative predecessors} of $\alpha$. Now set
\begin{equation*}
Y:=\{\gamma\in B(\eta)\:|\:PP(\gamma)\subseteq B(\eta)\}
\end{equation*}
It is easily seen that $Y$ contains $0$,$\Omega$ and $\Gamma_\beta$ for any $\beta\leq\theta$. $PP(\gamma+\delta)\subseteq PP(\gamma)\cup PP(\delta)$ so $Y$ is closed under $+$. $PP(\varphi\gamma\delta)\subseteq\{\gamma,\delta\}$ so $Y$ is closed under $\varphi$. Finally $PP(\psi\gamma)=\{\psi\gamma\}$ for any $\gamma<\eta$ by \ref{collapseprop}v). It follows that $Y$ must contain $B(\eta)$, which proves the $\Rightarrow$ direction.\\

\noindent iii) Suppose $\alpha=_{NF}\psi\gamma$, the $\Leftarrow$ direction is clear by the closure of $B(\eta)$ under $\psi|_\eta$. For the other direction suppose $\alpha\in B(\eta)$, from this we get $\psi\gamma<\psi\eta$ which gives us $\gamma<\eta$. Now by assumption $\gamma\in B(\gamma)$, and $B(\gamma)\subseteq B(\eta)$ so $\gamma\in B(\eta)\cap\eta$.
\end{proof}
\noindent In order to create an ordinal notation system from the ordinal functions described above, we single out a set $R(\theta)$ of ordinals which have a unique canonical description.
\begin{definition}{\em We give an inductive definition of the set $R(\theta)$, and the complexity $G\alpha <\omega$ for every $\alpha\in R(\theta)$
\begin{description}
\item[(R1)] $0,\Omega\in R(\theta)$ and $G0:=G\Omega:=0$.
\item[(R2)] For each $\beta\leq\theta$, $\Gamma_\beta\in R(\theta)$ and $G\Gamma_\beta:=0$.
\item[(R3)] If $\alpha=_{NF}\alpha_1+\ldots+\alpha_n$ and $\alpha_1,\ldots,\alpha_n\in R(\theta)$ then $\alpha\in R(\theta)$ and $G\alpha:=\text{max}\{G\alpha_1,\ldots,G\alpha_n\}+1$.
\item[(R4)] If $\gamma,\delta<\Omega$,  $\alpha=_{NF}\varphi\gamma\delta$ and $\gamma,\delta\in R(\theta)$ then $\alpha\in R(\theta)$ and $G\alpha:=\text{max}\{G\gamma, G\delta\}+1$.
\item[(R5)] If $\gamma\geq\Omega$,  $\alpha=_{NF}\varphi 0\gamma$ and $\gamma\in R(\theta)$ then $\alpha\in R(\theta)$ and $G\alpha:=G\gamma+1$
\item[(R6)] If $\alpha=_{NF}\psi\gamma$ and $\gamma\in R(\theta)$ then $\alpha\in R(\theta)$ and $G\alpha:=G\gamma+1$
\end{description}
}\end{definition}
\begin{lemma}\label{!R}{\em
Every element $\alpha\in R(\theta)$ is included due to precisely one of the rules (R1)-(R6) and thus the complexity $G\alpha$ is uniquely defined.
}\end{lemma}
\begin{proof}
This follows immediately from \ref{!NF}.
\end{proof}

\noindent Our goal is to turn $R(\theta)$ into a formal representation system, the main obstacle to this is that it is not immediately clear how to deal with the constraint $\gamma\in B(\gamma)$ in a computable way. This problem leads to the following definition.

\begin{definition}\label{K}{\em To each $\alpha\in R(\theta)$ we assign a set $K\alpha$ of ordinal terms by induction on the complexity $G\alpha$:
\begin{description}
\item[(K1)] $K0:=K\Omega:=K\Gamma_\beta:=\emptyset$ for all $\beta\leq\theta$.
\item[(K2)] If $\alpha=_{NF}\alpha_1+\ldots+\alpha_n$ then $K\alpha:=K\alpha_1\cup\ldots\cup K\alpha_n$.
\item[(K3)] If $\alpha=_{NF}\varphi \gamma\delta$ then $K\alpha:=K\gamma\cup K\delta$.
\item[(K4)] If $\alpha=_{NF}\psi\gamma$ then $K\alpha:=\{\gamma\}\cup K\gamma$.
\end{description}
}\end{definition}

\noindent $K\alpha$ consists of the ordinals that occur as arguments of the $\psi$ function in the normal form representation of $\alpha$. Note that each ordinal in $K\alpha$ belongs to $R(\theta)$ itself and has complexity lower than $G\alpha$.

\begin{lemma}\label{KK}{\em For any $\alpha,\eta\in R(\theta)$
\begin{equation*}
\alpha\in B(\eta)\quad\text{if and only if}\quad (\forall \xi\in K\alpha)(\xi<\eta)
\end{equation*}
}\end{lemma}
\begin{proof}
The proof is by induction on $G\alpha$. If $G\alpha=0$ then $\alpha\in B(\eta)$ for any $\eta$, and $K\alpha=\emptyset$ by (K1) so the result holds.\\

\noindent Case 1. If $\alpha=_{NF}\alpha_1+\ldots+\alpha_n$ then $\alpha\in B(\eta)$ iff $\alpha_1,\ldots,\alpha_n\in B(\eta)$ by \ref{!NF}i). Now inductively $\alpha_1,\ldots,\alpha_n\in B(\eta)$ iff $(\forall\xi\in K\alpha_1\cup \ldots\cup K\alpha_n)(\xi<\eta)$, but by (K2) $K\alpha=K\alpha_1\cup \ldots\cup K\alpha_n$.\\

\noindent Case 2. If $\alpha=_{NF}\varphi\gamma\delta$ we may argue in a similar fashion to Case 1, using \ref{!NF}ii) and (K3) instead.\\

\noindent Case 3. If $\alpha=_{NF}\psi\gamma$ then $\alpha\in B(\eta)$ iff $\gamma\in B(\eta)\cap\eta$ by \ref{!NF}iii). Now by induction hypothesis $\gamma\in B(\eta)\cap \eta$ iff $(\forall\xi\in K\gamma)(\xi<\eta)\:\text{and}\:\gamma<\eta$, and by (K4) this occurs precisely when $(\forall\xi\in K\alpha)(\xi<\eta)$.
\end{proof}
\noindent Recall that $\theta$ is the rank of $X$.   Let
\begin{align*}
\mathcal{L}_{\theta}:&=\{0,\Omega,+,\varphi,\psi\}\cup\{\Gamma_{\xi}\::\:\xi\leq\theta\}\quad\text{and}\\
\mathcal{L}_{\theta}^*:&=\{s\;|\;s\text{ is a finite string of symbols from $\mathcal{L}_{\theta}$}\}.
\end{align*}
Now let $T(\theta)\subseteq \mathcal{L}_{\theta}^*$ be the set of strings that correspond to ordinals in $R(\theta)$ expressed in normal form. Owing to Lemma \ref{!R} there is a one to one correspondence between $T(\theta)$ and $R(\theta)$. The ordering on $T(\theta)$ induced from the ordering of the ordinals in $R(\theta)$ will be denoted $\prec$. To differentiate between elements of the two sets, Greek letters $\al,\beta,\gamma,\eta,\xi,\ldots$ range over ordinals and Roman letters $a,b,c,d,e,\ldots$ range over finite strings from $\mathcal{L}_{\theta}^*$.
\begin{theorem}\label{ttheta}{\em
 The set $T(\theta)$ and the relation $\prec$ on $T(\theta)$ are set primitive recursive in $\theta$.
}\end{theorem}
\begin{proof} Below a $\theta$-primitive recursive procedure means a procedure that is primitive recursive in the two parameters $\theta$ and the ordering $<_{\theta}$ on the ordinals $\xi\leq \theta$.
We need to provide the following two procedures
\begin{description}
\item[A)] A $\theta$-primitive recursive procedure which decides for $a\in \mathcal{L}_{\theta}^*$ whether $a\in T(\theta)$.
\item[B)] A $\theta$-primitive recursive procedure which decides for non-identical $a,b\in T(\theta)$ whether $a\prec b$ or $b\prec a$.
\end{description}
\noindent We define \textbf{A)} and \textbf{B)} simultaneously by induction on the term complexity $Ga$.\\

\noindent For the base stage of \textbf{A)} we have $0,\Omega\in T(\theta)$ and $\Gamma_{\xi}\in T(\theta)$ for all $\xi\leq  \theta$.\\

\noindent For the base stage of \textbf{B)} we have $0\prec\Gamma_{\xi}\prec\Omega$ for all $\xi\leq \theta$ and the terms $\Gamma_{\xi}$ inherit the ordering from $\theta$, for which we have access to an oracle.\\

\noindent For the inductive stage of \textbf{A)} we require the following 3 things:\\

\noindent \textbf{A1)} A $\theta$-primitive recursive procedure that on input $a_1,\ldots,a_n\in T(\theta)$ decides whether $a_1+\ldots+a_n\in T(\theta)$.\\

\noindent \textbf{A2)} A $\theta$-primitive recursive procedure that on input $a_1,a_2\in T(\theta)$ decides whether $\varphi a_1 a_2\in T(\theta)$.\\

\noindent \textbf{A3)} A $\theta$-primitive recursive procedure that on input $a\in T(\theta)$ decides whether $\psi a\in T(\theta)$.\\

\noindent For \textbf{A1)} we need to decide if $n>1$ and if $a_1\succeq \ldots\succeq a_n,$ which we can do by the induction hypothesis. We also need to decide if $a_1,\ldots,a_n$ are additive principal; all terms other than those of the form $b_1+\ldots+b_m$ ($m>1$) and $0$ are additive principal.\\

\noindent For \textbf{A2)}, first let $ORD_{\theta}$ denote the set of $\mathcal{L}_{\theta}$ strings which represent an ordinal (not necessarily in normal form), ie. each function symbol has the correct arity. Next we define the set of strings which correspond to the strongly critical ordinals, where $\equiv$ signifies identity of strings.
\begin{equation*}
SC_{\theta}:=\{\Omega\}\cup\{\Gamma_{\xi}\::\;\xi\leq \theta\}\cup\{a\in ORD_\theta\::\:a\equiv \psi b\}
\end{equation*}
We may decide membership of $SC_\theta$ in a $\theta$-primitive recursive fashion. For the decision procedure we split into cases based upon the form of $a_2$:
\begin{description}
\item[i)] If $a_2\equiv 0$ then $\varphi a_1a_2\in T(\theta)$ whenever $a_1\notin SC_{\theta}$
\item[ii)] If $a_2\in SC_{\theta}$ then $\varphi a_1a_2\in T(\theta)$ whenever $a_1\succeq a_2$ and $a_2\neq\Omega$.
\item[iii)] If $a_2\succ\Omega$ then $\varphi a_1a_2\in T(\theta)$ exactly when $a_1=0$.
\item[iv)] If $a_2\equiv b_1+\ldots+b_n\prec\Omega$, with $n>1$ then $\varphi a_1a_2\in T(\theta)$ regardless of the form of $a_1$.
\item[iv)] If $a_2\equiv\varphi b_1 b_2\prec\Omega$ then $\varphi a_1 a_2\in T(\theta)$ whenever $a_1\succeq b_1$.
\end{description}
For a rigorous treatment of the $\varphi$ function see \cite{sch77}.\\

\noindent The function $K$ from Definition \ref{K} lifts to a $\theta$-primitive recursive function on $T(\theta)$. Moreover every $b\in Ka$ is a member of $T(\theta)$ of lower complexity than $a$. Owing to Lemma \ref{KK}, for the decision procedure \textbf{A3)} we may first compute $Ka$, then check whether $(\forall b\in Ka)(b\prec a)$, which we may do by the induction hypothesis.\\

\noindent Finally for the inductive stage of \textbf{B)}, given two elements of $T(\theta)$ we may decide their ordering using the following procedure.
\begin{description}
\item[B1)] $0\prec a$ for every $a\neq 0$.
\item[B2)] $\Gamma_\xi\prec\Omega$ for every $\xi\leq \theta$.
\item[B3)] The elements $\Gamma_\xi$ inherit the ordering from $\theta$.
\item[B4)] If $a\in SC_\theta$ or $a\equiv\varphi bc$ then $a_1+\ldots+a_n\prec a$ if $a_1\prec a$.
\item[B5)] If $a\in SC_\theta$ then $\varphi bc\prec a$ if $b,c\prec a$.
\item[B6)] $\psi b\prec\Omega$ for all $b$.
\item[B7)] $\psi a\succ\Gamma_\xi$ for all $\xi\leq \theta$.
\item[B8)] $a_1+\ldots+a_n\prec b_1+\ldots+b_m$ \begin{minipage}[t]{15cm} if $n<m\:\text{and}\:(\forall i\leq n )[a_i\equiv b_i]$\\
$\text{or }\exists i\leq\text{min}(n,m)[\forall j<i (a_j=b_j)\text{ and }a_i\prec b_i].$\end{minipage}
\item[B9)] $\varphi a_1 b_1\prec\varphi a_2 b_2$ \begin{minipage}[t]{15cm} if $a_1\prec a_2\:\wedge\: b_1\prec\varphi a_2 b_2$\\
or  $a_1=a_2\:\wedge\:b_1\prec b_2$\\
or  $a_2\prec a_1\:\wedge\:\varphi a_1b_1\prec b_2$.\end{minipage}
\item[B10)] $\psi a\prec\psi b$ if $a\prec b$.
\end{description}
\end{proof}

\section{The proof theory of $\RSX$}
\subsection{A Tait-style sequent calculus formulation of \textbf{KP}}
\begin{definition}{\em The language of \textbf{KP} consists of free variables $a_0,a_1,\ldots$, bound variables $x_0,x_1,\ldots$, the binary predicate symbols $\in$, $\notin$ and the logical symbols $\vee, \wedge, \forall, \exists$ as well as parentheses $),($.\\

\noindent The atomic formulas are those of the form
\begin{equation*}
(a\in b)\quad ,\quad (a\notin b)
\end{equation*}
The formulas of $\textbf{KP}$ are defined inductively by:
\begin{description}
\item[i)] Atomic formulas are formulas.
\item[ii)] If $A$ and $B$ are formulas then so are $A\vee B$ and $A\wedge B$.
\item[iii)] If $A(b)$ is a formula in which the bound variable $x$ does not occur, then $\forall x A(x)$, $\exists x A(x)$, $(\forall x\in a)A(x)$ and $(\exists x\in a)A(x)$ are all formulas.
\end{description}
\noindent Quantifiers of the form $\exists x$ and $\forall x$ will be called unbounded and those of the form $(\exists x\in a)$ and $(\forall x\in a)$ will be referred to as bounded quantifiers.\\

\noindent A formula is said to be $\Delta_0$ if it contains no unbounded quantifiers. A formula is said to be $\Sigma$ ($\Pi$) if it contains no unbounded universal (existential) quantifiers.\\

\noindent The negation $\neg A$ of a formula $A$ is obtained from $A$ by undergoing the following operations:

\begin{description}
\item[i)]Replacing every occurrence of $\in$,$\notin$ with $\notin$,$\in$ respectively.
\item[ii)]Replacing any occurrence of $\wedge,\vee,\forall x,\exists x,(\forall x\in a),(\exists x\in a)$ with $\vee,\wedge,\exists x,\forall x,(\exists x\in a),(\forall x\in a)$ respectively.
\end{description}
\noindent Thus the negation of a formula $A$ is in negation normal form. The expression $A\rightarrow B$ will be considered shorthand for $\neg A\vee B$.\\

\noindent The expression $a=b$ is to be treated as an abbreviation for $(\forall x\in a)(x\in b)\wedge (\forall x\in b)(x\in a)$.\\

\noindent The derivations of \textbf{KP} take place in a Tait-style sequent calculus, finite sets of formulae denoted by Greek capital letters are derived. Intuitively the sequent $\Gamma$ may be read as the disjunction of formulae occuring in $\Gamma$.\\

\noindent The axioms of \textbf{KP} are:\\

\noindent \begin{tabular} {ll}
{\em {Logical axioms:}} &$\Gamma,A,\neg A$ for any formula $A$.\\
{\em {Extensionality:}} &$\Gamma,a=b\wedge B(a)\rightarrow B(b)$ for any formula $B(a)$.\\
{\em {Pair:}} & $\Gamma,\exists z(a\in z\wedge b\in z)$.\\
{\em {Union:}} & $\Gamma,\exists z(\forall y\in a)(\forall x\in y)(x\in z)$.\\
{\em {$\Delta_0$-Separation:}} & $\Gamma,\exists y[(\forall x\in y)(x\in a\wedge B(x))\wedge (\forall x\in a)(B(x)\rightarrow x\in y)]$\\
& for any $\Delta_0$-formula $B(a)$.\\
{\em {Set Induction:}} &$\Gamma,\forall x[(\forall y\in x F(y)\rightarrow F(x)]\rightarrow\forall x F(x)$ for any formula $F(a)$.\\
{\em {Infinity:}} &$\Gamma,\exists x[(\exists z\in x)(z\in x)\wedge(\forall y\in x)(\exists z\in x)(y\in z)]$.\\
{\em {$\Delta_0$-Collection:}} & $\Gamma,(\forall x\in a)\exists y G(x,y)\rightarrow\exists z(\forall x\in a)(\exists y\in z)G(x,y)$\\
& for any $\Delta_0$-formula $G$.
\end{tabular}\\

\noindent The rules of inference are
\begin{prooftree}
\AxiomC{$\Gamma,A$}
\AxiomC{$\Gamma,B$}
\LeftLabel{$(\wedge)$}
\BinaryInfC{$\Gamma,A\wedge B$}
\end{prooftree}
\begin{prooftree}
\AxiomC{$\Gamma, A$}
\LeftLabel{$(\vee)$}
\UnaryInfC{$\Gamma,A\vee B$}
\AxiomC{$\Gamma, B$}
\UnaryInfC{$\Gamma,A\vee B$}
\noLine
\BinaryInfC{}
\end{prooftree}
\begin{prooftree}
\Axiom$\fCenter\Gamma, a\in b\wedge F(a)$
\LeftLabel{$(b\exists)$}
\UnaryInf$\fCenter\Gamma,(\exists x\in b)F(x)$
\Axiom$\fCenter\Gamma, F(a)$
\LeftLabel{$(\exists)$}
\UnaryInf$\fCenter\Gamma,\exists xF(x)$
\noLine
\BinaryInf$\fCenter$
\end{prooftree}
\begin{prooftree}
\Axiom$\fCenter\Gamma, a\in b\rightarrow F(a)$
\LeftLabel{$(b\forall)$}
\UnaryInf$\fCenter\Gamma,(\forall x\in b)F(x)$
\Axiom$\fCenter\Gamma, F(a)$
\LeftLabel{$(\forall)$}
\UnaryInf$\fCenter\Gamma,\forall xF(x)$
\noLine
\BinaryInf$\fCenter$
\end{prooftree}
\begin{prooftree}
\AxiomC{$\Gamma,A$}
\AxiomC{$\Gamma,\neg A$}
\LeftLabel{(Cut)}
\BinaryInfC{$\Gamma$}
\end{prooftree}
In both $(b\forall)$ and $(\forall)$, the variable $a$ must not be present in the conclusion, such a variable is referred to as the \textit{eigenvariable} of the inference.\\

\noindent The \textit{minor formulae} of an inference are those rendered prominently in the premises, the other formulae in the premises will be referred to as \textit{side formulae}. The \textit{principal} formula of an inference is the one rendered prominently in the conclusion. Note that the principal formula can also be a side formula of that inference, when this happens we say that there has been a \textit{contraction}. The rule $\Cut$ has no principal formula.
}\end{definition}
\noindent Formally, bounded and unbounded quantifiers are treated as logically separate operations. However, it is important to know and ensure that they interact with one another as expected. %As an %example of a \textbf{KP} derivation, it is informative to show that the bounded and unbounded quantifiers % with one another as expected.
\begin{lemma}{\em The following are derivable within \textbf{KP}:
\begin{description}
\item[i)]$(\forall x\in b)F(x)\leftrightarrow\forall x(x\in b\rightarrow F(x))$.
\item[ii)]$(\exists x\in b)F(x)\leftrightarrow\exists x (x\in b\wedge F(x))$.
\end{description}
}\end{lemma}
\begin{proof}
We verify only i) as the proof of ii) is very similar. First note that $a\in b\wedge\neg F(a), a\in b\rightarrow F(a)$ is a logical axiom of \textbf{KP}, we have the following derivation in \textbf{KP}.
\begin{prooftree}
\Axiom$\fCenter a\in b\wedge\neg F(a),a\in b\rightarrow F(a)$
\LeftLabel{$(b\exists)$}
\UnaryInf$\fCenter (\exists x\in b)\neg F(x),a\in b\rightarrow F(a)$
\LeftLabel{$(\forall)$}
\UnaryInf$\fCenter (\exists x\in b)\neg F(x),\forall x(x\in b\rightarrow F(x))$
\LeftLabel{$(\vee)$ twice}
\UnaryInf$\fCenter (\forall x\in b)F(x)\rightarrow\forall x (x\in b\rightarrow F(x))$
\Axiom$\fCenter a\in b\wedge\neg F(a), a\in b\rightarrow F(a)$
\LeftLabel{$(\exists)$}
\UnaryInf$\fCenter \exists x(x\in b\wedge \neg F(x)), a\in b\rightarrow F(a)$
\LeftLabel{$(b\forall)$}
\UnaryInf$\fCenter\exists x(x\in b\wedge \neg F(x)), (\forall x\in b)F(x)$
\LeftLabel{$(\vee)$ twice}
\UnaryInf$\fCenter\forall x (x\in b\rightarrow F(x))\rightarrow (\forall x\in b)F(x)$
\LeftLabel{$(\wedge)$}
\BinaryInf$\fCenter(\forall x\in b)F(x)\leftrightarrow\forall x (x\in b\rightarrow F(x))$
\end{prooftree}
\end{proof}
\subsection{The infinitary system $\RSX$}
\noindent Let $X$ be an arbitrary (well founded) set and let $\theta$ be the set-theoretic rank of $X$ (hereby referred to as the $\in$-rank). Henceforth all ordinals are assumed to belong to the ordinal notation system  $T(\theta)$ developed in the previous section. The system $\RSX$ will be an infinitary proof system based on $L_\OO(X)$; the relativised constructible hierarchy up to $\OO$.
\begin{definition}\label{level1}{\em
We give an inductive definition of the set $\mathcal{T}$ of $\RSX$ terms, to each term $t\in\mathcal{T}$ we assign an ordinal level $\lev t$
\begin{description}
\item[i)] For every $u\in TC(\{X\})$, $\bar{u}\in\mathcal{T}$ and $\lev{\bar{u}}:=\Gamma_{\text{rank}(u)}$ [here rank$(u)$ is the $\in$-rank of $u$ and $TC$ denotes the \textit{transitive closure} operator.] Note that $\mbox{rank}(u)\leq\theta$.
\item[ii)]For every $\alpha<\Omega$, $\mathbb{L}_\alpha (X)\in\mathcal{T}$ and $\lev{\mathbb{L}_\al(X)}:=\Gamma_{\theta+1} + \alpha$.
\item[iii)]If  $\alpha<\Omega$, $A(a,b_1,\ldots,b_n)$ is a formula of \textbf{KP} with all free variables displayed and $s_1,\ldots,s_n$ are terms with levels less than $\Gamma_{\theta +1}+\alpha$ then
\begin{equation*}
[x\in\mathbb{L}_\alpha (X) | A(x,s_1,\ldots,s_n)^{\mathbb{L}_\alpha(X)}]
\end{equation*}
is a term of level $\Gamma_{\theta+1} + \alpha$. Here the superscript $\mathbb{L}_\alpha (X)$ indicates that all unbounded quantifiers occuring in $A$ are replaced by quantifiers bounded by $\mathbb{L}_\alpha (X)$.
\end{description}
}\end{definition}
\noindent The terms of $\RSX$ are to be viewed as purely formal, syntactic objects. However their names are highly suggestive of the intended interpretation in the relativised constructible hierarchy up to $\Omega$.
\begin{definition}{\em
The formulae of $\RSX$ are of the form $A(s_1,\ldots,s_n)$, where $A(a_1,\ldots,a_n)$ is a formula of $\KP$ with all free variables displayed and $s_1,\ldots,s_n$ are $\RSX$ terms.\\

\noindent Formulae of the form $\bar{u}\in\bar{v}$ and $\bar{u}\notin\bar{v}$ will be referred to as \textit{basic}. The properties $\Delta_0$, $\Sigma$ and $\Pi$ are inherited from $\KP$ formulae.
}\end{definition}
\noindent Note that the system $\RSX$ does not contain free variables\\

\noindent For the remainder of this section we shall refer to $\RSX$ terms and formulae simply as terms and formulae.\\

\noindent For any formula $A$ we define
\begin{align*}
k(A):=&\{\lev t\;|\;\text{$t$ occurs in $A$, subterms included}\}\\
&\cup\{\Omega\;|\;\text{if $A$ contains an unbounded quantifier}\}.
\end{align*}
If $\Gamma$ is a finite set of the $\RSX$ formulae $A_1,\ldots,A_n$ then we define
\begin{equation*}
k(\Gamma):=k(A_1)\cup\ldots\cup k(A_n).
\end{equation*}
\begin{abbreviatio}\label{abbreviations}{\em $\phantom{a}$
\begin{description}
\item[i)] For $\RSX$ terms $s$ and $t$, the expression $s=t$ will be considered as shorthand for
\begin{equation*}(\forall x\in s)(x\in t)\wedge (\forall x\in t)(x\in s).
\end{equation*}
\item[ii)] If $\lev s<\lev t$,  $A(s,t)$ is an $\RSX$ formula and $\Diamond$ is a propositional connective we define:
\begin{equation*}
s\dotin t \:\Diamond\: A(s,t):= \begin{cases} s\in t\:\Diamond\: A(s,t) &\mbox{if } t \equiv \bar{u} \\
A(s,t) & \mbox{if } t \equiv \mathbb{L}_\alpha (X)\\
B(s)\:\Diamond\: A(s,t) & \mbox{if } t\equiv [x\in\mathbb{L}_\alpha (X)\;|\;B(x)]
 \end{cases}
\end{equation*}
\end{description}
}\end{abbreviatio}
\noindent Our aim will be to remove cuts from certain $\RSX$ derivations of $\Sigma$ sentences. In order to do this we need to express a certain kind of uniformity in infinite derivations. The right tool for expressing this uniformity was developed by Buchholz in \cite{bu93} and is termed \textit{operator control}.
\begin{definition}\label{operator}{\em
Let $\mathcal{P}(\textbf{ON}):=\{Y : \text{ Y is a set of ordinals}\}$. A class function
\begin{equation*}
\mathcal{H}:\mathcal{P}(\textbf{ON})\rightarrow\mathcal{P}(\textbf{ON})
\end{equation*}
is called an Operator if the following conditions are satisfied for $Y,Y^\prime\in\mathcal{P}(\textbf{ON})$
\begin{description}
\item[(H1)] $0\in\mathcal{H}(Y)$ and $\Gamma_\beta\in\mathcal{H}(Y)$ for any $\beta\leq\theta+1$.
\item[(H2)] If $\alpha=_{\text{NF}}\alpha_1+\ldots+\alpha_n$ then $\alpha\in\mathcal{H}(Y)$ iff $\alpha_1,\ldots,\alpha_n\in\mathcal{H}(Y)$.
\item[(H3)] If $\alpha=_{\text{NF}}\varphi\alpha_1 \alpha_2$ then $\alpha\in\mathcal{H}(Y)$ iff $\alpha_1,\alpha_2\in\mathcal{H}(Y)$
\item[(H4)] $Y\subseteq\mathcal{H}(Y)$
\item[(H5)] $Y^\prime\subseteq\mathcal{H}(Y)\Rightarrow\mathcal{H}(Y^\prime)\subseteq\mathcal{H}(Y)$
\end{description}
}\end{definition}
\noindent Note that this definition of operator, as with the infinitary system $\RSX$ is dependent on the set $X$ and its $\in$-rank $\theta$.
\begin{abbreviatio}{\em For an operator $\mathcal{H}$:
\begin{description}
\item[i)] We write $\alpha\in\mathcal{H}$ instead of $\alpha\in\mathcal{H}(\emptyset)$.
\item[ii)] Likewise $Y\subseteq\mathcal{H}$ is shorthand for $Y\subseteq\mathcal{H}(\emptyset)$.
\item[iii)] For any $\RSX$ term $t$, $\mathcal{H}[t](Y):=\mathcal{H}(Y\cup\lev t)$.
 \item[iv)] If $\mathfrak{X}$ is an $\RSX$ formula or set of formulae then $\mathcal{H}[\mathfrak{X}](Y):=\mathcal{H}(Y\cup k(\mathfrak{X}))$.
\end{description}
}\end{abbreviatio}
\begin{lemma}\label{RSXop}{\em
Let $\CH$ be an operator $s$ an $\RSX$ term and $\mathfrak{X}$ an $\RSX$ formula or set of formulae.
\begin{description}
\item[i)] If $Y\subseteq Y^\prime$ then $\CH(Y)\subseteq\CH(Y^\prime)$.
\item[ii)] $\CH[s]$ and $\CH[\mathfrak{X}]$ are operators.
\item[iii)] If $\lev s\in\CH$ then $\CH[s]=\CH$.
\item[iv)] If $k(\mathfrak{X})\subseteq\CH$ then $\CH[\mathfrak{X}]=\CH$.
\end{description}
}\end{lemma}
\begin{proof}
These results are easily checked, they are proved in full in \cite{mi99}.
\end{proof}
\begin{definition}{\em
If $\mathcal{H}$ is an operator, $\alpha$ an ordinal and $\Gamma$ a finite set of $\RSX$-formulae, we give an inductive definition of the relation $\CH\;\prov{\alpha}{}{\GA}$ by recursion on $\al$. (\textit{$\mathcal{H}$-controlled derivability in $\RSX$}.) We require always that
\begin{equation*}
\{\alpha\}\cup k(\Gamma)\subseteq\mathcal{H}
\end{equation*}
this condition will not be repeated in the inductive clauses pertaining to the axioms and inference rules below.
We have the following axioms:
\begin{align*}
\provx{\mathcal{H}}{\alpha}{}{\Gamma,\bar{u}\in\bar{v}}\quad &\text{if}\quad u,v\in TC(X)\:\text{and}\: u\in v\\
\provx{\mathcal{H}}{\alpha}{}{\Gamma,\bar{u}\notin\bar{v}}\quad &\text{if}\quad u,v\in TC(X)\:\text{and}\: u\notin v.
\end{align*}
The following are the inference rules of $\RSX$, the column on the right gives the requirements on the ordinals, terms and formulae for each rule.
$$ \begin{array}{lcr}

(\wedge) & \inftwo{\CH}{\al_0}{\GA,A}{\CH}{\al_1}{\GA,B}
{\CH}{\al}{\GA,A\wedge B}
& \al_0,\al_1 < \alpha \\[0.6cm]

(\vee) & \infone{\mathcal H}{\alpha_0}{}
{\Gamma,C\quad\text{for some $C\in\{A,B\}$}} {\mathcal H}{\alpha}{}{\GA, A\vee B}
& \alpha_{0} < \alpha
\end{array}$$

$$\begin{array}{lcr}
(\notin) & \infone{\mathcal H[s]}{\alpha_s}{}
{\Gamma,s\dotin t\rightarrow r\neq s\quad\text{for all $\lev s<\lev t$}} {\mathcal H}{\alpha}{}{\GA, r\notin t}
&\begin{array}{r}\alpha_{s} < \alpha\\ \text{$r\in t$ is not basic}
\end{array}\\[0.6cm]

(\in) & \infone{\CH}{\al_0}{}{\Gamma,s\dotin t\wedge r=s}
{\CH}{\al}{}{\Gamma, r\in t} &\begin{array}{r} \al_0<\al\\
\lev s<\lev t\\
\lev s<\Gamma_{\theta+1}+\al\\
\text{$r\in t$ is not basic}\end{array}
\end{array}$$

$$\begin{array}{lcr}
(b\forall) &\infone{\CH[s]}{\al_s}{}{\GA,s\dotin t\rightarrow A(s)\quad\text{for all $\lev s<\lev t$}}
{\CH}{\al}{}{\GA,(\forall x\in t)A(x)} & \al_s<\al \\[0.6cm]

(b\exists) &\infone{\CH}{\al_0}{}{\GA, s\dotin t\wedge A(s)}
{\CH}{\al}{}{\GA, (\exists x\in t)A(x)} & \begin{array}{r} \al_0<\al\\
\lev s<\lev t\\
\lev s<\Gamma_{\theta+1}+\al\end{array}
\end{array}$$

$$\begin{array}{lcr}
(\forall) &\infone{\CH[s]}{\al_s}{}{\GA, A(s)\quad\text{for all $s$}}
{\CH}{\al}{}{\GA,\forall xA(x)} & \al_s<\al\\[0.6cm]

(\exists) &\infone{\CH}{\al_0}{}{\GA, A(s)}
{\CH}{\al}{}{\GA,\exists xA(x)} & \begin{array}{r} \al_0<\al\\
\lev s<\Gamma_{\theta+1}+\al\end{array}
\end{array}$$

$$\begin{array}{lcr}
\Cut & \inftwo{\CH}{\al_0}{\Gamma,A}{\CH}{\al_0}{\GA,\neg A}
{\CH}{\al}{\GA} & \al_0<\al\\[0.6cm]

\SRX &\infone{\CH}{\al_0}{}{\Gamma, A}
{\CH}{\al}{}{\GA,\exists zA^z} & \begin{array}{r} \al_0,\OO<\al\\
A\text{ is a $\Sigma$ formula}
\end{array}
\end{array}$$
}\end{definition}
\noindent $A^z$ results from $A$ by restricting all unbounded quantifiers in $A$ to $z$. The reason for the condition preventing the derivation of basic formulas in the rules $(\in)$ and $(\notin)$ is to prevent derivations of sequents which are already axioms, as this would cause a hindrance to cut-elimination. The condition that $\lev s<\Gamma_{\theta+1}+\alpha$ in $(\in)$ and $(\exists)$ inferences will allow us to place bounds on the location of witnesses in derivable $\Sigma$ formulas.

\subsection{Cut elimination for $\RSX$}

\noindent We need to keep track of the complexity of cuts appearing in a derivation, to this end we define the \textit{rank} of an $\RSX$ formula.
\begin{definition}{\em
The rank of a term or formula is defined by recursion on the construction as follows:
\begin{description}
\item[1.] $rk(\bar{u}):=\Gamma_{\textit{rank}(u)}$
\item[2.] $rk(\mathbb{L}_\alpha (X)):=\Gamma_{\theta +1} + \omega\cdot\alpha$
\item[3.] $rk([x\in\mathbb{L}_\alpha (X)|F(x)]):=\text{max}(\Gamma_{\theta +1} + \omega\cdot\alpha +1, rk(F(\bar{\emptyset}))+2)$
\item[4.] $rk(s\in t):=rk(s\notin t):=\text{max}(rk(t)+1,rk(s)+6)$
\item[5.] $rk((\exists x\in \bar{u})F(x)):=rk((\forall x\in \bar{u})F(x)):=\text{max}(rk(\bar{u})+3,rk(F(\bar{\emptyset}))+2)$.
\item[6.] $rk((\exists x\in t)F(x)):=rk((\forall x\in t)F(x)):=\text{max}(rk(t),rk(F(\bar{\emptyset}))+2)$ if $t$ is not of the form $\bar{u}$.
\item[7.] $rk(\exists xF(x)):=rk(\forall xF(x)):=\text{max}(\OO,rk(F(\bar\emptyset))+1)$
\item[8.] $rk(A\wedge B):=rk(A\vee B):=\text{max}(rk(A),rk(B))+1$
\end{description}
$\provx{\mathcal{H}}{\alpha}{\rho}{\Gamma}$ will be used to denote that $\provx{\mathcal{H}}{\alpha}{}{\Gamma}$ and all cut formulas appearing in the derivation have rank $<\rho$.
}\end{definition}

\begin{observation}\label{obs}{\em
\begin{description}
\item[i)] For each term $t$, $rk(t)=\omega\cdot\lev t+n$ for some $n<\omega$.
\item[ii)] For each formula $A$, $rk(A)=\omega\cdot\text{max}(k(A))+n$ for some $n<\omega$.
\item[iii)] $rk(A)<\OO$ if and only if $A$ is $\Delta_0$.
\end{description}
}\end{observation}
\noindent The next Lemma shows that the rank of a formula $A$ is determined \textit{only} by max$(k(A))$ and the logical structure of $A$.
\begin{lemma}\label{rksub}{\em
For each formula $A(s)$, if $\lev s<\text{max}(k(A(s)))$ then $rk(A(s))=rk(A(\bar{\emptyset}))$.
}\end{lemma}
\begin{proof}The proof is by induction on the complexity of $A$.\\

\noindent Case 1. If $A(s)\equiv s\in t$ then by assumption $\lev s<\lev t$, so $rk(A(s))=rk(t)+1=rk(A(\bar{\emptyset}))$.\\

\noindent Case 2. If $A(s)\equiv t\in s$ we may argue in a similar fashion to Case 1.\\

\noindent Case 3. It cannot be the case that $A(s)\equiv s\in s$.\\

\noindent Case 4. If $A(s)\equiv (\exists y\in \bar{u})B(y,s)$ then
\begin{equation*}
rk(A(s))=\text{max}(rk(\bar{u})+3,rk(B(\bar{\emptyset},s))+2)
\end{equation*}
and
\begin{equation*}
rk(A(\bar{\emptyset}))=\text{max}(rk(\bar{u})+3,rk(B(\bar{\emptyset},\bar{\emptyset}))+2).
\end{equation*}
4.1 If $\lev{\bar{u}}>\text{max}(k(B(\bar{\emptyset},\bar{\emptyset})))$ then $\lev{s}<\lev{\bar{u}}$ by assumption, so using observation \ref{obs}ii) gives us
\begin{equation*}
rk(A(s))=rk(\bar{u})+3=rk(A(\bar{\emptyset})).
\end{equation*}
4.2 If $\lev{\bar{u}}\leq\text{max}(k(B(\bar{\emptyset},\bar{\emptyset}))$ then $\lev s<\text{max}(k(B(\bar{\emptyset},\bar{\emptyset})))$ by assumption, so by induction hypothesis
\begin{equation*}
rk(B(\bar{\emptyset},s))=rk(B(\bar{\emptyset},\bar{\emptyset}))
\end{equation*}
and hence using Observation \ref{obs}ii) gives us
\begin{equation*}
rk(A(s))=rk(B(\bar{\emptyset},\bar{\emptyset}))+2=rk(A(\bar{\emptyset})).
\end{equation*}
Case 5. If $A(s)\equiv (\exists y\in t)B(y,s)$ for some $t$ not of the form $\bar{u}$, we may argue in a similar way to case 4.\\

\noindent Case 6. $A(s)\equiv (\exists y\in s)B(y,s)$, now $\lev s<\text{max}(k(A(\bar{\emptyset})))=\text{max}(k(B(\bar{\emptyset},\bar{\emptyset})))$, so by induction hypothesis
\begin{equation*}
rk(B(\bar{\emptyset},s))=rk(B(\bar{\emptyset},\bar{\emptyset}))
\end{equation*}
and hence using observation \ref{obs} we see that
\begin{align*}
rk(A(s))&=rk(B(\bar{\emptyset},s))+2\\
&=rk(B(\bar{\emptyset},\bar{\emptyset}))+2\\
&=rk(A(\bar{\emptyset})).
\end{align*}
Case 7. If $A(s)\equiv\exists xB(x,s)$ then by assumption $\lev s<\text{max}(k(A(s)))=\text{max}(k(B(\emptyset,s)))$ so we may apply the induction hypothesis to see that $rk(A(s))=\text{max}(\OO,rk(B(\emptyset,s))+1)=\text{max}(\OO,rk(B(\emptyset,\emptyset))+1)=rk(A(\emptyset)).$\\

\noindent Case 8. All other cases are either propositional in which case we may just use the induction hypothesis directly or are dual to cases already considered.
\end{proof}
\begin{definition}{\em To each non-basic formula $A$ we assign an infinitary disjunction $(\bigvee A_i)_{i\in y}$ or conjunction $(\bigwedge A_i)_{i\in y}$ as follows:
\begin{description}
\item[1.] $r\in t :\simeq\bigvee (s\dotin t\wedge r=s)_{\lev s<\lev t}$ provided $r\in t$ is not a basic formula.
\item[2.] $(\exists x\in t)B(x):\simeq\bigvee(s\dotin t\wedge B(s))_{\lev s<\lev t}$
\item[3.] $\exists xB(x):\simeq\bigvee(B(s))_{s\in\mathcal{T}}$
\item[4.] $B_0\vee B_1 :\simeq\bigvee(B_i)_{i\in\{0,1\}}$
\item[5.] $\neg B :\simeq\bigwedge (\neg B_i)_{i\in y}$ if $B$ is of the form considered in 1.-4.
\end{description}
The idea is that the infinitary conjunction or disjunction lists the premises required to derive $A$ as the principal formula of an $\RSX$-inference different from $\SRX$ or $\Cut$.
}\end{definition}
\begin{lemma} \label{ranklemma}
If $A\simeq(\bigvee A_i)_{i\in y}$ or $A\simeq(\bigwedge A_i)_{i\in y}$ then
\begin{equation*}
\forall i\in y(rk(A_i)<rk(A))
\end{equation*}
\end{lemma}
\begin{proof}
We need only treat the case where $A\simeq(\bigvee A_i)_{i\in y}$ since the other case is dual to this one. We proceed by induction on the complexity of $A$.\\

\noindent Case 1. Suppose $A\equiv r\in t$ then by assumption either $r$ or $t$ is not of the form $\bar{u}$, we split cases based on the form of $t$.\\

\noindent 1.1 If $t\equiv\bar{u}$ then $r$ is not of the form $\bar v$ and $rk(A)=rk(r)+6$. In this case $A_i\equiv\bar{v}\in\bar{u}\wedge\bar{v}=r$ for some $\lev{\bar{v}}<\lev{\bar{u}}$ and we have
\begin{align*}
rk(A_i)&=\text{max}(rk(\bar{v}\in\bar{u}),rk(\bar{v}=r))+1\\
&=rk(\bar{v}=r)+1\\
&=\text{max}(rk((\forall x\in\bar{v})(x\in r)),rk((\forall x\in r)(x\in \bar{v})))+2\\
&=rk(r)+5<rk(r)+6=rk(A)
\end{align*}
1.2 If $t\equiv\mathbb{L}_\alpha (X)$ then $A_i\equiv s=r$ for some $\lev s<\lev t$. So we have
\begin{align*}
rk(A_i)&=rk((\forall x\in s)(x\in r)\wedge(\forall x\in r)(x\in s))\\
&=\text{max}(rk(s)+4,rk(r)+4)\\
&<\text{max}(rk(r)+1,rk(t)+6)=rk(A)
\end{align*}
1.3 If $t\equiv [x\in\mathbb{L}_\alpha (X)|B(x)]$ then $A_i\equiv B(s)\wedge s=r$ for some $\lev s<\lev t$. So we have
\begin{equation*}
rk(A_i)=\text{max}(rk(B(s))+1,rk(r=s)+1).
\end{equation*}
First note that $rk(r=s)+1=\text{max}(rk(s)+5,rk(r)+5)<rk(A)$. So it remains to verify that $rk(B(s))+1<rk(A)$, for this it is enough to show that $rk(B(s))<rk(t)$.\\

\noindent 1.3.1 If $\text{max}(k(B(s)))\leq\lev s$ then by Observation \ref{obs}ii) we have $rk(B(s))+1<\omega\cdot\lev s +\omega\leq rk(t)$.\\

\noindent 1.3.2 Otherwise $\text{max}(k(B(s)))>\lev s$ then by Lemma \ref{rksub} we have
\begin{align*}
rk(B(s))+1&=rk(B(\bar\emptyset))+1\\
&<\text{max}(\Gamma_{\theta+1}+\omega\cdot\alpha +1, rk(B(\bar\emptyset))+2)=rk(t)
\end{align*}

\noindent Case 2. Suppose $A\equiv (\exists x\in t)B(x)$, we split into cases based on the form of $t$.\\

\noindent 2.1 If $t\equiv\bar u$ then $rk(A):=\text{max}(rk(\bar{u})+3,rk(B(\bar{\emptyset}))+2)$. In this case $A_i\equiv\bar{v}\in\bar{u}\wedge B(\bar{v})$ for some $\lev{\bar{v}}<\lev{\bar{u}}$, so we have
\begin{equation*}
rk(A_i)=\text{max}(rk(\bar{u})+2,rk(B(\bar{v}))+1).
\end{equation*}
Clearly $rk(\bar{u})+2<rk(\bar{u})+3$ so it remains to verify that $rk(B(\bar{v}))+1<rk(A)$\\

\noindent 2.1.1 If $|\bar{v}|\geq\text{max}(k(B(\bar{v})))$ then by Observation \ref{obs}i) $rk(B(\bar{v}))+1<rk(\bar{u})<rk(\bar{u})+3$.\\

\noindent 2.1.2 If $|\bar{v}|<\text{max}(k(B(\bar{v})))$ then by Lemma \ref{rksub} $rk(B(\bar{v}))+1=rk(B(\bar{\emptyset}))+1<rk(B(\bar{\emptyset}))+2$.\\

\noindent 2.2 Now suppose $t\equiv\mathbb{L}_\alpha (X)$, so $rk(A)=\text{max}(rk(t), rk(B(\bar\emptyset))+2)$. In this case $A_i=B(s)$ for some $\lev s<\lev t$.\\

\noindent 2.2.1 If $\lev s\geq\text{max}(k(B(s)))$ then $rk(B(s))<rk(t)$ by Observation \ref{obs}.\\

\noindent 2.2.2 If $\lev s<\text{max}(k(B(s)))$ then by Lemma \ref{rksub} $rk(B(s))=rk(B(\bar\emptyset))<rk(A)$.\\

\noindent 2.3. Now suppose $t\equiv [y\in\mathbb{L}_\alpha (X)\:|\:C(y)]$, so we have
\begin{align*}
rk(A)&:=\text{max}(rk(t),rk(B(\bar\emptyset))+2)\\
&=\text{max}(\Gamma_{\theta+1}+\omega\cdot\alpha+1,rk(C(\bar\emptyset))+2,rk(B(\bar\emptyset))+2).
\end{align*}

\noindent In this case $A_i\equiv C(s)\wedge B(s)$ for some $\lev s<\lev t$.\\

\noindent 2.3.1 If $\lev s<\text{max}(k(B(s)))$ then $rk(B(s))+1=rk(B(\bar\emptyset))+1<rk(B(\bar\emptyset))+2$. It remains to show that $rk(C(s))<rk(A)$.\\

\noindent 2.3.1.1 If $\text{max}(k(C(s)))<\lev t$ then $rk(C(s))+1<rk(t)$ by Observation \ref{obs}.\\

\noindent 2.3.1.2 Now if $\text{max}(k(C(s)))\geq\lev t$ then we may apply Lemma \ref{rksub} to give
\begin{equation*}
rk(C(s))+1=rk(C(\bar\emptyset))+1<rk(C(\bar\emptyset))+2\leq rk(A).
\end{equation*}

\noindent 2.3.2 If $\lev s\geq\text{max}(k(B(s)))$ then $rk(B(s))<\Gamma_{\theta+1}+\omega\cdot\alpha$ by Observation \ref{obs}. Now we may apply the same argument as in 2.3.1.1 and 2.3.1.2 to yield $rk(C(s))+1<rk(A)$.\\

\noindent Case 3. If $A\equiv\exists xB(x)$ then $rk(A):=\text{max}(\OO,rk(B(\bar\emptyset))+1)$. In this case $A_i\equiv B(s)$ for some term $s$.\\

\noindent 3.1 If $B$ contains an unbounded quantifier then by Lemma \ref{rksub} $rk(B(s))=rk(B(\bar\emptyset))<rk(A)$.\\

\noindent 3.2 If $B$ does not contain an unbounded quantifier then $rk(B(s))<\OO$ by Observation \ref{obs}iii)\\

\noindent Case 4. If $A\equiv B\vee C$ then the result is clear immediately from the definition of $rk(A)$.
\end{proof}
\begin{lemma}\label{xinversion}{\em Let $\mathcal{H}$ be an arbitrary operator.  \begin{description}
\item[i)] If $\alpha\leq\alpha^{\prime}\in\mathcal{H}$, $\rho\leq\rho^{\prime}$, $k(\Gamma^{\prime})\subseteq\mathcal{H}$ and $\provx{\mathcal{H}}{\alpha}{\rho}{\Gamma}$
then $\provx{\mathcal{H}}{\alpha^{\prime}}{\rho^{\prime}}{\Gamma,\Gamma^{\prime}}$.
\item[ii)] If $C$ is a basic formula which holds true in the set $X$ and $\provx{\mathcal{H}}{\alpha}{\rho}{\Gamma,\neg C}$ then $\provx{\mathcal{H}}{\alpha}{\rho}{\Gamma}$.
\item[iii)] If $\provx{\mathcal{H}}{\alpha}{\rho}{\Gamma,A\vee B}$ then $\provx{\mathcal{H}}{\alpha}{\rho}{\Gamma,A,B}$.
\item[iv)] If $A\simeq\bigwedge(A_i)_{i\in y}$ and $\provx{\mathcal{H}}{\alpha}{\rho}{\Gamma,A}$ then $(\forall i\in y)\:\provx{\mathcal{H}[i]}{\alpha}{\rho}{\Gamma,A_i}$.
\item[v)] If $\gamma\in\mathcal{H}$ and $\provx{\mathcal{H}}{\alpha}{\rho}{\Gamma,\forall xF(x)}$ then $\provx{\mathcal{H}}{\alpha}{\rho}{\Gamma,(\forall x\in\mathbb{L}_\gamma (X))F(x)}$.
\end{description}
}\end{lemma}
\begin{proof}
All proofs are by induction on $\alpha$.\\

\noindent i) If $\Gamma$ is an axiom then $\Gamma,\Gamma^\prime$ is also an axiom, and since $\{\alpha^\prime\}\cup k(\Gamma^\prime)\subseteq\mathcal{H}$ there is nothing to show.\\

\noindent Now suppose $\Gamma$ is the result of an inference
\begin{prooftree}
\Axiom$\ldots\mathcal{H}_i\:\fCenter\prov{\alpha_i}{\rho}{\Gamma_i\ldots}$
\RightLabel{$ (i\in y)\quad\alpha_i<\alpha$}
\LeftLabel{(I)}
\UnaryInf$\mathcal{H}\:\fCenter\prov{\alpha}{\rho}{\Gamma}$
\end{prooftree}
Using the induction hypothesis we have
\begin{equation*}
\ldots\provx{\mathcal{H}_i}{\alpha_i}{\rho^\prime}{\Gamma_i,\Gamma^\prime}\ldots\quad (i\in y)\quad\alpha_i<\alpha
\end{equation*}
It's worth noting that $k(\Gamma^\prime)\subseteq\mathcal{H}_i$, since $\mathcal{H}_i(\emptyset)\supseteq\mathcal{H}(\emptyset)$, this can be observed by looking at each inference rule.\\

\noindent Finally we may apply the inference (I) again to obtain
\begin{equation*}
\provx{\mathcal{H}}{\alpha^\prime}{\rho^\prime}{\Gamma,\Gamma^\prime}
\end{equation*}
as required.\\

\noindent ii) If $\Gamma,\neg C$ is an axiom then so is $\Gamma$ so there is nothing to show.\\

\noindent Now suppose $\Gamma,\neg C$ was derived as the result of an inference rule (I), then $\neg C$ cannot have been the principal formula since it is basic so we have the premise(s)
\begin{equation*}
\provx{\mathcal{H}_i}{\alpha_i}{\rho}{\Gamma_i,\neg C}\quad \alpha_i<\alpha.
\end{equation*}
Now by induction hypothesis we obtain
\begin{equation*}
\provx{\mathcal{H}_i}{\alpha_i}{\rho}{\Gamma_i}\quad \alpha_i<\alpha
\end{equation*}
to which we may apply the inference rule (I) to complete the proof.\\

\noindent iii) If $\GA,A\vee B$ is an axiom then $\Gamma,A,B$ is also an axiom. If $A\vee B$ was not the principal formula of the last inference then we can apply the induction hypothesis to its premises and then the same inference again.\\

\noindent Now suppose that $A\vee B$ was the principal formula of the last inference. So we have
\begin{equation*}
\provx{\mathcal{H}}{\alpha_0}{\rho}{\Gamma,C}\quad\text{or}\quad\provx{\mathcal{H}}{\alpha_0}{\rho}{\Gamma,C,A\vee B}\quad\text{where $C\in\{A,B\}$ and $\alpha_0<\alpha$}
\end{equation*}
By i) we may assume that we are in the latter case. By the induction hypothesis, and a contraction, we obtain
\begin{equation*}
\provx{\mathcal{H}}{\alpha_0}{\rho}{\Gamma,A,B}
\end{equation*}
Finally using i) yields
\begin{equation*}
\provx{\mathcal{H}}{\alpha}{\rho}{\Gamma,A,B}.
\end{equation*}
\noindent iv) If $\Gamma,A$ is an axiom, then $\Gamma$ is also an axiom since $A$ cannot be the \textit{active part} of an axiom, so $\Gamma,A_i$ is an axiom for any $i\in y$. If $A$ was not the principal formula of the last inference then we may apply the induction hypothesis to its premises and then use that inference again.\\

\noindent Now suppose $A$ was the principal formula of the last inference. With the possible use of part $i)$, we may assume we are in the following situation:
\begin{equation*}
\provx{\mathcal{H}[i]}{\alpha_i}{\rho}{\Gamma, A, A_i}\quad (\forall i\in y)\quad \alpha_i<\alpha.
\end{equation*}
Inductively and via a contraction we obtain
\begin{equation*}
\provx{\mathcal{H}[i]}{\alpha_i}{\rho}{\Gamma,A_i}.
\end{equation*}
Here it is important to note that $\CH[i][i]\equiv\CH[i]$. To which we may apply part i) to obtain
\begin{equation*}
\provx{\mathcal{H}[i]}{\alpha}{\rho}{\Gamma,A_i}
\end{equation*}
as required.\\

\noindent v) The interesting case is where $\forall xF(x)$ was the principal formula of the last inference. In this case we may assume we are in the following situation:
\begin{equation*}
\tag{1}\provx{\mathcal{H}[s]}{\alpha_s}{\rho}{\Gamma, \forall xF(x), F(s)}\quad \text{for all terms $s$, with $\alpha_s<\alpha$.}
\end{equation*}
Using the induction hypothesis yields
\begin{equation*}
\tag{2}\provx{\mathcal{H}[s]}{\alpha_s}{\rho}{\Gamma, (\forall x\in\mathbb{L}_\gamma(X))F(x), F(s)}
\end{equation*}
Note that for $\lev s<\Gamma_{\theta+1}+\gamma$ we have $s\dotin\mathbb{L}_\gamma (X)\rightarrow F(s)\equiv F(s)$. So as a subset of (2) we have
\begin{equation*}
\provx{\mathcal{H}[s]}{\alpha_s}{\rho}{\Gamma, (\forall x\in\mathbb{L}_\gamma(X))F(x), s\dotin\mathbb{L}_\gamma (X)\rightarrow F(s)}\quad \text{for all $\lev{s}<\Gamma_{\theta+1}+\gamma$, with $\alpha_s<\alpha$.}
\end{equation*}
From which one application of $(b\forall)$ gives us the desired result.
\end{proof}

\begin{lemma}[Reduction for $\RSX$]\label{xreduction}{\em Suppose $C\equiv\bar{u}\in\bar{v}$ or $C\simeq\bigvee (C_i)_{i\in y}$ and $rk(C):=\rho\neq\Omega$.
\begin{equation*}
\text{If}\quad [\provx{\mathcal{H}}{\alpha}{\rho}{\Lambda,\neg C}\quad \&\quad \provx{\mathcal{H}}{\beta}{\rho}{\Gamma, C}]\quad\text{then}\quad\provx{\mathcal{H}}{\alpha+\beta}{\rho}{\Lambda,\Gamma}
\end{equation*}
}\end{lemma}
\begin{proof}
If $C\equiv\bar{u}\in\bar{v}$ then by \ref{xinversion}ii) we have either $\provx{\mathcal{H}}{\alpha}{\rho}{\Lambda}$ or $\provx{\mathcal{H}}{\beta}{\rho}{\Gamma}$. Hence using \ref{xinversion}i) we obtain $\provx{\mathcal{H}}{\alpha+\beta}{\rho}{\Lambda,\Gamma}$ as required.\\

\noindent Now suppose $C\simeq\bigvee (C_i)_{i\in y}$, we proceed by induction on $\beta$. We have
\begin{align*}
\tag{1}&\provx{\mathcal{H}}{\alpha}{\rho}{\Lambda,\neg C}\\
\tag{2}&\provx{\mathcal{H}}{\beta}{\rho}{\Gamma, C}.
\end{align*}
If $C$ was not the principal formula of the last inference in (2), then we may apply the induction hypothesis to the premises of that inference and then the same inference again. Now suppose $C$ was the principal formula of the last inference in (2). If $B$ was the principal formula of the inference $\SRX$, then $B$ is of the form $\exists zF(s_1,\ldots,s_n)^z$, which implies $rk(B)=\Omega$, therefore the last inference in (2) was not $\SRX$. So we have
\begin{equation*}
\tag{3}\provx{\mathcal{H}}{\beta_0}{\rho}{\Gamma,C,C_{i_0}}\quad\text{for some $i_0\in y$, $\beta_0<\beta$ with $\lev{i_0}<\Gamma_{\theta+1}+\beta$.}
\end{equation*}
The induction hypothesis applied to (2) and (3) yields
\begin{equation*}
\tag{4}\provx{\mathcal{H}}{\alpha+\beta_0}{\rho}{\Lambda,\Gamma,C_{i_0}}.
\end{equation*}
Now applying Lemma \ref{xinversion}iv) to (1) provides
\begin{equation*}
\tag{5}\provx{\mathcal{H}[i_0]}{\alpha}{\rho}{\Lambda,\neg C_{i_0}}.
\end{equation*}
But $\lev{i_0}\in\mathcal{H}$ by (4), which means $\mathcal{H}[i_0]=\mathcal{H}$ by Lemma \ref{RSXop}iv), so in fact we have
\begin{equation*}
\tag{6}\provx{\mathcal{H}}{\alpha}{\rho}{\Lambda,\neg C_{i_0}}.
\end{equation*}
Thus we may apply $\Cut$ to (4) and (6) (noting that $rk(C_{i_0})<rk(C):=\rho$ by Lemma \ref{ranklemma}) to obtain
\begin{equation*}
\provx{\mathcal{H}}{\alpha+\beta}{\rho}{\Lambda,\Gamma}
\end{equation*}
as required.
\end{proof}
\begin{theorem}[Predicative cut elimination for $\RSX$]\label{xpredce}{\em$\phantom{a}$ \\
\noindent If $\provx{\mathcal{H}}{\beta}{\rho+\omega^\alpha}{\Gamma}$ and $\Omega\notin [\rho,\rho +\omega^\alpha)$ then $\provx{\mathcal{H}}{\varphi\alpha\beta}{\rho}{\Gamma}$
}\end{theorem}
\begin{proof}
The proof is by main induction on $\alpha$ and subsidiary induction on $\beta$. If $\Gamma$ is an axiom then the result is immediate. If the last inference was anything other that (Cut) we may apply the subsidiary induction hypothesis to its premises and then the same inference again. The crucial case is where the last inference was (Cut), so suppose there is a formula $C$ with $rk(C)<\rho+\omega^\alpha$ such that
\begin{align*}
\tag{1}&\provx{\mathcal{H}}{\beta_0}{\rho+\omega^\alpha}{\Gamma,C}\quad\text{with $\beta_0<\beta$.}\\
\tag{2}&\provx{\mathcal{H}}{\beta_0}{\rho+\omega^\alpha}{\Gamma,\neg C}\quad\text{with $\beta_0<\beta$.}
\end{align*}
Applying the subsidiary induction hypothesis to (1) and (2) yields
\begin{align*}
\tag{3}&\provx{\mathcal{H}}{\varphi\al\beta_0}{\rho}{\Gamma,C}.\\
\tag{4}&\provx{\mathcal{H}}{\varphi\al\beta_0}{\rho}{\Gamma,\neg C}.
\end{align*}
Case 1. If $rk(C)<\rho$ then we may apply (Cut) to (3) and (4), noting that $\varphi\alpha\beta_0+1<\varphi\alpha\beta\in\mathcal{H}$, to give the desired result.\\

\noindent Case 2. Now suppose $rk(C)\in[\rho,\rho+\omega^\alpha)$, so we may write $rk(C)$ in the following form:
\begin{equation*}
\tag{5}rk(C)=\rho+\omega^{\alpha_1}+\ldots+\omega^{\alpha_n}\quad\text{with $\alpha>\alpha_1\geq \ldots\geq\alpha_n$.}
\end{equation*}
 If $n=0$, this means  that $rk(C)=\rho$. From (3) we know that $k(C)\subseteq\mathcal{H}$ and thus $rk(C)\in\mathcal{H}$. Now (5) and (H2) and (H3) from Definition \ref{operator} give us $\alpha_1,\ldots,\alpha_n\in\mathcal{H}$. Since $rk(C)\neq\Omega$ we may apply the Reduction Lemma \ref{xreduction} to (3) and (4) to obtain
\begin{equation*}
\tag{6}\provx{\mathcal{H}}{\varphi\alpha\beta_0+\varphi\alpha\beta_0}{\rho+\omega^{\alpha_1}+\ldots+\omega^{\alpha_n}}{\Gamma}.
\end{equation*}
Now $\varphi\alpha\beta_0+\varphi\alpha\beta_0<\varphi\alpha\beta$, so by Lemma \ref{xinversion}i) we have
\begin{equation*}
\tag{7}\provx{\mathcal{H}}{\varphi\alpha\beta}{\rho+\omega^{\alpha_1}+\ldots+\omega^{\alpha_n}}{\Gamma}.
\end{equation*}
Applying the main induction hypothesis (since $\alpha_n<\alpha$) to (7) gives
\begin{equation*}
\provx{\mathcal{H}}{\varphi\alpha_n(\varphi\alpha\beta)}{\rho+\omega^{\alpha_1}+\ldots+\omega^{\alpha_{n-1}}}{\Gamma}.
\end{equation*}
But since $\varphi\alpha\beta$ is a fixed point of the function $\varphi\alpha_n (\cdot)$ we have
\begin{equation*}
\provx{\mathcal{H}}{\varphi\alpha\beta}{\rho+\omega^{\alpha_1}+\ldots+\omega^{\alpha_{n-1}}}{\Gamma}.
\end{equation*}
Now since $\alpha_1,\ldots,\alpha_{n-1}<\alpha$ we may repeat this application of the main induction hypothesis a further $n-1$ times to obtain
\begin{equation*}
\provx{\mathcal{H}}{\varphi\alpha\beta}{\rho}{\Gamma}
\end{equation*}
as required.
\end{proof}
\begin{lemma}[Boundedness for $\RSX$]\label{xboundedness}{\em
If $C$ is a $\Sigma$ formula, $\alpha\leq\beta<\Omega$, $\beta\in\mathcal{H}$ and $\provx{\mathcal{H}}{\alpha}{\rho}{\Gamma,C}$ then $\provx{\mathcal{H}}{\alpha}{\rho}{\Gamma,C^{\mathbb{L}_\beta (X)}}$.
}\end{lemma}
\begin{proof}
The proof is by induction on $\alpha$. If $C$ is basic then $C\equiv C^{\mathbb{L}_\beta (X)}$ so there is nothing to show. If $C$ was not the principal formula of the last inference then we may apply the induction hypothesis to its premises and then the same inference again. Now suppose $C$ was the principal formula of the last inference. The last inference cannot have been $\SRX$ since $\al<\OO$.\\

\noindent Case 1. Suppose $C\simeq\bigwedge(C_i)_{i\in y}$ and $\provx{\mathcal{H}[i]}{\alpha_i}{\rho}{\Gamma,C,C_i}$ with $\alpha_i<\alpha$. Since $C$ is a $\Sigma$ formula, there must be some $\eta\in\CH(\emptyset)\cap\Omega$ such that $(\forall s\in y)(\lev s<\eta)$. Therefore $C^{\mathbb{L}_\beta (X)}\simeq\bigwedge(C_i^{\mathbb{L}_\beta (X)})_{i\in y}$. Now two applications of the induction hypothesis gives
\begin{equation*}
\provx{\mathcal{H}[i]}{\alpha_i}{\rho}{\Gamma,C^{\mathbb{L}_\beta (X)}, C_i^{\mathbb{L}_\beta (X)}}
\end{equation*}
to which we may apply the appropriate inference to gain the desired result.\\

\noindent Case 2. Now suppose $C\simeq\bigvee(C_i)_{i\in y}$ and $\provx{\mathcal{H}}{\alpha_0}{\rho}{\Gamma,C,C_{i_0}}$, with $i_0\in y$, $\lev{i_0}<\Gamma_{\theta+1}+\alpha$ and $\alpha_0<\alpha$. In this case $C^{\mathbb{L}_\beta (X)}\simeq\bigvee(C_i)_{i\in y^\prime}$ where either $y^\prime= y$ or $y^\prime=\{i\in y\:|\:\lev i<\Gamma_{\theta+1}+\beta\}$. Now by assumption $\lev{i_0}<\Gamma_{\theta+1}+\alpha<\Gamma_{\theta+1}+\beta$, so $i_0\in y^\prime$. Thus using the same inference again, or $(b\exists)$ in the case that the last inference was $(\exists)$, we obtain
\begin{equation*}
\provx{\mathcal{H}}{\alpha}{\rho}{\Gamma, C^{\mathbb{L}_\beta (X)}}
\end{equation*}
as required.
\end{proof}
\begin{definition}{\em For each $\eta\in T(\theta)$ we define
\begin{align*}
\CH_\eta&:\mathcal{P}(\textbf{ON})\mapsto\mathcal{P}(\textbf{ON})\\
\mathcal{H}_\eta (Y)&:=\bigcap\{B(\alpha)\:|\:Y\subseteq B(\alpha)\:\text{and}\:\eta<\alpha\}
\end{align*}
}\end{definition}
\begin{lemma}{\em
For any $\eta$, $\mathcal{H}_\eta$ is an operator.
}\end{lemma}
\begin{proof}
We must verify the conditions (H1) - (H5) from Definition \ref{operator}.\\

\noindent (H1) Clearly $0\in\mathcal{H}_\eta (Y)$ and $\{\Gamma_\beta\:|\:\beta\leq\theta\}\subseteq\mathcal{H}_\eta (Y)$ since these belong in any of the sets $B(\alpha)$. It remains to note that $\mathcal{H}_\eta (Y)\supseteq B(1)$ and since $\Gamma_{\theta+1}=\psi 0\in B(1)$ we have $\Gamma_{\theta+1}\in\mathcal{H}_\eta (Y)$.\\

\noindent (H2) and (H3) follow immediately from Lemma \ref{!NF}i) and ii) respectively.\\

\noindent (H4) is clear from the definition. Now for (H5) suppose $Y^\prime\subseteq\mathcal{H}_\eta (Y)$, then $Y^\prime\subseteq B(\alpha)$ for every $\alpha$ such that $\eta<\alpha$ and $Y\subseteq B(\alpha)$. It follows that $\mathcal{H}_\eta(Y^\prime)\subseteq\mathcal{H}_\eta(Y)$.
\end{proof}
\begin{lemma}\label{Heta1}{\em
\begin{description}
\item[i)] $\mathcal{H}_\eta (Y)$ is closed under $\varphi$ and $\psi|_{\eta+1}$.
\item[ii)] If $\delta<\eta$ then $\mathcal{H}_\delta (Y)\subseteq\mathcal{H}_\eta (Y)$
\item[iii)] If $\delta<\eta$ and $\provx{\mathcal{H}_\delta}{\alpha}{\rho}{\Gamma}$ then $\provx{\mathcal{H}_\eta}{\alpha}{\rho}{\Gamma}$
\end{description}
}\end{lemma}
\begin{proof}
i) Note that for any $X$, $\CH_\eta(X)=B(\al)$ for some $\al\geq\eta+1$. \\

\noindent ii) follows immediately from the definition of $\mathcal{H}_\eta$ and iii) follows easily from ii).
\end{proof}
\begin{lemma}\label{Heta2}{\em
Suppose $\eta\in B(\eta)$ and for any ordinal $\beta$ let $\hat{\beta}:=\eta+\omega^{\Omega+\beta}$.
\begin{description}
\item[i)] If $\alpha\in\mathcal{H}_\eta$ then $\hat{\alpha},\psi\hat{\alpha}\in\mathcal{H}_{\hat{\alpha}}$
\item[ii)] If $\alpha_0\in\mathcal{H}_\eta$ and $\alpha_0<\alpha$ then $\psi\hat{\alpha_0}<\psi\hat{\alpha}$
\end{description}
}\end{lemma}
\begin{proof}
i) First note that $\mathcal{H}_\eta (\emptyset)=B(\eta+1)$. Now from $\alpha,\eta\in B(\eta+1)$ we get $\hat\alpha\in B(\eta+1)$ and thus $\hat\alpha\in B(\hat\alpha)$. It follows that $\psi\hat\alpha\in B(\hat\alpha +1)=\mathcal{H}_{\hat\alpha}(\emptyset)$.\\

\noindent ii) Suppose that $\alpha_0\in\mathcal{H}_\eta$ and $\alpha_0<\alpha$, using the preceding argument we get that $\psi\hat{\alpha_0}\in B(\hat{\alpha_0}+1)\subseteq B(\hat\alpha)$, thus $\psi\hat{\alpha_0}<\psi\hat\alpha$.
\end{proof}
\begin{theorem}[Collapsing for $\RSX$]\label{ximpredce}{\em Suppose $\Gamma$ is a set of $\Sigma$ formulae and $\eta\in B(\eta)$.
\begin{equation*}
\text{If}\quad\provx{\mathcal{H}_\eta}{\alpha}{\Omega+1}{\Gamma}\quad\text{then}\quad\provx{\mathcal{H}_{\hat\alpha}}{\psi\hat\alpha}{\psi\hat\alpha}{\Gamma}
\end{equation*}
}\end{theorem}
\begin{proof}
We proceed by induction on $\alpha$. First note that from $\alpha\in\mathcal{H}_\eta$ we get $\hat\alpha,\psi\hat\alpha\in\mathcal{H}_{\hat\alpha}$ from Lemma \ref{Heta2}i).\\

\noindent If $\Gamma$ is an axiom then the result follows by Lemma \ref{xinversion}i). So suppose $\Gamma$ arose as the result of an inference, we shall distinguish cases according to the last inference of $\provx{\mathcal{H}_\eta}{\alpha}{\Omega+1}{\Gamma}$.\\

\noindent Case 1. Suppose $A\simeq\bigwedge(A_i)_{i\in y}\in\Gamma$ and $\provx{\mathcal{H}_\eta[i]}{\alpha_i}{\Omega+1}{\Gamma,A_i}$ with $\alpha_i<\alpha$ for each $i\in y$. Since $A$ is a $\Sigma$ formula, we must have $\text{sup}\{\lev i\:|\:i\in y\}<\Omega$, therefore as $k(A)\subseteq\mathcal{H}_\eta=B(\eta+1)$ we must have $\text{sup}\{\lev i\:|\:i\in y\}<\psi(\eta+1)$. It follows that for any $i\in y$ $\lev i\in\mathcal{H}_\eta$ and thus $\mathcal{H}_\eta[i]=\mathcal{H}_\eta$. This means that we may use the induction hypothesis to give
\begin{equation*}
\provx{\mathcal{H}_{\hat{\alpha_i}}}{\psi\hat{\alpha_i}}{\psi\hat{\alpha_i}}{\Gamma,A_i}\quad\text{for all $i\in y$}.
\end{equation*}
Now applying Lemma \ref{Heta1}ii) we get
\begin{equation*}
\provx{\mathcal{H}_{\hat{\alpha}}}{\psi\hat{\alpha_i}}{\psi\hat{\alpha_i}}{\Gamma,A_i}\quad\text{for all $i\in y$}.
\end{equation*}
Upon noting that $\psi\hat{\alpha_i}<\psi\hat\alpha$ by \ref{Heta2}ii) we may apply the appropriate inference to obtain
\begin{equation*}
\provx{\mathcal{H}_{\hat\alpha}}{\psi\hat\alpha}{\psi\hat\alpha}{\Gamma}.
\end{equation*}
Case 2. Now suppose that $A\simeq\bigvee(A_i)_{i\in y}\in\Gamma$ and $\provx{\mathcal{H}_\eta}{\alpha_0}{\Omega+1}{\Gamma,A_{i_0}}$ with $i_0\in y$, $|i_0|\in\mathcal{H}_\eta$ and $\alpha_0<\alpha$. We may immediately apply the induction hypothesis to obtain
\begin{equation*}
\provx{\mathcal{H}_{\hat\alpha}}{\psi\hat{\alpha_0}}{\psi\hat{\alpha_0}}{\Gamma,A_{i_0}}
\end{equation*}
Now we want to be able to apply the appropriate inference to derive $\Gamma$ but first we must check that $\lev{i_0}<\Gamma_{\theta+1}+\psi\hat\alpha$. Since $\lev{i_0}\in\mathcal{H}_\eta=B(\eta+1)$ we have
\begin{equation*}
\lev{i_0}<\psi(\eta+1)<\psi\hat\alpha\leq\Gamma_{\theta+1}+\psi\hat\alpha.
\end{equation*}
Therefore we may apply the appropriate inference to yield
\begin{equation*}
\provx{\mathcal{H}_{\hat\alpha}}{\psi\hat\alpha}{\psi\hat\alpha}{\Gamma}.
\end{equation*}
Case 3. Now suppose the last inference was $\SRX$ so we have $\exists zF^z\in\Gamma$ and $\provx{\mathcal{H}_\eta}{\alpha_0}{\Omega+1}{\Gamma,F}$ with $\alpha_0<\alpha$ and $F$ a $\Sigma$ formula. Applying the induction hypothesis we have
\begin{equation*}
\provx{\mathcal{H}_{\hat\alpha}}{\psi\hat{\alpha_0}}{\psi\hat{\alpha_0}}{\Gamma,F}.
\end{equation*}
Applying Boundedness \ref{xboundedness} we obtain
\begin{equation*}
\provx{\mathcal{H}_{\hat\alpha}}{\psi\hat{\alpha_0}}{\psi\hat{\alpha_0}}{\Gamma,F^{\mathbb{L}_{\psi\hat{\al_0}}(X)}}.
\end{equation*}
Now by Lemma \ref{Heta2} $\lev{\mathbb{L}_{\psi\hat{\al_0}}(X)}=\Gamma_{\theta+1}+\psi\hat{\al_0}<\Gamma_{\theta+1}+\psi\hat\al$, so we may apply $(\exists)$ to obtain
\begin{equation*}
\provx{\CH_{\hat\al}}{\psi\hat\al}{\psi\hat\al}{\Gamma,\exists zF^z}
\end{equation*}
as required.\\

\noindent Case 4. Finally suppose the last inference was $\Cut$, so for some $A$ with $rk(A)\leq\OO$ we have
\begin{align*}
\tag{1}&\provx{\CH_\eta}{\al_0}{\OO+1}{\Gamma, A}\quad\text{with $\al_0<\al$.}\\
\tag{2}&\provx{\CH_\eta}{\al_0}{\OO+1}{\Gamma, \neg A}\quad\text{with $\al_0<\al$.}
\end{align*}

\noindent 4.1 If $rk(A)<\Omega$ then $A$ is $\Delta_0$. In this case both $A$ and $\neg A$ are $\Sigma$ formulae so we may immediately apply the induction hypothesis to both (1) and (2) giving
\begin{align*}
\tag{3}&\provx{\CH_{\hat{\al_0}}}{\psi\hat{\alpha_0}}{\psi\hat{\alpha_0}}{\Gamma, A}\\
\tag{4}&\provx{\CH_{\hat{\al_0}}}{\psi\hat{\alpha_0}}{\psi\hat{\alpha_0}}{\Gamma, \neg A}.
\end{align*}
Since $k(A)\subseteq\mathcal{H}_\eta (\emptyset)=B(\eta+1)$ and $A$ is $\Delta_0$ it follows from Observation \ref{obs} that $rk(A)\in B(\eta+1)\cap\OO$. Thus $rk(A)<\psi(\eta+1)<\psi\hat\alpha$, so we may apply $\Cut$ to complete this case.\\

\noindent 4.2 Finally suppose $rk(A)=\Omega$. Without loss of generality we may assume that $A\equiv\exists zF(z)$ with $F$ a $\Delta_0$ formula. We may immediately apply the induction hypothesis to (1) giving
\begin{equation*}
\tag{5}\provx{\mathcal{H}_{\hat{\alpha_0}}}{\psi\hat{\alpha_0}}{\psi\hat{\alpha_0}}{\Gamma,A}.
\end{equation*}
Applying Boundedness \ref{xboundedness} to (5) yields
\begin{equation*}\tag{6}
\provx{\mathcal{H}_{\hat{\alpha_0}}}{\psi\hat{\alpha_0}}{\psi\hat{\alpha_0}}{\Gamma,A^{\mathbb{L}_{\psi\hat{\alpha_0}}(X)}}.
\end{equation*}
Now using Lemma \ref{xinversion}v) on (2) yields
\begin{equation*}\tag{7}
\provx{\mathcal{H}_{\hat{\alpha_0}}}{\alpha_0}{\Omega+1}{\Gamma,\neg A^{\mathbb{L}_{\psi\hat{\alpha_0}}(X)}}.
\end{equation*}
Observe that since $\eta,\alpha_0\in\mathcal{H}_\eta$ we have $\hat{\alpha_0}\in B(\eta+1)\subseteq B(\hat{\alpha_0})$. So since $\Gamma,\neg A^{\mathbb{L}_{\psi\hat{\alpha_0}}(X)}$ is a set of $\Sigma$-formulae we may apply the induction hypothesis to (7) giving
\begin{equation*}\tag{8}
\provx{\mathcal{H}_{\alpha_1}}{\psi\alpha_1}{\psi\alpha_1}{\Gamma,\neg A^{\mathbb{L}_{\psi\hat{\alpha_0}}}}\quad\text{where }\alpha_1:=\hat{\alpha_0}+\omega^{\Omega+\alpha_0}.
\end{equation*}
Now
\begin{equation*}
\alpha_1=\hat{\alpha_0}+\omega^{\Omega+\alpha_0}=\eta+\omega^{\Omega+\alpha_0}+\omega^{\Omega+\alpha_0}<\eta+\omega^{\Omega+\alpha}:=\hat\alpha.
\end{equation*}
Owing to Lemma \ref{Heta2}ii) we have $\psi\hat{\alpha_0},\psi\alpha_1<\psi\hat\alpha$, thus we may apply (Cut) to (6) and (8) giving
\begin{equation*}
\provx{\mathcal{H}_{\hat\alpha}}{\psi\hat\alpha}{\psi\hat\alpha}{\Gamma}
\end{equation*}
as required.
\end{proof}
\section{Embedding \textbf{KP} into $\RSX$}
\begin{definition}{\em
\begin{description}
\item[i)] Given ordinals $\alpha_1,\ldots,\alpha_n$. The expression $\omega^{\alpha_1}\#\ldots\#\omega^{\alpha_n}$ denotes the ordinal $\omega^{\alpha_{p(1)}}+\ldots+\omega^{\alpha_{p(n)}}$, where $p:\{1,\ldots,n\}\mapsto \{1,\ldots,n\}$ such that $\alpha_{p(1)}\geq\ldots\geq\alpha_{p(n)}$. More generally $\alpha\#0:=0\#\alpha:=0$ and $\alpha\#\beta:=\omega^{\alpha_1}\#\ldots\#\omega^{\alpha_n}\#\omega^{\beta_1}\#\ldots\#\omega^{\beta_m}$ for $\alpha=_{NF}\omega^{\alpha_1}+\ldots+\omega^{\alpha_n}$ and $\beta=_{NF}\omega^{\beta_1}+\ldots+\omega^{\beta_m}$.

\item[ii)] If $A$ is any $\RSX$-formula then $no(A):=\omega^{rk(A)}$.

\item[iii)] If $\Gamma=\{A_1,\ldots,A_n\}$ is a set of $\RSX$-formulae then $no(\Gamma):=no(A_1)\#\ldots\#no(A_n)$.

\item[iv)] $\Vdash\Gamma$ will be used to abbreviate that
\begin{equation*}
\provx{\mathcal{H}[\Gamma]}{no(\Gamma)}{0}{\Gamma}\quad\text{holds for any operator $\mathcal{H}$}
\end{equation*}
\item[v)] $\Vdash^\al_\rho\Gamma$ will be used to abbreviate that
\begin{equation*}
\provx{\mathcal{H}[\Gamma]}{no(\Gamma)\#\al}{\rho}{\Gamma}\quad\text{holds for any operator $\mathcal{H}$}
\end{equation*}
\noindent As might be expected $\Vdash^\al\GA$ and $\Vdash_\rho\GA$ stand for $\Vdash^\al_0\GA$ and $\Vdash^0_\rho\GA$ respectively.
\end{description}
}\end{definition}
\noindent The following lemma shows that under certain conditions we may use $\Vdash$ as a calculus.
\begin{lemma}\label{xcalc}{\em i) If $\GA$ follows from premises $\GA_i$ by an $\RSX$ inference other than (Cut) or $\SRX$ and without contractions then
\begin{equation*}
\text{if}\;\Vdash^\al_\rho\GA_i\quad\text{then}\quad\Vdash^\al_\rho\GA
\end{equation*}
ii) If $\Vdash^\al_\rho \GA,A,B$ then $\Vdash^\al_\rho\GA,A\vee B$.
}\end{lemma}
\begin{proof}
Part i) follows from Lemma \ref{ranklemma}. It also needs to be noted that if the last inference was universal with premises $\{\GA_i\}_{i\in Y}$, then $\CH[\GA_i]\subseteq\CH[i]$.\\

\noindent For part ii) suppose $\Vdash^\al_\rho \GA,A,B$, so we have
\begin{equation*}
\provx{\mathcal{H}[\Gamma]}{no(\Gamma,A,B)\#\al}{\rho}{\Gamma,A,B}.
\end{equation*}
Two applications of $(\vee)$ and a contraction yields
\begin{equation*}
\provx{\mathcal{H}[\Gamma]}{no(\Gamma,A,B)\#\al+2}{\rho}{\Gamma,A\vee B}.
\end{equation*}
It remains to note that since $\omega^{rk(A\vee B)}$ is additive principal, Lemma \ref{ranklemma} gives us
\begin{equation*}
no(\Gamma,A,B)\#\al+2=no(\Gamma)\#\al\#\omega^{rk(A)}\#\omega^{rk(B)}+2<no(\Gamma)\#\al\#\omega^{rk(A\vee B)}=no(\Gamma,A\vee B)\#\al.
\end{equation*}
So we may complete the proof with an application of Lemma \ref{xinversion}i).
\end{proof}
\begin{lemma}\label{xsetuplemma}{\em
Let $A$ be an $\RSX$ formula and $s,t$ be $\RSX$ terms.
\begin{description}
\item{i)} $\Vdash A,\neg A$
\item{ii)} $\Vdash s\notin s$
\item{iii)} $\Vdash s\subseteq s$ where $s\subseteq s:\equiv (\forall x\in s)(x\in s)$
\item{iv)} If $\lev s<\lev t$ then $\Vdash s\dotin t\rightarrow s\in t$ and $\Vdash \neg(s\dotin t), s\in t$
\item{v)} $\Vdash s\neq t,t=s$
\item{vi)} If $\lev s<\lev t$ and $\Vdash \Gamma, A, B$ then $\Vdash\Gamma,s\dotin t\rightarrow A, s\dotin t\wedge B$
\item{vii)} If $\lev s<\Gamma_{\theta+1}+\al$ then $\Vdash s\in\mathbb{L}_\al(X)$
\end{description}
}\end{lemma}
\begin{proof}
i) We use induction of $rk(A)$, and split into cases based upon the form of $A$.:\\

\noindent Case 1. Suppose $A\equiv \bar{u}\in\bar{v}$. In this case either $A$ or $\neg A$ is an axiom so there is nothing to show.\\

\noindent Case 2. Suppose $A\equiv r\in t$ where $\text{max}(\lev r,\lev t)\geq\Gamma_{\theta+1}$. By Lemma \ref{ranklemma} and the induction hypothesis we have $\Vdash s\dotin t\wedge r=s,s\dotin t\rightarrow r\neq s$ for all $\lev s<\lev t$. Thus we have the following template for derivations in $\RSX$:

\begin{prooftree}
\Axiom$\fCenter\Vdash s\dotin t\wedge r=s,s\dotin t\rightarrow r\neq s$
\LeftLabel{$(\in)$}
\UnaryInf$\fCenter\Vdash r\in t,s\dotin t\rightarrow r\neq s$
\LeftLabel{$(\notin)$}
\UnaryInf$\fCenter\Vdash r\in t,r\notin t$
\end{prooftree}

\noindent Case 3. Suppose $A\equiv (\exists x\in t)F(x)$. By Lemma \ref{ranklemma} and the induction hypothesis we have $\Vdash s\dotin t\wedge F(s), s\dotin t\rightarrow\neg F(s)$ for all $\lev s<\lev t$. We have the following template for derivations in $\RSX$:

\begin{prooftree}
\Axiom$\fCenter\Vdash s\dotin t\wedge F(s), s\dotin t\rightarrow\neg F(s)\quad\text{for all $\lev s<\lev t$}$
\LeftLabel{$(b\exists)$}
\UnaryInf$\fCenter\Vdash (\exists x\in t)F(x),s\dotin t\rightarrow\neg F(s)$
\LeftLabel{$(b\forall)$}
\UnaryInf$\fCenter\Vdash (\exists x\in t)F(x), (\forall x\in t)\neg F(x)$
\end{prooftree}

\noindent Case 4. $A\equiv A_0\vee A_1$. We have the following template for derivations in $\RSX$:
\begin{prooftree}
\Axiom$\fCenter\Vdash A_0,\neg A_0$
\LeftLabel{$(\vee)$}
\UnaryInf$\fCenter\Vdash A_0\vee A_1,\neg A_0$
\LeftLabel{$(\vee)$}
\Axiom$\fCenter\Vdash A_1,\neg A_1$
\UnaryInf$\fCenter\Vdash A_0\vee A_1,\neg A_1$
\LeftLabel{$(\wedge)$}
\BinaryInf$\fCenter\Vdash A_0\vee A_1, \neg A_0\wedge\neg A_1$
\end{prooftree}
All other cases may be seen as variations of those above.\\

\noindent ii) We proceed by induction on $rk(s)$. If $s$ is of the form $\bar{u}$ then $s\notin s$ is already an axiom. Inductively we have $\Vdash r\notin r$ for all $\lev r<\lev s$. Now suppose $s$ is of the form $\mathbb{L}_\alpha (X)$, in this case $r\notin r\equiv r\dotin s\wedge r\notin r$ so we have the following template for derivations in $\RSX$:

\begin{prooftree}
\Axiom$\fCenter\Vdash r\dotin s\wedge r\notin r$
\LeftLabel{$(b\exists)$}
\UnaryInf$\fCenter\Vdash (\exists x\in s)(x\notin r)$
\LeftLabel{$(\vee)$}
\UnaryInf$\fCenter\Vdash s\neq r$
\LeftLabel{\ref{abbreviations}ii)}
\UnaryInf$\fCenter\Vdash r\dotin s\rightarrow s\neq r$
\LeftLabel{$(\notin)$}
\UnaryInf$\fCenter\Vdash s\notin s$
\end{prooftree}
\noindent Now suppose $s$ is of the form $[x\in\mathbb{L}_\alpha(X)\:|\:B(x)]$, by i) we have $\Vdash B(r),\neg B(r)$ for any $\lev r<\lev s$. We have the following template for derivations in $\RSX$:

\begin{prooftree}
\Axiom$\fCenter\Vdash r\notin r\quad\Vdash B(r),\neg B(r)\quad\text{for any $\lev r<\lev s$}$
\LeftLabel{$(\wedge)$}
\UnaryInf$\fCenter\Vdash B(r)\wedge r\notin r,\neg B(r)$
\LeftLabel{$(b\exists)$}
\UnaryInf$\fCenter\Vdash (\exists x\in s)(x\notin r),\neg B(r)$
\LeftLabel{$(\vee)$}
\UnaryInf$\fCenter\Vdash s\neq r,\neg B(r)$
\LeftLabel{Lemma \ref{xcalc}ii)}
\UnaryInf$\fCenter\Vdash B(r)\rightarrow s\neq r$
\LeftLabel{$(\notin)$}
\UnaryInf$\fCenter\Vdash s\notin s$
\end{prooftree}
\noindent iii) Again we proceed by induction on $rk(s)$. If $s\equiv\bar{u}$ then $\Vdash \bar{v}\notin\bar{u},\bar{v}\in\bar{u}$ for any $\lev{\bar{v}}<\lev{\bar{u}}$ by part i), so we have the following template for derivations in $\RSX$:
\begin{prooftree}
\Axiom$\fCenter\Vdash\bar{v}\notin\bar{u},\bar{v}\in\bar{u}$
\LeftLabel{Lemma \ref{xcalc}ii)}
\UnaryInf$\fCenter\Vdash\bar{v}\in\bar{u}\rightarrow\bar{v}\in\bar{u}$
\LeftLabel{$(b\forall)$}
\UnaryInf$\fCenter\Vdash(\forall x \in s)(x\in s)$
\end{prooftree}
\noindent Suppose $s\equiv\mathbb{L}_\alpha (X)$, by the induction hypothesis we have $\Vdash r\subseteq r$ for all $\lev r<\lev s$. We have the following template for derivations in $\RSX$:
\begin{prooftree}
\Axiom$\fCenter\Vdash r\subseteq r$
\Axiom$\fCenter\Vdash r\subseteq r$
\LeftLabel{$(\wedge)$}
\BinaryInf$\fCenter\Vdash r=r$
\LeftLabel{$(\in)$}
\UnaryInf$\fCenter\Vdash r\in s$
\LeftLabel{\ref{abbreviations}ii)}
\UnaryInf$\fCenter\Vdash r\dotin s\rightarrow r\in s$
\LeftLabel{$(b\forall)$}
\UnaryInf$\fCenter\Vdash (\forall x\in s)(x\in s)$
\end{prooftree}
\noindent Finally suppose $s\equiv [x\in\mathbb{L}_\alpha (X)\:|\:B(x)]$, again by the induction hypothesis we have $\Vdash r\subseteq r$ for all $\lev r<\lev s$. Also by part i) we have $\Vdash \neg B(r),B(r)$ for all such $r$. We have the following template for derivations in $\RSX$:
\begin{prooftree}
\Axiom$\fCenter\Vdash\neg B(r),r\subseteq r$
\LeftLabel{$(\wedge)$}
\UnaryInf$\fCenter\Vdash\neg B(r),r=r$
\Axiom$\fCenter\Vdash\neg B(r),B(r)$
\LeftLabel{$(\wedge)$}
\BinaryInf$\fCenter\Vdash\neg B(r), B(r)\wedge r=r$
\LeftLabel{$(\in)$}
\UnaryInf$\fCenter\Vdash\neg B(r),r\in s$
\LeftLabel{Lemma \ref{xcalc}ii)}
\UnaryInf$\fCenter\Vdash B(r)\rightarrow r\in s$
\LeftLabel{$(b\forall)$}
\UnaryInf$\fCenter\Vdash (\forall x\in s)(x\in s)$
\end{prooftree}
\noindent iv) Was shown whilst proving iii).\\

\noindent v) By part i) we have $\Vdash\neg(s\subseteq t),s\subseteq t$ and $\Vdash\neg(t\subseteq s),t\subseteq s$ for all $\lev s<\lev t$. We have the following template for derivations in $\RSX$.
\begin{prooftree}
\Axiom$\fCenter\Vdash\neg(s\subseteq t),s\subseteq t$
\LeftLabel{$(\vee)$}
\UnaryInf$\fCenter\Vdash\neg(s\subseteq t)\vee\neg(t\subseteq s),s\subseteq t$
\Axiom$\fCenter\Vdash\neg(t\subseteq s),t\subseteq s$
\LeftLabel{$(\vee)$}
\UnaryInf$\fCenter\Vdash\neg(t\subseteq s)\vee\neg(s\subseteq t),t\subseteq s$
\LeftLabel{$(\wedge)$}
\BinaryInf$\fCenter\Vdash\neg (s\subseteq t)\vee\neg (t\subseteq s), s\subseteq t\wedge t\subseteq s$
\LeftLabel{\ref{abbreviations}i)}
\UnaryInf$\fCenter\Vdash s\neq t,t=s$
\end{prooftree}
\noindent vi) If $t\equiv\mathbb{L}_\alpha (X)$ then this result is trivial since $s\dotin t\rightarrow A :=A$ and $s\dotin t\wedge B:=B$.\\

\noindent Now if $t\equiv \bar{u}$ then $s\dotin t:=s\in t$ and if $t\equiv[x\in\mathbb{L}_\alpha (X)\:|\:C(x)]$ then $s\dotin t:= C(s)$. In either case we have the following template for derivations in $\RSX$:
\begin{prooftree}
\Axiom$\fCenter\Vdash\Gamma,A,B$
\LeftLabel{$(\vee)$}
\UnaryInf$\fCenter\Vdash\Gamma,s\dotin t\rightarrow A,B$
\Axiom$\fCenter\Vdash\Gamma,\neg (s\dotin t),s\dotin t\quad\text{by i)}$
\LeftLabel{$(\vee)$}
\UnaryInf$\fCenter\Vdash\Gamma,s\dotin t\rightarrow A,s\dotin t$
\LeftLabel{$(\wedge)$}
\BinaryInf$\fCenter\Vdash\Gamma, s\dotin t\rightarrow A, s\dotin t\wedge B$
\end{prooftree}
vii)  By part iii) we have $\Vdash s=s$ for all $\lev s<\Gamma_{\theta+1}+\al$ which means we have $\Vdash s\dotin\mathbb{L}_\al (X) \wedge s=s$ for all such $s$. From which one application of $(\in)$ gives the desired result.
\end{proof}

\begin{lemma}[Extensionality]\label{xextensionality}{\em
For any $\RSX$ formula $A(s_1,\ldots,s_n)$,
\begin{equation*}
\Vdash[s_1\neq t_1],\ldots,[s_n\neq t_n],\neg A(s_1,\ldots,s_n), A(t_1,\ldots,t_n).
\end{equation*}
Where $[s_i\neq t_i]:=\neg(s_i\subseteq t_i),\neg(t_i\subseteq s_i)$.
}\end{lemma}
\begin{proof}
The proof is by induction on $rk(A(s_1,\ldots,s_n))\# rk(A(t_1,\ldots,t_n))$.\\

\noindent Case 1. Suppose $A(s_1,s_2)\equiv s_1\in s_2$. By the induction hypothesis we have $\Vdash[s_1\neq t_1],[s\neq t],s_1\neq s,t_1=t$ for all $\lev s<\lev{s_2}$ and all $\lev t<\lev{t_2}$. What follows is a template for derivations in $\RSX$, for ease of reading the principal formula of each inference is underlined (some lines do not necessarily represent single inferences, but in these cases it is clear how to extend the concept of "principal formula" in a sensible way).
\begin{prooftree}
\Axiom$\fCenter\Vdash[s_1\neq t_1],[s\neq t],s_1\neq s,t_1=t$
\LeftLabel{$(\vee)$}
\UnaryInf$\fCenter\Vdash[s_1\neq t_1],\underline{s\neq t},s_1\neq s,t_1=t$
\LeftLabel{Lemma \ref{xsetuplemma} vi)}
\UnaryInf$\fCenter\Vdash[s_1\neq t_1],\underline{t\dotin t_2\rightarrow s\neq t},s_1\neq s,\underline{t\dotin t_2\wedge t_1=t}$
\LeftLabel{$(\in)$}
\UnaryInf$\fCenter\Vdash[s_1\neq t_1],t\dotin t_2\rightarrow s\neq t,s_1\neq s,\underline{t_1\in t_2}$
\LeftLabel{$(\notin)$}
\UnaryInf$\fCenter\Vdash[s_1\neq t_1],\underline{s\notin t_2},s_1\neq s,t_1\in t_2$
\LeftLabel{Lemma \ref{xsetuplemma} vi)}
\UnaryInf$\fCenter\Vdash[s_1\neq t_1],\underline{s\dotin s_2\wedge s\notin t_2},\underline{s\dotin s_2\rightarrow s_1\neq s},t_1\in t_2$
\LeftLabel{$(b\exists)$}
\UnaryInf$\fCenter\Vdash[s_1\neq t_1],\underline{(\exists x\in s_2)(x\notin t_2)},s\dotin s_2\rightarrow s_1\neq s,t_1\in t_2$
\LeftLabel{$(\notin)$}
\UnaryInf$\fCenter\Vdash[s_1\neq t_1],(\exists x\in s_2)(x\notin t_2),\underline{s_1\notin s_2},t_1\in t_2$
\LeftLabel{Lemma \ref{xinversion}i)}
\UnaryInf$\fCenter\Vdash[s_1\neq t_1],\underline{s_2\neq t_2},s_1\notin s_2,t_1\in t_2$
\end{prooftree}

\noindent Case 2. Suppose $A(s_1)\equiv s_1\in s_1$. In this case $\neg A(s_1)\equiv s_1\notin s_1$ so the result follows from Lemma \ref{xsetuplemma}ii).\\

\noindent Case 3. Suppose $A(s_1,\ldots,s_n)\equiv (\exists y\in s_i)(B(y,s_1,\ldots,s_n))$ for some $1\leq i\leq n$. Inductively we have
\begin{equation*}
\Vdash[s_1\neq t_1],\ldots,[s_n\neq t_n],\neg B(r,s_1,\ldots,s_n), B(r,t_1,\ldots,t_n)
\end{equation*}
for all $\lev r<\lev{s_i}$. Now by applying \ref{xsetuplemma}iv) we obtain
\begin{equation*}
\Vdash[s_1\neq t_1],\ldots,[s_n\neq t_n],r\dotin s_i\rightarrow\neg B(r,s_1,\ldots,s_n), r\dotin s_i\wedge B(r,t_1,\ldots,t_n)
\end{equation*}
To which we may apply $(b\exists)$ followed by $(b\forall)$ to arrive at the desired conclusion.\\

\noindent Case 4. Suppose $A(s_1,\ldots,s_n)\equiv (\exists x\in r)B(x,s_1,\ldots,s_n)$ for some $r$ not present in $s_1,\ldots,s_n$. From the induction hypothesis we have
\begin{equation*}
\Vdash [s_1\neq t_1],\ldots,[s_n\neq t_n],p\dotin r\rightarrow\neg B(p,s_1,\ldots,s_n),p\dotin r\wedge B(p,t_1,\ldots,t_n)\quad\text{for all $\lev p<\lev r$.}
\end{equation*}
Applying $(b\exists)$ followed by $(b\forall)$ gives us the desired result.\\

\noindent The cases where $A(s_1,\ldots,s_n)\equiv\exists xB(x,s_1,\ldots,s_n)$ or $A(s_1,\ldots,s_n)\equiv B\vee C$ may be treated in a similar manner to case 4. All other cases are dual to one of the ones considered above.
\end{proof}

\begin{lemma}[Set Induction]\label{xfoundation}{\em For any $\RSX$-formula $F$:
\begin{equation*}
\force{\omega^{rk(A)}}{}\forall x[(\forall y\in x)F(y)\rightarrow F(x)]\rightarrow\forall xF(x)
\end{equation*}
where $A:=\forall x[(\forall y\in x)F(y)\rightarrow F(x)]$.
}\end{lemma}
\begin{proof}
Claim:
\begin{equation*}
\tag{*}\provx{\mathcal{H}[A,s]}{\omega^{rk(A)}\#\omega^{\lev s+1}}{0}{\neg A,F(s)}\quad\text{for any term $s$.}
\end{equation*}
We begin by verifying (*) using induction on $\lev s$. From the induction hypothesis we know that
\begin{equation}
\tag{1}\provx{\mathcal{H}[A,t]}{\omega^{rk(A)}\#\omega^{\lev t+1}}{0}{\neg A,F(t)}\quad\text{for all $\lev t<\lev s$.}
\end{equation}
By applying $(\vee)$ if necessary to (1) we obtain
\begin{equation*}
\tag{2}\provx{\mathcal{H}[A,t,s]}{\omega^{rk(A)}\#\omega^{\lev t+1}+1}{0}{\neg A,t\dotin s\rightarrow F(t)}\quad\text{for all $\lev t<\lev s$.}
\end{equation*}
To which we may apply $(b\forall)$ yielding
\begin{equation*}
\tag{3}\provx{\mathcal{H}[A,s]}{\eta+2}{0}{\neg A,(\forall y\in s)F(y)}\quad\text{where $\eta:=\omega^{rk(A)}\#\omega^{\lev s}$}.
\end{equation*}
Observe that $no(\neg F(s),F(s))<\omega^{rk(A)}$, so by Lemma \ref{xsetuplemma}i) we have
\begin{equation*}
\tag{4}\provx{\mathcal{H}[A,s]}{\eta+2}{0}{\neg F(s), F(s)}.
\end{equation*}
Applying $(\wedge)$ to (3) and (4) yields
\begin{equation*}
\tag{5}\provx{\mathcal{H}[A,s]}{\eta+3}{0}{\neg A,(\forall y\in s)F(y)\wedge\neg F(s),F(s)}.
\end{equation*}
To which we may apply $(\exists)$ to otain
\begin{equation*}
\tag{6}\provx{\mathcal{H}[A,s]}{\eta+4}{0}{\neg A,\exists x[(\forall y\in x)F(y)\wedge\neg F(x)],F(s)}.
\end{equation*}
It remains to observe that $\neg A\equiv\exists x[(\forall y\in x)F(y)\wedge\neg F(x)]$ and that $\eta+4<\omega^{rk(A)}\#\omega^{\lev s+1}$, and hence we may apply Lemma \ref{xinversion}i) to provide
\begin{equation*}
\tag{7}\provx{\mathcal{H}[A,s]}{\omega^{rk(A)}\#\omega^{\lev s+1}}{0}{\neg A,F(s)}
\end{equation*}
so the claim is verified.\\

\noindent  Applying $(\forall)$ to (*) gives
\begin{equation*}
\provx{\mathcal{H}[A]}{\omega^{rk(A)}\#\OO}{0}{\neg A,\forall xF(x)}.
\end{equation*}
Now by two applications of $(\vee)$ we may conclude
\begin{equation*}
\provx{\mathcal{H}[A]}{\omega^{rk(A)}\#\OO+2}{0}{A\rightarrow\forall xF(x)}.
\end{equation*}
It remains to note that $no(A\rightarrow\forall xF(x))\geq\omega^{\OO+1}>\OO+2$, so we have
\begin{equation}
\force{\omega^{rk(A)}}{0}A\rightarrow(\forall x\in\mathbb{L}_\alpha (X))F(x)
\end{equation}
as required.
\end{proof}
\begin{lemma}[Infinity]\label{xinfinity}{\em
Suppose $\omega<\mu<\Omega$, then
\begin{equation*}
\Vdash (\exists x\in\mathbb{L}_\mu(X))[(\exists z\in x)(z\in x)\wedge(\forall y\in x)(\exists z\in x)(y\in z)].
\end{equation*}
}\end{lemma}
\begin{proof}
The following gives a template for derivations in $\RSX$, the idea is that $\mathbb{L}_\omega (X)$ serves as a witness inside $\mathbb{L}_\mu (X)$.
\begin{prooftree}
\Axiom$\fCenter\quad\text{Lemma \ref{xsetuplemma}vii)}$
\UnaryInf$\fCenter\Vdash s\in\mathbb{L}_k (X)\quad\text{for any $\lev s<\lev{\mathbb{L}_k(X)}$ and $k<\omega$.}$
\LeftLabel{\ref{abbreviations}ii)}
\UnaryInf$\fCenter\Vdash\mathbb{L}_k (X)\dotin\mathbb{L}_\omega (X)\wedge s\in\mathbb{L}_k (X)$
\LeftLabel{$(b\exists)$}
\UnaryInf$\fCenter\Vdash(\exists z\in\mathbb{L}_\omega (X))(s\in\mathbb{L}_k (X))$
\LeftLabel{\ref{abbreviations}ii)}
\UnaryInf$\fCenter\Vdash s\dotin\mathbb{L}_\omega (X)\rightarrow (\exists z\in\mathbb{L}_\omega (X))(s\in z)$
\LeftLabel{$(b\forall)$}
\UnaryInf$\fCenter\Vdash (\forall y\in\mathbb{L}_\omega(X))(\exists z\in\mathbb{L}_\omega (X))(y\in z)$
\Axiom$\fCenter\Vdash\mathbb{L}_0 (X)\in\mathbb{L}_\omega (X)$
\LeftLabel{\ref{abbreviations}ii)}
\UnaryInf$\fCenter\Vdash\mathbb{L}_0 (X)\dotin\mathbb{L}_\omega (X)\wedge\mathbb{L}_0 (X)\in\mathbb{L}_\omega (X)$
\LeftLabel{$(b\exists)$}
\UnaryInf$\fCenter\Vdash (\exists z\in\mathbb{L}_\omega (X))(z\in\mathbb{L}_\omega (X))$
\LeftLabel{$(\wedge)$}
\BinaryInf$\fCenter\Vdash (\forall y\in\mathbb{L}_\omega(X))(\exists z\in\mathbb{L}_\omega (X))(y\in z)\wedge (\exists z\in\mathbb{L}_\omega (X))(z\in\mathbb{L}_\omega (X))$
\LeftLabel{\ref{abbreviations}ii)}
\UnaryInf$\fCenter\Vdash \mathbb{L}_\omega(X)\dotin\mathbb{L}_\mu(X)\wedge [(\forall y\in\mathbb{L}_\omega(X))(\exists z\in\mathbb{L}_\omega (X))(y\in z)\wedge (\exists z\in\mathbb{L}_\omega (X))(z\in\mathbb{L}_\omega (X))]$
\LeftLabel{$(b\exists)$}
\UnaryInf$\fCenter\Vdash (\exists x\in\mathbb{L}_\mu(X))[(\exists z\in x)(z\in x)\wedge(\forall y\in x)(\exists z\in x)(y\in z)]$
\end{prooftree}
\end{proof}
\begin{lemma}[$\Delta_0$-Separation]\label{xseparation}{\em
Suppose $A(a,b_1,\ldots,b_n)$ be a $\Delta_0$-formula of $\KP$ with all free variables indicated, $\mu$ a limit ordinal and $\lev s,\lev{t_0},\ldots,\lev{t_n}<\Gamma_{\theta+1}+\mu$.
\begin{equation*}
\Vdash(\exists y\in\mathbb{L}_\mu (X))[(\forall x\in y)(x\in s\wedge A(x,t_1,\ldots,t_n))\wedge(\forall x\in s)(A(x,t_1,\ldots,t_n)\rightarrow x\in y)]
\end{equation*}
}\end{lemma}
\begin{proof}
Let $\alpha:=\text{max}\{\lev s,\lev{t_0},\ldots,\lev{t_n}\}+1$ and note that $\alpha<\Gamma_{\theta+1}+\mu$ since $\mu$ ia a limit. Now let $\beta$ be the unique ordinal such that $\alpha=\Gamma_{\theta+1}+\beta$ if such an ordinal exists, if not set $\beta:=0$. Now define
\begin{equation*}
t:=[z\in\mathbb{L}_\beta (X)\:|\:z\in s\wedge B(z)]
\end{equation*}
where $B(z):=A(z,t_1,\ldots,t_n)$. We have the following templates for derivations in $\RSX$:
\begin{prooftree}
\Axiom$\fCenter{\phantom{\Vdash}}\quad\quad\quad\text{Lemma \ref{xsetuplemma} i)}$
\UnaryInf$\fCenter\Vdash\neg(r\in s\wedge B(r)),r\in s\wedge B(r)\quad\text{for all $\lev r<\al$}$
\LeftLabel{Lemma \ref{xcalc}ii)}
\UnaryInf$\fCenter\Vdash(r\in s\wedge B(r))\rightarrow r\in s\wedge B(r)$
\LeftLabel{\ref{abbreviations}ii)}
\UnaryInf$\fCenter\Vdash r\dotin t\rightarrow r\in s\wedge B(r)$
\LeftLabel{$(b\forall)$}
\UnaryInf$\fCenter\Vdash(\forall x\in t)(x\in s\wedge B(r))$
\end{prooftree}
In the following derivation $r$ ranges over terms $\lev r<\lev s$.
\begin{prooftree}
\Axiom$\fCenter\phantom{\Vdash}\text{Lemma \ref{xsetuplemma} iv)}$
\UnaryInf$\fCenter\Vdash\neg(r\dotin s),r\in s$
\Axiom$\fCenter\phantom{\Vdash}\text{Lemma \ref{xsetuplemma} i)}$
\UnaryInf$\fCenter\Vdash\neg B(r),B(r)$
\LeftLabel{$(\wedge)$}
\BinaryInf$\fCenter\Vdash\neg(r\dotin s),\neg B(r),r\in s\wedge B(r)$
\Axiom$\fCenter\phantom{\Vdash}\text{Lemma \ref{xsetuplemma} iii)}$
\UnaryInf$\fCenter\Vdash r=r$
\LeftLabel{$(\wedge)$}
\BinaryInf$\fCenter\Vdash\neg(r\dotin s),\neg B(r),(r\in s\wedge B(r))\wedge r=r$
\LeftLabel{\ref{abbreviations}ii)}
\UnaryInf$\fCenter\Vdash\neg(r\dotin s),\neg B(r),r\dotin t\wedge r=r$
\LeftLabel{$(\in)$}
\UnaryInf$\fCenter\Vdash\neg(r\dotin s),\neg B(r), r\in t$
\LeftLabel{Lemma \ref{xcalc}ii)}
\UnaryInf$\fCenter\Vdash\neg(r\dotin s),(B(r)\rightarrow r\in t)$
\LeftLabel{Lemma \ref{xcalc}ii)}
\UnaryInf$\fCenter\Vdash r\dotin s\rightarrow (B(r)\rightarrow r\in t)$
\LeftLabel{$(b\forall)$}
\UnaryInf$\fCenter\Vdash (\forall x\in s)(B(x)\rightarrow x\in t)$
\end{prooftree}
Now applying $(\wedge)$ to the two preceding derivations and noting that $\lev t<\Gamma_{\theta+1}+\mu$ gives us
\begin{equation*}
\Vdash t\dotin\mathbb{L}_\mu (X)\wedge [(\forall x\in t)(x\in s\wedge B(r))\wedge (\forall x\in s)(B(x)\rightarrow x\in t)]
\end{equation*}
to which we may apply $(b\exists)$ to obtain
\begin{equation*}
\Vdash(\exists y\in\mathbb{L}_\mu (X))[(\forall x\in y)(x\in s\wedge B(x))\wedge(\forall x\in s)(B(x)\rightarrow x\in y)].
\end{equation*}
It should also be checked that
\begin{equation*}
t\in\mathcal{H}[(\exists y\in\mathbb{L}_\mu (X))[(\forall x\in y)(x\in s\wedge B(x))\wedge(\forall x\in s)(B(x)\rightarrow x\in y)]]
\end{equation*}
but this is the case since
\begin{equation*}
\lev s,\lev{t_0},\ldots,\lev{t_n}\in k((\exists y\in\mathbb{L}_\mu (X))[(\forall x\in y)(x\in s\wedge B(x))\wedge(\forall x\in s)(B(x)\rightarrow x\in y)])
\end{equation*}
and $\lev t=\text{max}\{\text{max}\{\lev s,\lev{t_0},\ldots,\lev{t_n}\}+1,\Gamma_{\theta+1}\}$.
\end{proof}
\begin{lemma}[Pair and Union]\label{xpairandunion}{\em Let $\mu$ be a limit ordinal and let $s,t$ be $\RSX$-terms such that $\lev s,\lev t<\Gamma_{\theta+1}+\mu$, then
\begin{description}
\item[i)] $\Vdash(\exists z\in\mathbb{L}_\mu (X))(s\in z\wedge t\in z)$
\item[ii)] $(\exists z\in\mathbb{L}_\mu (X))(\forall y\in s)(\forall x\in y)(x\in z)$
\end{description}
}\end{lemma}

\begin{proof}
Let $\alpha:=\text{max}\{\lev s,\lev t\}+1$, now let $\beta$ be the unique ordinal such that $\alpha=\Gamma_{\theta+1}+\beta$ if such an ordinal exists, otherwise set $\beta:=0$. Now by Lemma \ref{xsetuplemma}vii) we have
\begin{equation*}
\Vdash s\in\mathbb{L}_\beta (X)\quad\text{and}\quad\Vdash t\in\mathbb{L}_\beta (X).
\end{equation*}
Now by $(\wedge)$ and noticing that $\beta<\mu$ since $\mu$ is a limit, we have
\begin{equation*}
\Vdash\mathbb{L}_\beta (X)\dotin\mathbb{L}_\mu (X)\wedge (s\in\mathbb{L}_\beta (X)\wedge t\in\mathbb{L}_\beta (X)).
\end{equation*}
To which we may apply $(b\exists)$ to obtain the desired result.\\

\noindent ii) Let $\beta$ be the unique ordinal such that $\lev s=\Gamma_{\theta+1}+\beta$ if such an ordinal exists, otherwise let $\beta=0$. By Lemma \ref{xsetuplemma}vii) we have $\Vdash r\in\mathbb{L}_\beta(X)$ for any $\lev r<\lev s$. In the following template for derivations in $\RSX$, $r$ and $t$ range over terms such that $\lev r<\lev t<\lev s$:

\begin{prooftree}
\Axiom$\fCenter\Vdash r\in\mathbb{L}_\beta (X)$
\LeftLabel{$(\vee)$ if necessary}
\UnaryInf$\fCenter\Vdash r\dotin t\rightarrow r\in\mathbb{L}_\beta (X)$
\LeftLabel{$(b\forall)$}
\UnaryInf$\fCenter\Vdash (\forall x\in t)(x\in\mathbb{L}_\beta (X))$
\LeftLabel{$(\vee)$ if necessary}
\UnaryInf$\fCenter\Vdash t\dotin s\rightarrow (\forall x\in t)(x\in\mathbb{L}_\beta (X))$
\LeftLabel{$(b\forall)$}
\UnaryInf$\fCenter\Vdash (\forall y\in s)(\forall x\in y)(x\in\mathbb{L}_\beta (X))$
\LeftLabel{\ref{abbreviations}ii)}
\UnaryInf$\fCenter\Vdash\mathbb{L}_\beta (X)\dotin\mathbb{L}_\mu (X)\wedge (\forall y\in s)(\forall x\in y)(x\in\mathbb{L}_\beta (X))\quad\text{since $\beta<\mu$}$
\LeftLabel{$(b\exists)$}
\UnaryInf$\fCenter\Vdash (\exists z\in\mathbb{L}_\mu (X))(\forall y\in s)(\forall x\in y)(x\in z)$
\end{prooftree}
\end{proof}

\begin{lemma}[$\Delta_0$-Collection]\label{xcollection}{\em
Suppose $F(a,b)$ is any $\Delta_0$ formula of $\KP$.
\begin{equation*}
\Vdash (\forall x\in s)\exists yF(x,y)\rightarrow\exists z(\forall x\in s)(\exists y\in z)F(x,y)
\end{equation*}
}\end{lemma}
\begin{proof}
By Lemma \ref{xsetuplemma}i) we have
\begin{equation*}
\Vdash\neg(\forall x\in s)\exists yF(x,y), (\forall x\in s)\exists yF(x,y).
\end{equation*}
Applying $\SRX$ yields
\begin{equation*}
\CH[(\forall x\in s)\exists yF(x,y)]\;\prov{\al+1}{0}{\neg(\forall x\in s)\exists yF(x,y),\exists z(\forall x\in s)(\exists y\in z)F(x,y)}
\end{equation*}
where $\al:=\omega^{rk((\forall x\in s)\exists yF(x,y))}\#\omega^{rk((\forall x\in s)\exists yF(x,y))}$. Now two applications of $(\vee)$ provides
\begin{equation*}
\CH[(\forall x\in s)\exists yF(x,y)]\;\prov{\al+3}{0}{(\forall x\in s)\exists yF(x,y)\rightarrow\exists z(\forall x\in s)(\exists y\in z)F(x,y)}.
\end{equation*}
It remains to note that
\begin{equation*}
\al+3<\omega^{rk(\forall x\in s)\exists yF(x,y)+1}=no((\forall x\in s)\exists yF(x,y)\rightarrow\exists z(\forall x\in s)(\exists y\in z)F(x,y))
\end{equation*}
so the proof is complete.
\end{proof}
\begin{theorem}\label{xembed}{\em
If $\KP\vdash\Gamma(a_1,\ldots,a_n)$ where $\Gamma(a_1,\ldots,a_n)$ is a finite set of formulae whose free variables are amongst $a_1,\ldots,a_n$, then there is some $m<\omega$ (which we may compute from the derivation) such that
\begin{equation*}
\provx{\mathcal{H}[s_1,\ldots,s_n]}{\Omega\cdot\omega^m}{\Omega+m}{\Gamma(s_1,\ldots,s_n)}
\end{equation*}
for any operator $\CH$ and any $\RSX$ terms $s_1,\ldots,s_n$.
}\end{theorem}
\begin{proof}
Suppose $\Gamma(a_1,\ldots,a_n)\equiv\{A_1(a_1,\ldots,a_n),\ldots,A_k(a_1,\ldots,a_n)\}$. Note that for any choice of terms $s_1,\ldots,s_n$ and each $1\leq i\leq k$
\begin{align*}
rk(A_i(s_1,\ldots,s_n))&=\omega\cdot\text{max}(k(A_i(s_1,\ldots,s_n)))+m_i\quad\text{for some $m_i<\omega$}\\
&\leq\omega\cdot\Omega+m_i=\Omega+m_i.
\end{align*}
Therefore
\begin{equation*}
no(A_i(s_1,\ldots,s_n))=\omega^{rk(A_i(s_1,\ldots,s_n))}\leq\omega^{\Omega+m_i}=\omega^\Omega\cdot\omega^{m_i}=\Omega\cdot\omega^{m_i}.
\end{equation*}
So letting $m=\text{max}(m_1,\ldots,m_k)+1$ we have
\begin{align*}
no(\Gamma(s_1,\ldots,s_n))&\leq\Omega\cdot\omega^{m_1}\#\ldots\#\Omega\cdot\omega^{m_n}\\
&=\Omega\cdot (\omega^{m_1}\#\ldots\#\omega^{m_n})\\
&\leq\Omega\cdot\omega^m
\end{align*}
The proof now proceeds by induction on the \textbf{KP} derivation. If $\Gamma(a_1,\ldots,a_n)$ is an axiom of \textbf{KP} then the result follows from \ref{xsetuplemma}i), \ref{xextensionality}, \ref{xfoundation}, \ref{xinfinity}, \ref{xseparation}, \ref{xpairandunion} or \ref{xcollection}.\\

\noindent Now suppose that $\Gamma(a_1,\ldots,a_n)$ arises as the result of an inference rule.\\

\noindent Case 1. Suppose the last inference was $(b\forall)$, so $(\forall x\in a_i)F(x,\bar a)\in\GA(\bar a)$ and we are in the following situation in $\KP$
\begin{prooftree}
\AxiomC{$\Gamma(\bar{a}), c\in a_i\rightarrow F(c,\bar{a})$}
\LeftLabel{$(b\forall)$}
\UnaryInfC{$\Gamma(\bar{a})$}
\end{prooftree}
where $c$ is different from $a_1,\ldots,a_n$. Inductively we have some $m<\omega$ such that
\begin{equation*}
\tag{1}\CH[\bar{s},r]\;\prov{\Omega\cdot\omega^m}{\OO+m}{\GA(\bar s), r\in s_i\rightarrow F(r,\bar s)}\quad\text{for all $\lev r<\lev{s_i}$.}
\end{equation*}
1.1 If $s_i$ is of the form $\bar u$ we may immediately apply $(b\forall)$ to complete this case.\\

\noindent Suppose $s_i\equiv\mathbb{L}_\al(X)$ for some $\al$. Applying Lemma \ref{xinversion}iii) to (1) gives
\begin{equation*}
\tag{2}\CH[\bar{s},r]\;\prov{\Omega\cdot\omega^m}{\OO+m}{\GA(\bar s), \neg(r\in s_i), F(r,\bar s)}.
\end{equation*}
Since $\lev r<\lev s$, by Lemma \ref{xsetuplemma}vii) we have
\begin{equation*}
\tag{3}\Vdash r\in s.
\end{equation*}
Applying $\Cut$ to (1) and (2) yields
\begin{equation*}
\tag{4}\CH[\bar{s},r]\;\prov{\Omega\cdot\omega^m+1}{\OO+m}{\GA(\bar s), F(r,\bar s)}.
\end{equation*}
To which we may apply $(b\forall)$ to complete this case.\\

\noindent Suppose $s_i\equiv[x\in\mathbb{L}_\al(X)\;|\;B(x)]$, again we may apply Lemma \ref{xinversion}iii) to (1) to obtain
\begin{equation*}
\tag{5}\CH[\bar{s},r]\;\prov{\Omega\cdot\omega^m}{\OO+m}{\GA(\bar s), \neg(r\in s_i), F(r,\bar s)}.
\end{equation*}
Since $\lev r<\lev s$ by Lemma \ref{xsetuplemma}iv) we have
\begin{equation*}
\tag{6}\Vdash \neg (r\dotin s),r\in s.
\end{equation*}
Applying $\Cut$ to (5) and (6) yields
\begin{equation*}
\tag{7}\CH[\bar{s},r]\;\prov{\Omega\cdot\omega^m+1}{\OO+m}{\GA(\bar s), \neg(r\dotin s_i), F(r,\bar s)}.
\end{equation*}
Now two applications of $(\vee)$ provide
\begin{equation*}
\tag{8}\CH[\bar{s},r]\;\prov{\Omega\cdot\omega^m+3}{\OO+m}{\GA(\bar s), r\dotin s_i\rightarrow F(r,\bar s)}.
\end{equation*}
To which we may apply $(b\forall)$ to complete this case.\\

\noindent Case 2. Suppose the last inference was $(\forall)$ so $\forall xA(x,\bar a)\in\GA(\bar a)$ and we are in the following situation in $\KP$
\begin{prooftree}
\AxiomC{$\Gamma(\bar{a}), F(c,\bar{a})$}
\LeftLabel{$(\forall)$}
\UnaryInfC{$\Gamma(\bar{a})$}
\end{prooftree}
where $c$ is different from $a_1,\ldots a_n$. Inductively we have some $m<\omega$ such that
\begin{equation*}
\CH[\bar s,r]\;\prov{\Omega\cdot\omega^m}{\OO+m}{\GA(\bar s), F(r,\bar s)}\quad\text{for all terms $r$.}
\end{equation*}
We may immediately apply $(\forall)$ to complete this case.\\

\noindent Case 3. Suppose the last inference was $(b\exists)$ so $(\exists x\in s_i)F(x,\bar s)\in\GA(\bar s)$ and we are in the following situation in $\KP$

\begin{prooftree}
\AxiomC{$\Gamma(\bar{a}), c\in a_i\wedge F(c,\bar{a})$}
\LeftLabel{$(b\exists)$}
\UnaryInfC{$\Gamma(\bar{a})$}
\end{prooftree}
3.1 Suppose $c$ is different from $a_1,\ldots,a_n$. Using the induction hypothesis we find some $m<\omega$ such that
\begin{equation*}
\tag{9}\CH[\bar{s}]\;\prov{\Omega\cdot\omega^m}{\OO+m}{\GA(\bar s), \bar\emptyset\in s_i\wedge F(\bar\emptyset,\bar s)}.
\end{equation*}
3.1.1 If $s_i$ is of the form $\bar u$ we may immediately apply $(b\exists)$ to complete the case.\\

\noindent 3.1.2 Suppose $s_i$ is of the form $\mathbb{L}_\al(X)$. Applying Lemma \ref{xinversion}iv) to (1) yields
\begin{equation*}
\tag{10}\CH[\bar{s}]\;\prov{\Omega\cdot\omega^m}{\OO+m}{\GA(\bar s), F(\bar\emptyset,\bar s)}.
\end{equation*}
Noting that in this case $\bar\emptyset\dotin s\wedge F(\bar\emptyset,\bar s)\equiv F(\bar\emptyset,\bar s)$, we may apply $(b\exists)$ to complete this case.\\

\noindent 3.1.3 Suppose $s_i$ is of the form $[x\in\mathbb{L}_\al(X)\;|\;B(x)]$. First we must verify the following claim
\begin{equation*}
\tag{*}\Vdash \neg(\bar\emptyset\in s_i\wedge F(\bar\emptyset,\bar s)),\bar\emptyset\dotin s_i\wedge F(\bar\emptyset,\bar s).
\end{equation*}
Note that owing to Lemma \ref{xextensionality} we have $\Vdash [r\neq\bar\emptyset],\neg B(r),B(\emt)$ for all $\lev r<\lev{s_i}$. In the following template for derivations in $\RSX$ $r$ ranges over terms $\lev r<\lev{s_i}$.
\begin{prooftree}
\Axiom$\fCenter\Vdash[r\neq\bar\emptyset],\neg B(r),B(\bar{\emptyset})$
\LeftLabel{Lemma \ref{xcalc}ii)}
\UnaryInf$\fCenter\Vdash r\neq\bar\emptyset,\neg B(r),B(\emt)$
\LeftLabel{Lemma \ref{xcalc}ii)}
\UnaryInf$\fCenter\Vdash B(r)\rightarrow r\neq\bar\emptyset,B(\emt)$
\LeftLabel{$(\notin)$}
\UnaryInf$\fCenter\Vdash \neg(\emt\in s_i),B(\emt)$
\Axiom$\fCenter\quad\text{Lemma \ref{xsetuplemma}i)}$
\UnaryInf$\fCenter\Vdash \neg F(\emt,\bar s),F(\emt,\bar s)$
\LeftLabel{$(\wedge)$}
\BinaryInf$\fCenter\Vdash\neg(\emt\in s_i),\neg F(\emt,\bar s),B(\emt)\wedge F(\emt,\bar s)$
\LeftLabel{Lemma \ref{xcalc}ii)}
\UnaryInf$\fCenter\Vdash\neg(\emt\in s_i)\vee\neg F(\emt,\bar s),B(\emt)\wedge F(\emt,\bar s)$
\end{prooftree}
Now applying $\Cut$ to (9) and (*) we get
\begin{equation*}
\tag{11}\CH[\bar{s}]\;\prov{\Omega\cdot\omega^m+1}{\OO+m^{\prime}}{\GA(\bar s), \bar\emptyset\dotin s_i\wedge F(\bar\emptyset,\bar s)}.
\end{equation*}
Note the possible increase in cut rank. We may apply $(b\exists R)$ to (11) to complete this case.\\

\noindent 3.2 Suppose $c$ is one of $a_1,\ldots,a_n$, without loss of generality let us assume $c=a_1$. Applying the induction hypothesis we can compute some $m<\omega$ such that
\begin{equation*}
\tag{12}\CH[\bar s]\;\prov{\Omega\cdot\omega^m}{\OO+m}{\GA(\bar s),s_1\in s_i\wedge F(s_1,\bar s).}
\end{equation*}
Note that in fact 3.2 subsumes 3.1 since we can conclude (12) from the induction hypothesis regardless of whether or not $c$ is a member of $\bar a$. To help with clarity 3.1 is left in the proof above, but in later embeddings we shall dispense with such cases.\\

\noindent If $s_1$ and $s_i$ are of the form $\bar u$ and $\bar v$ with $\lev {s_1}<\lev {s_i}$ then we may immediately apply $(b\exists)$ to complete this case. If this is not the case then we verify the following claim
\begin{equation*}
\tag{**}\Vdash\neg(s_1\in s_i\wedge F(s_1,\bar s)),(\exists x\in s_i)F(x,\bar s).
\end{equation*}
To prove (**) we split into cases based on the form of $s_i$.\\

\noindent 3.2.1 Suppose $s_i$ is of the form $\bar u$.\\

\noindent 3.2.1.1 If $s_1$ is also of the form $\bar v$ [remember that by assumption $\lev{s_1}\geq\lev{s_i}$] then $\neg (s_1\in s_i),F(s_1,\bar s),(\exists x\in s_i)F(x,\bar s)$ is an axiom so we may apply $(\vee)$ twice to complete this case.\\

\noindent 3.2.1.2 Now suppose $s_1$ is not of the form $\bar v$. We have following template for derivations in $\RSX$, here $r$ ranges over terms with $\lev r<\lev{s_i}$.
\begin{prooftree}
\Axiom$\fCenter\quad\text{Lemma \ref{xsetuplemma}i)}$
\UnaryInf$\fCenter\Vdash\neg(r\in s_i),r\in s_i$
\Axiom$\fCenter\quad\text{Lemma \ref{xextensionality}}$
\UnaryInf$\fCenter\Vdash r\neq s_1,\neg F(s_1,\bar s),F(r,\bar s)$
\LeftLabel{$(\wedge)$}
\BinaryInf$\fCenter\Vdash\neg(r\in s_i),r\neq s_1,\neg F(s_1,\bar s),r\in s_i\wedge F(r,\bar s)$
\LeftLabel{$(b\exists)$}
\UnaryInf$\fCenter\Vdash\neg(r\in s_i),r\neq s_1,\neg F(s_1,\bar s),(\exists x\in s_i)F(x,\bar s)$
\LeftLabel{Lemma \ref{xcalc}ii)}
\UnaryInf$\fCenter\Vdash r\in s_i\rightarrow r\neq s_1,\neg F(s_1,\bar s),(\exists x\in s_i)F(x,\bar s)$
\LeftLabel{$(\notin)$}
\UnaryInf$\fCenter\Vdash \neg(s_1\in s_i),\neg F(s_1,\bar s),(\exists x\in s_i)F(x,\bar s)$
\LeftLabel{Lemma \ref{xcalc}ii)}
\UnaryInf$\fCenter\Vdash \neg(s_1\in s_i)\vee\neg F(s_1,\bar s),(\exists x\in s_i)F(x,\bar s)$
\end{prooftree}
3.2.2 Now suppose $s_i$ is of the form $\mathbb{L}_\al(X)$. In the following template for derivations in $\RSX$ $r$ ranges over terms with $\lev r<\lev{s_i}$.
\begin{prooftree}
\Axiom$\fCenter\quad\text{Lemma \ref{xextensionality}}$
\UnaryInf$\fCenter\Vdash r\neq s_1,\neg F(s_1,\bar s), F(x,\bar s)$
\LeftLabel{\ref{abbreviations}ii)}
\UnaryInf$\fCenter\Vdash r\neq s_1,\neg F(s_1,\bar s),r\dotin s_i\wedge F(x,\bar s)$
\LeftLabel{$(b\exists)$}
\UnaryInf$\fCenter\Vdash r\neq s_1,\neg F(s_1,\bar s),(\exists x\in s_i)F(x,\bar s)$
\LeftLabel{\ref{abbreviations}ii)}
\UnaryInf$\fCenter\Vdash r\dotin s_i\rightarrow r\neq s_1,\neg F(s_1,\bar s),(\exists x\in s_i)F(x,\bar s)$
\LeftLabel{$(\notin)$}
\UnaryInf$\fCenter\Vdash \neg(s_1\in s_i),\neg F(s_1,\bar s),(\exists x\in s_i)F(x,\bar s)$
\LeftLabel{Lemma \ref{xcalc}ii)}
\UnaryInf$\fCenter\Vdash \neg(s_1\in s_i)\vee\neg F(s_1,\bar s),(\exists x\in s_i)F(x,\bar s)$
\end{prooftree}
3.2.3 Finally suppose $s_i$ is of the form $[x\in\mathbb{L}_\al\;|\;B(x)]$. In the following template for derivations in $\RSX$ $r$ ranges over terms with $\lev r<\lev{s_i}$.
\begin{prooftree}
\Axiom$\fCenter\text{Lemma \ref{xsetuplemma}i)}$
\UnaryInf$\fCenter\Vdash \neg B(r),B(r)$
\Axiom$\fCenter\quad\text{Lemma \ref{xextensionality}}$
\UnaryInf$\fCenter\Vdash r\neq s,\neg F(s_1,\bar s),F(r,\bar s)$
\LeftLabel{$(\wedge)$}
\BinaryInf$\fCenter\Vdash \neg B(r), r\neq s_1,\neg F(s_1,\bar s),B(r)\wedge F(r,\bar s)$
\LeftLabel{$(b\exists)$}
\UnaryInf$\fCenter\Vdash \neg B(r), r\neq s_1,\neg F(s_1,\bar s),(\exists x\in s_i)F(x,\bar s)$
\LeftLabel{Lemma \ref{xcalc}ii)}
\UnaryInf$\fCenter\Vdash B(r)\rightarrow r\neq s_1,\neg F(s_1,\bar s),(\exists x\in s_i)F(x,\bar s)$
\LeftLabel{$(\notin)$}
\UnaryInf$\fCenter\Vdash \neg(s_1\in s_i),\neg F(s_1,\bar s),(\exists x\in s_i)F(x,\bar s)$
\LeftLabel{Lemma \ref{xcalc}ii)}
\UnaryInf$\fCenter\Vdash \neg(s_1\in s_i)\vee\neg F(s_1,\bar s),(\exists x\in s_i)F(x,\bar s)$
\end{prooftree}
This completes the proof of the claim (**). It remains to note that we may apply $\Cut$ to (**) and (12) to complete Case 3.\\

\noindent Case 4. Suppose the last inference was $(\exists)$ so $\exists xF(x,\bar s)\in\GA(\bar s)$ and we are in the following situation in $\KP$:
\begin{prooftree}
\AxiomC{$\Gamma(\bar{a}), F(c,\bar{a})$}
\LeftLabel{$(\exists)$}
\UnaryInfC{$\Gamma(\bar{a})$}
\end{prooftree}
Let $p=s_j$ if $c=a_j$ otherwise let $p=\emt$, from the induction hypothesis we can compute some $m<\omega$ such that
\begin{equation*}
\CH[\bar s]\;\prov{\OO\cdot\omega^m}{\OO+m}{\GA(\bar s),F(p,\bar s)}.
\end{equation*}
Applying $(\exists)$ completes this case.\\

\noindent Case 5. If the last inference was $(\wedge)$ or $(\vee)$ the result follows immediately by applying the corresponding $\RSX$ inference to the induction hypotheses.\\

\noindent Case 6. Finally suppose the last inference was $\Cut$. So we are in the following situation in $\KP$
\begin{prooftree}
\Axiom$\fCenter\Gamma(\bar a), B(\bar a,\bar b)$
\Axiom$\fCenter\Gamma(\bar a),\neg B(\bar a,\bar b)$
\LeftLabel{(Cut)}
\BinaryInf$\fCenter\Gamma(\bar{a})$
\end{prooftree}
Here $\bar b:=b_1,\ldots,b_l$ denotes the free variables occurring in $B$ that are different from $a_1,\ldots,a_n$. Let $\bar\emt$ denote the sequence of $l$ occurrences of $\emt$. From the induction hypothesis we find $m_1$ and $m_2 $ such that
\begin{align*}
&\CH[\bar s]\;\prov{\OO\cdot\omega^{m_1}}{\OO+m_1}{\GA(\bar s),B(\bar s,\bar\emt)}\\
&\CH[\bar s]\;\prov{\OO\cdot\omega^{m_1}}{\OO+m_2}{\GA(\bar s),\neg B(\bar s,\bar\emt)}
\end{align*}
To which we may apply $\Cut$ to complete the proof.
\end{proof}
\section{A well ordering proof in KP}
The aim of this section is to give a well ordering proof in $KP$ for initial segments of formal ordinal terms from $T(\theta)$. First let
\begin{eqnarray}\label{en}
e_0 &:=&\Omega+1\\ \nonumber
e_{n+1}&:=&\omega^{e_n}.
\end{eqnarray}
Each $e_n$ is  a formal term belonging to every representation system $T(\theta)$ from \ref{ttheta}. Although the term is the same, the order type of terms in $T(\theta)$ below $e_n$ will be dependent upon $\theta$.
 We aim to verify that for every $n<\omega$
\begin{align*}
\textbf{KP}\vdash A_n(\theta):=&\exists\alpha\exists f[\text{dom}(f)=\alpha\:\wedge\:\text{range}(f)=\{a\in T(\theta)\:|\:a\prec\psi_\theta (e_n))\}\\
&\wedge\forall\gamma,\delta\in\text{dom}(f)(\gamma<\delta\rightarrow f(\gamma)\prec f(\delta))]
\end{align*} where in the above formula $\prec$ denotes the ordering on $T(\theta)$.
Formally $A_n(\theta)$ is a $\Sigma$-formula of $\KP$ in which $\theta$ is a parameter (free variable) ranging over ordinals. For the remainder of this section we argue informally in \textbf{KP}. The symbols $\alpha,\beta,\gamma,\ldots$ are to be \textbf{KP}-variables ranging over ordinals and are ordered by $<$, the symbols $a,b,c,\ldots$ are seen as \textbf{KP}-variables ranging over codes of formal terms from $T(\theta)$, these are ordered by $\prec$. For the remainder of this section the variable $\theta$ will remain free as we argue in \textbf{KP}, for ease of reading we shall simply $\OO$ and $\psi$ instead of $\OO_\theta$ and $\psi_\theta$. This proof is an adaptation to the relativised case of a well ordering proof in \cite{mi99} or \cite{repl}.
\begin{definition}{\em The set Acc$_\theta$ is defined by
\begin{align*}
\text{Acc}_\theta:=&\{a\in T(\theta)\mid a\prec\Omega\;\wedge\;\exists\alpha\exists f[\text{dom}(f)=\alpha\:\wedge\:\text{range}(f)=\{b\::\:b\preceq a\}\\
&\wedge\forall\gamma,\delta\in\text{dom}(f)(\gamma<\delta\rightarrow f(\gamma)\prec f(\delta))]\}.
\end{align*}
}\end{definition}
\begin{lemma}[$\text{Acc}_\theta$-induction]\label{WO1}{\em
For any \textbf{KP}-formula $F(a)$ we have
\begin{equation*}
(\forall a\in\text{Acc}_\theta)[(\forall b\prec a) F(b)\rightarrow F(a)]\rightarrow (\forall a\in\text{Acc}_\theta)F(a).
\end{equation*}
}\end{lemma}
\begin{proof}
For $a\in\text{Acc}_\theta$ let $o(a)$ and $f_a$ be the unique ordinal and function such that $o(a)=\text{dom}(f_a)$, $\{b\;:\:b\preceq a\}=\text{range}(f_a)$ and $\forall\gamma,\delta\in o(a)(\gamma<\delta\rightarrow f_a(\gamma)\prec f_a (\delta))$. Now for a contradiction let us assume that
\begin{equation*}
(\forall a\in\text{Acc}_\theta)[(\forall b\prec a)F(b)\rightarrow F(a)]\quad\text{but}\quad\neg F(a_0) \text{ for some } a_0\in\text{Acc}_\theta
\end{equation*}
Using set induction/foundation we may pick $a_0$ such that $o(a_0)$ is minimal. (Note that here we must make use of the full set induction schema of \textbf{KP} since the formula $F$ is of unbounded complexity) Now for any $b\prec a_0$ we have $o(b)< o(a_0)$, thus by our choice of $a_0$ we get $F(b)$, thus we have
\begin{equation*}
(\forall b\prec a_0)F(b).
\end{equation*}
So by assumption we have $F(a_0)$, contradiction.
\end{proof}
\begin{lemma}\label{WO2}{\em Acc$_\theta$ has the following closure properties:
\begin{description}
\item[i)] $b\in\text{Acc}_\theta\wedge a\prec b\quad\rightarrow\quad a\in \text{Acc}_\theta$
\item[ii)] $(\forall a\prec b)(a\in\text{Acc}_\theta)\quad\rightarrow\quad b\in\text{Acc}_\theta$
\item[iii)] $a,b\in\text{Acc}_\theta\quad\rightarrow\quad a+b\in\text{Acc}_\theta$
\item[iv)] $a,b\in\text{Acc}_\theta\quad\rightarrow\quad\varphi ab\in\text{Acc}_\theta$
\item[v)] $(\forall\beta\leq\theta)\;\Gamma_\beta\in\text{Acc}_\theta$
\end{description}
}\end{lemma}
\begin{proof}
i) Using the notation defined at the start of the proof of Lemma \ref{WO1} we may define
\begin{equation*}
o(a):=\{\delta\in o(b)\:|\: f_b(\delta)\preceq a\}\quad\text{and}\quad f_a:=f_b|_{o(a)+1}
\end{equation*}
thus witnessing that $a\in\text{Acc}_\theta$.\\

\noindent ii) Let us assume that $(\forall a\prec b)(a\in\text{Acc}_\theta)$, we must verify that $b\in\text{Acc}_\theta$. Using $\Delta_0$-Separation and Infinity we may form the set $\{a\:|\:a\prec b\}$, therefore $f:=\cup_{a\prec b}f_a$ is a set by $\Delta_0$-Collection and Union. Let $\beta:=\text{dom}(f)$. Setting $o(b):=\beta+1$ and $f_b:=f\cup\{(\beta,b)\}$ furnishes us with the correct witnesses to confirm that $b\in\text{Acc}_\theta$.\\

\noindent iii) Firstly we must specify what $a+b$ means, since it may not be the case that the string $a+b$ is a term in $T(\theta)$. However, we may define a $\theta$-primitive recursive function $+:T(\theta) \times T(\theta)\rightarrow T(\theta)$ which corresponds to ordinal addition.\\

\noindent Let us assume that $(\forall c\prec b)(a+c\in\text{Acc}_\theta)$, now if we can show that $a+b\in\text{Acc}_\theta$ then the desired result will follow from $\text{Acc}_\theta$-induction (\ref{WO1}). Now let $d\prec a+b$, either $d\preceq a$ in which case $d\in\text{Acc}_\theta$ by i) or $d\succ a$ and thus $d=a+c$ for some unique $c\prec b$. Such a $c$ may be determined in a $\theta$-primitive recursive fashion, hence $d\in\text{Acc}_\theta$ by assumption. Thus we have
\begin{equation*}
(\forall d\prec a +b)(d\in\text{Acc}_\theta).
\end{equation*}
From which we may use ii) to obtain $a+b\in\text{Acc}_\theta$, completing the proof.\\

\noindent iv) Again a function $\varphi: T(\theta)\times T(\theta)\rightarrow T(\theta)$ may be defined in a $\theta$-primitve recursive fashion. It is our aim to show $(\forall x,y\in\text{Acc}_\theta)(\varphi xy\in\text{Acc}_\theta)$, to this end let
\begin{equation*}
F(a):=(\forall b\in\text{Acc}_\theta)(\varphi ab\in\text{Acc}_\theta)
\end{equation*}
and assume
\begin{equation*}\tag{*}\label{iv1}
(\forall z\prec a)F(z)
\end{equation*}
by \ref{WO1} it suffices to verify $F(a)$. So let us assume
\begin{equation*}\tag{**}\label{iv2}
a,b\in\text{Acc}_\theta\quad\text{and}\quad (\forall y\prec b)(\varphi ay\in\text{Acc}_\theta)
\end{equation*}
now we must verify $\varphi ab\in\text{Acc}_\theta$. To do this we prove that
\begin{equation*}
d\prec\varphi ab\Rightarrow d\in \text{Acc}_\theta
\end{equation*}
by induction on $Gd$; the term complexity of $d$.\\

\noindent 1) If $d$ is strongly critical then $d\preceq a$ or $d\preceq b$ in which case $d\in\text{Acc}_\theta$ by (\ref{iv1}) or (\ref{iv2}).\\

\noindent 2) If $d\equiv\varphi d_0 d_1$ then we have the following subcases:\\

\noindent 2.1) If $d_0\prec a$ and $d_1\prec\varphi ab$ then since $Gd_1<Gd$ we get $d_1\in\text{Acc}_\theta$ from the induction hypothesis. So by (\ref{iv1}) we get $d\equiv\varphi d_0d_1\in\text{Acc}_\theta$\\

\noindent 2.2) If $d\equiv\varphi ad_1$ and $d_1\prec b$ then $d\in\text{Acc}_\theta$ by (\ref{iv2}).\\

\noindent 2.3 If $a\prec d_0$ and $d\prec b$ then $d\in\text{Acc}_\theta$ since $b\in\text{Acc}_\theta$.\\

\noindent 3. If $d\equiv d_1+\ldots+d_n$ and $n>1$ we get $d_1,\ldots,d_n\in\text{Acc}_\theta$ from the induction hypothesis and thus $d\in\text{Acc}_\theta$ follows from iii).\\

\noindent Thus we have verified that
\begin{equation*}
(\forall b\in\text{Acc}_\theta)[(\forall y\prec b)(\varphi ay\in\text{Acc}_\theta)\rightarrow\varphi ab\in\text{Acc}_\theta]
\end{equation*}
So, from $\text{Acc}_\theta$-induction we get $(\forall b\in\text{Acc}_\theta)(\varphi ab\in\text{Acc}_\theta)$, ie. $F(a)$ completing the proof.\\

\noindent v) We aim to show that
\begin{equation*}
(\forall\beta\leq\theta)[(\forall\gamma<\beta)(\Gamma_\gamma\in\text{Acc}_\theta)\rightarrow\Gamma_\beta\in\text{Acc}_\theta]
\end{equation*}
from which we may use transfinite induction along $\theta$ (since $\theta$ is an ordinal) to obtain the desired result.\\

\noindent So suppose $\beta\leq\theta$ and $(\forall\delta<\beta)(\Gamma_\delta\in\text{Acc}_\theta)$. Now suppose $b\prec\Gamma_\beta$, by induction on the term complexity of $b$ we verify that $b\in\text{Acc}_\theta$.\\

\noindent If $b\equiv 0$ we are trivially done by ii) or if $b\equiv\Gamma_\delta$ for some $\delta<\beta$ then we know $b\in\text{Acc}_\theta$ by assumption.\\

\noindent If $b\equiv b_0+\ldots+b_n$ or $b\equiv\varphi b_0 b_1$ then we may use parts iii) and iv) and the induction hypothesis since the components $b_i$ have smaller term complexity.\\

\noindent It cannot be the case that $b\equiv\psi b_0$ since $\psi a\succ\Gamma_\theta$ for every $a$.\\

\noindent Thus using ii) we get that $\Gamma_\beta\in\text{Acc}_\theta$ and the proof is complete.
\end{proof}
\begin{definition}{\em By recursion through the construction of ordinal terms in $T(\theta)$ we define the set $SC_{\prec\Omega}(a)$ which lists the most recent strongly critical ordinal below $\Omega$ used in the build up of the ordinal term $a$:
\begin{description}
\item[1)] $SC_{\prec\Omega}(0):=SC_{\prec\Omega}(\Omega):=\emptyset$
\item[2)] $SC_{\prec\Omega}(a):=\{a\}$ if $a\equiv\Gamma_\beta$ for some $\beta\leq\theta$ or $a\equiv\psi a_0$.
\item[3)] $SC_{\prec\Omega}(a_1+\ldots+a_n):=\cup_{1\leq i\leq n}SC_{\prec\Omega}(a_i)$
\item[4)] $SC_{\prec\Omega}(\varphi a_0 a_1):=SC_{\prec\Omega}(a_0)\cup SC_{\prec\Omega}(a_1)$
\item[5)] $SC_{\prec\OO}(\psi a):=\{\psi a\}$.
\end{description}
Now let
\begin{equation*}
M_\theta:=\{a\in T(\theta)\:|\:SC_{\prec\Omega}(a)\subseteq \text{Acc}_\theta\}
\end{equation*}
and
\begin{equation*}
a\prec_{M_\theta} b:= a,b\in M_\theta\wedge a\prec b.
\end{equation*}
Finally for a definable class $U$ we define the following formula
\begin{equation*}
\text{Prog}_{M_\theta}(U):=(\forall y\in M_\theta)[(\forall z\prec_{M_\theta} y)(z\in U)\rightarrow (y\in U)]
\end{equation*}
}\end{definition}
\begin{lemma}\label{WO3}{\em
\begin{equation*}
\text{Acc}_\theta= M_\theta\cap\Omega:=\{a\in M_\theta\:|\:a\prec\Omega\}
\end{equation*}
}\end{lemma}
\begin{proof}
Suppose that $a\in\text{Acc}_\theta$ and observe that $(\forall x\in SC_{\prec\Omega}(a))(x\preceq a)$, thus $SC_{\prec\Omega}(a)\subseteq\text{Acc}_\theta$ by \ref{WO2}i) thus we have verified that $a\in M_\theta\cap\Omega$.\\

\noindent Now let us suppose that $a\in M_\theta\cap\Omega$, so we know that $SC_{\prec\Omega}(a)\subseteq\text{Acc}_\theta$. By induction on the term complexity $Ga$ we verify that $a\in\text{Acc}_\theta$.\\

\noindent Clearly $0\in\text{Acc}_\theta$ and if $a\equiv\Gamma_\beta$ for some $\beta\leq\theta$ then $a\in\text{Acc}_\theta$ by Lemma \ref{WO2}v).\\

\noindent If $a\equiv a_1+\ldots+a_n$ then we get $a_1,\ldots,a_n\in M_\theta\cap\Omega$ since $SC_{\prec\Omega}(a_i)\subseteq SC_{\prec\Omega}(a)$ for each $i$. Now using the induction hypothesis we get $a_1,\ldots,a_n\in\text{Acc}_\theta$ and so by Lemma \ref{WO2}ii) we have $a\in\text{Acc}_\theta$.\\

\noindent If $a\equiv \varphi bc$ then we get $b,c\in M_\theta\cap\Omega$, so using the induction hypothesis we get $b,c\in\text{Acc}_\theta$ and so by Lemma \ref{WO2}iii) we have $a\in\text{Acc}_\theta$.\\

\noindent If $a\equiv\psi a_0$ then $SC_{\prec\Omega}(a)=\{a\}$ so we have $a\in\text{Acc}_\theta$ by assumption.
\end{proof}
\begin{definition}{\em For a definable class $U$ let
\begin{equation*}
U^\delta:=\{b\in M_\theta\:|\:(\forall a\in M_\theta)[M_\theta\cap a\subseteq U\rightarrow M_\theta\cap a+\omega^b\subseteq U]\}
\end{equation*}
where $M_\theta\cap a:=\{b\in M_\theta\:|\:b\prec a\}$.
}\end{definition}
\begin{lemma} \label{WO4}{\em
$\textbf{KP}\vdash\text{Prog}_{M_\theta}(U)\rightarrow\text{Prog}_{M_\theta}(U^\delta)$
}\end{lemma}
\begin{proof}
Assume
\begin{align*}
&\text{Prog}_{M_\theta}(U) \tag{1}\\
&b\in M_\theta\tag{2}\\
&(\forall x\prec_{M_\theta} b)(z\in U^\delta)\tag{3}
\end{align*}
Under these assumptons we need to verify that $b\in U^\delta$. Since we already have that $b\in M_\theta$ by (2), it suffices to verify
\begin{equation*}
(\forall a\in M_\theta)[M_\theta\cap a\subseteq U\rightarrow M_\theta\cap a+\omega^b\subseteq U]
\end{equation*}
to this end we assume that
\begin{equation*}\tag{4}
a\in M_\theta\quad\text{and} \quad M_\theta\cap a\subseteq U
\end{equation*}
Now choose some $d\in M_\theta\cap a+\omega^b$, we must show that $d\in U$ under the assumptions (1)-(4).\\

\noindent If $d\prec a$ then we have $d\in U$ by (4).\\

\noindent If $d=a$ then using (1) and (4) we have $a\in U$.\\

\noindent If $d\succ a$ then since $d\prec a+\omega^b$, we may find $d_1,\ldots,d_k$ such that
\begin{equation*}
d=a+\omega^{d_1}+\ldots+\omega^{d_k}\quad\text{and}\quad d_k\preceq \ldots\preceq d_1\prec b
\end{equation*}
Since $M_\theta\cap a\subseteq U$ we get $M_\theta\cap a+\omega^{d_1}\subseteq U$ from (3).\\

\noindent In a similar fashion using (3) a further $k-1$ times we obtain
\begin{equation*}
M_\theta\cap a+\omega^{d_1}+\ldots+\omega^{d_k}\subseteq U
\end{equation*}
Finally using one application of Prog$_{M_\theta} (U)$ (assumption (1)) we have $d\in U$ and thus the proof is complete.
\end{proof}
\begin{definition}{\em We define the class $X_\theta$ in \textbf{KP} as
\begin{equation*}
X_\theta:=\{a\in M_\theta\:|\:(\exists x\in Ka)(x\succeq a)\vee\psi a\in\text{Acc}_\theta\}
\end{equation*}
}\end{definition}
\noindent Recall that the function $k$ was defined in Definition \ref{K} and can be computed in a $\theta$-primitive recursion fashion. The class $X_\theta$ may be thought of as those $a\in M_\theta$ for which either $\psi a$ is undefined or $\psi a\in\text{Acc}_\theta$.
\begin{lemma}\label{WO5}{\em
$\textbf{KP}\vdash\text{Prog}_{M_\theta}(X_\theta).$
}\end{lemma}
\begin{proof}
Assume
\begin{align*}
&a\in M_\theta\tag{1}\\
&(\forall z\prec_{M_\theta} a)(z\in X_\theta)\tag{2}
\end{align*}
We need to verify that $a\in X_\theta$. If $(\exists x\in Ka)(x\succeq a)$ then we are done, so assume $(\forall x\in Ka)(x\prec a)$ and thus $\psi a\in T(\theta)$ and we must verify that $\psi a\in\text{Acc}_\theta$. To achieve this we verify that
\begin{equation*}\tag{*}
b\prec\psi a\:\Rightarrow\:b\in\text{Acc}_\theta
\end{equation*}
from which we would be done by \ref{WO2}ii). To verify (*) we proceed by induction on $Gb$, the term complexity of $b$.\\

\noindent If $b\equiv 0$ or $b\equiv\Gamma_\beta$ for some $\beta\leq\theta$ we are done by \ref{WO2}v).\\

\noindent If $b\equiv b_0+\ldots+b_n$ or $b\equiv\varphi b_0 b_1$ then the result follows by the induction hypothesis and \ref{WO2}ii) or \ref{WO2}iii).\\

\noindent So suppose that $b\equiv\psi b_0$. It must be the case that $(\forall x\in Kb_0)(x\prec b_0)$ and $b_0\prec a$. We must now show that $b_0\in M_\theta$ in order to use (2) to conclude that $b_0\in X_\theta$. The claim is that
\begin{equation*}\tag{**}
\text{SC}_{\prec\Omega}(b_0)\subseteq\text{Acc}_\theta\quad\text{and thus}\quad b_0\in M_\theta
\end{equation*}
Suppose $d\in\text{SC}_{\prec\Omega}(b_0)$ then either $d\equiv\Gamma_\beta$ for some $\beta\leq\theta$ in which case $d\in\text{Acc}_\theta$ by \ref{WO2}v) or $d\equiv\psi d_0\prec\psi a$ for some $d_0$. But
\begin{equation*}
Gd\leq Gb_0<Gb
\end{equation*}
and thus $d\in\text{Acc}_\theta$ by induction hypothesis. Thus the claim (**) is verified. Now using (2) we obtain $b_0\in X_\theta$ which implies $b\equiv\psi b_0\in\text{Acc}_\theta$.
\end{proof}
\begin{lemma}\label{WO6}{\em
For any $n<\omega$ and any definable class $U$
\begin{equation*}
\textbf{KP}\vdash\text{Prog}_{M_\theta}(U)\:\rightarrow\: M_\theta\cap e_n\subseteq U\wedge e_n\in U.
\end{equation*}
}\end{lemma}
\begin{proof}
We proceed by induction on $n$ [outside of \textbf{KP}].\\

\noindent If $n=0$ then Prog$_{M_\theta}(U)$ says that
\begin{equation*}
(\forall a\in\text{Acc}_\theta)[(\forall b\prec a)(b\in U)\rightarrow a\in U].
\end{equation*}
So using Acc$_\theta$-induction (Lemma \ref{WO1}) we obtain Acc$_\theta\subseteq U$. Hence from \ref{WO3} we get $M_\theta\cap\Omega\subseteq U$. Now $\Omega,\Omega+1\in M_\theta$ so using  Prog$_{M_\theta}(U)$ a further two times we have $\Omega +1:=e_0\in U$ as required.\\

\noindent Now suppose the result holds up to $n$; since the induction hypothesis holds for \textit{all} definable classes we have that that
\begin{equation*}
\textbf{KP}\vdash\text{Prog}_{M_\theta}(U^\delta)\rightarrow M_\theta\cap e_n\subseteq U^\delta\wedge e_n\in U^\delta
\end{equation*}
and by Lemma \ref{WO4} we have
\begin{equation*}\tag{1}
\textbf{KP}\vdash\text{Prog}_{M_\theta}(U)\rightarrow M_\theta\cap e_n\subseteq U^\delta\wedge e_n\in U^\delta.
\end{equation*}
Now we argue informally in \textbf{KP}. Suppose Prog$_{M_\theta}(U)$, then from (1) we obtain
\begin{equation*}
M_\theta\cap e_n\subseteq U^\delta\quad\wedge\quad e_n\in U^\delta.
\end{equation*}
This says that
\begin{equation*}
(\forall b\in M_\theta\cap(e_n+1))(\forall a\in M_\theta)[M_\theta\cap a\subseteq U\rightarrow M_\theta\cap a+\omega^b\subseteq U].
\end{equation*}
Now if we put $a=0$ and $b=e_n$ (noting that $e_n\in M_\theta$) we obtain
\begin{equation*}
M_\theta\cap\omega^{e_n}\subseteq U
\end{equation*}
from which Prog$_{M_\theta}(U)$ implies $\omega^{e_n}\in U$ as required.
\end{proof}
\begin{theorem}\label{WO}{\em For every $n<\omega$
\begin{equation*}
\textbf{KP}\vdash\forall \theta\,\psi (e_n)\in\text{Acc}_\theta
\end{equation*}
and hence $\textbf{KP}\vdash \forall \theta\,A_n(\theta)$.
}\end{theorem}
\begin{proof}
By \ref{WO5} we have $\text{Prog}_{M_\theta}(X_\theta)$ recalling that
\begin{equation*}
X_\theta:=\{a\in M_\theta\:|\:(\exists x\in Ka)(x\succeq a)\vee\psi a\in\text{Acc}_\theta\}.
\end{equation*}
So from \ref{WO6} we get $e_n\in X_{\theta}$ for any $n<\omega$ and thus $\psi (e_n)\in\text{Acc}_\theta$.
\end{proof}
\section{The provably total set functions of \textbf{KP}}
At this point we should perhaps remind ourselves that the ordinal $\psi\alpha$ depends on a parameter $\theta$ which is the rank of $TC(\{X\})$ as $\psi$ is defined simultaneously with the sets $B_{\theta}(\alpha)$. After Definition \ref{theta} we adopted the convention to drop
the subscript $\theta$ from $\psi_{\theta}$. For the next application we have to be aware of this dependence.
For each $n<\omega$ we define the following recursive set function
\begin{equation*}
G_n(X):=L_{\psi_{\theta}(e_n)}(X)
\end{equation*} where $\theta$ is the rank of $TC(\{X\})$.
For a formula $A(a,b)$ of $\KP$ let
\begin{equation*}
\forall x\exists !yA(x,y):=\forall x\forall y_1\forall y_2[A(x,y_1)\wedge A(x,y_2)\rightarrow y_1=y_2]\wedge \forall x\exists y A(x,y).
\end{equation*}
\begin{definition}{\em If $T$ is a theory formulated in the language of set theory, $f$ a set function and $\mathfrak{X}$ a class of formulae. We say that $f$ is $\mathfrak{X}$ definable in $T$ if there is some $\mathfrak{X}$-formula $A_f(a,b)$ with exactly the free variables $a,b$ such that
\begin{description}
\item[i)] $V\models A_f(x,y)\leftrightarrow f(x)=y$.
\item[ii)] $T\vdash\forall x\exists!yA_f(x,y)$.
\end{description}
}\end{definition}
\begin{theorem}\label{xconc}{\em Suppose $f$ is a set function that is $\Sigma$ definable in $\KP$, then there is some $n$ (which we may compute from the finite derivation) such that
\begin{equation*}
V\models\forall x(f(x)\in G_n(x)).
\end{equation*}
Moreover $G_m$ is $\Sigma$ definable in $\KP$ for each $m<\omega$.
}\end{theorem}
\begin{proof}
Let $A_f(a,b)$ be the $\Sigma$ formula expressing $f$ such that $\KP\vdash\forall x\exists!yA_f(x,y)$ and fix an arbitrary set $X$. Let $\theta$ be the rank of $X$. Applying Theorem \ref{xembed} we can compute some $k<\omega$ such that
\begin{equation*}
\CH_0\;\prov{\OO\cdot\omega^k}{\OO+k}{\forall x\exists!yA_f(x,y)}.
\end{equation*}
Applying Lemma \ref{xinversion} iv) twice we get
\begin{equation*}
\CH_0\;\prov{\OO\cdot\omega^k}{\OO+k}{\exists yA_f(X,y)}.
\end{equation*}
Applying Theorem \ref{xpredce} (predicative cut elimination) we get
\begin{equation*}
\CH_0\;\prov{e_{k+1}}{\OO+1}{\exists yA_f(X,y)}.
\end{equation*}
Now by Theorem \ref{ximpredce} (collapsing) we have
\begin{equation*}
\CH_{e_{k+2}}\;\prov{\psi_{\theta}(e_{k+2})}{\psi_{\theta}(e_{k+2})}{\exists yA_f(X,y)}.
\end{equation*}
Applying Theorem \ref{xpredce} (predicative cut elimination) again yields
\begin{equation*}
\CH_\gamma\;\prov{\varphi(\psi_{\theta}\gamma)(\psi_{\theta}\gamma)}{0}{\exists yA_f(X,y)}\quad\text{where $\gamma:=e_{k+2}$.}
\end{equation*}
Now by Lemma \ref{xboundedness} (boundedness) we obtain
\begin{equation*}
\tag{1}\CH_\gamma\;\prov{\al}{0}{(\exists y\in\mathbb{L}_{\al}(X))A_f(X,y)^{\mathbb{L}_\al(X)}}\quad\text{where $\al:=\varphi(\psi_{\theta}\gamma)(\psi_{\theta}\gamma)$.}
\end{equation*}
Since (1) contains no instances of $\Cut$ or $\SRX$, it follows by induction on $\al$ that
\begin{equation*}
L_\al(X)\models\exists yA_f(X,y)
\end{equation*}
It remains to note that $L_\al(X)\subseteq G_{k+3}(X)$ to complete this direction of the proof.\\

\noindent For the other direction we argue informally in $\KP$. Let $X$ be an arbitrary set, we may specify the rank of $X$ in a $\Delta_0$ manner(\cite{ba} p. 29). By Theorem \ref{WO} we can find an ordinal of the same order type as $\psi_{\theta}(e_n)$ with $\theta$ being the rank of $TC(\{X\})$.
 We can now generate $L_{\psi_{\theta}(e_n)}(X)$ by $\Sigma$-recursion (\cite{ba} p. 26 Theorem 6.4).
\end{proof}
\noindent The comparison of Theorem \ref{PA-fun} with Theorem \ref{xconc} provides a pleasing relation between the arithmetic and set theoretic worlds.
\begin{remark}\label{rem8.3}{\em In fact the first part of \ref{xconc} can be carried out {\em inside $\KP$}, i.e. If $f$ is $\Sigma$ definable in $\KP$ then we can compute some $n$ such that $\KP\vdash\forall x(\exists!y\in G_n(x))A_f(x,y)$. This is not immediately obvious since it appears we need induction up to $\psi_{\theta}(\varepsilon_{\OO +1})$, which we do not have access to in $\KP$. The way to get around this is to note that we could, in fact, have managed with an infinitary system based on an ordinal representation built out of $B_\theta(e_m)$, provided $m$ is high enough, and we may compute how high $m$ needs to be from the finite derivation. We do have access to induction up to $\psi(e_m)$ for any ordinal $\theta$ in $\KP$ by Theorem \ref{WO}.
}\end{remark}

\section{Applications to semi-intuitionistic $\KP$}
$\PA$ is conservative over its intuitionistic cousin (called Heyting arithmetic, $\mathbf{HA}$) for $\Pi^0_2$-statements. One might wonder whether a corresponding result holds in set theory for $\Pi_2$-statements. As it turns out, such a result does not obtain for $\KP$ and its intuitionistic version $\IKP$,\footnote{See \cite{mar,book} for a definition of $\IKP$.} however,
adding the law of excluded middle for atomic formulas to $\IKP$ yields conservativity for $\Pi_2$ theorems.

A semi-intuitionistic version of $\IKP$ is obtained by assuming the law of excluded middle for atomic formulas, i.e., \begin{eqnarray}
\label{AEM} &&\forall x\forall y\,(x\in y\,\vee\, \neg x\in y).\end{eqnarray}
 Semi-intuitionistic versions of $\KP$ have become important in Feferman's work in connection with discussions of definiteness
of concepts and the continuum hypothesis (cf. \cite{F1,F2,F3,F4,R-CH}).

\begin{theorem} $\KP$ is $\Pi_2$ conservative over the semi-intuitionistic theory {\em $\IKP$} plus (\ref{AEM}).
\end{theorem}
\begin{proof} Let $T$ be the theory $\IKP$ augmented by (\ref{AEM}). Assume that $\KP\vdash \forall x\exists y A(x,y)$, where
$A(a,b)$ is  $\Delta_0$. We now argue in $T$. Let $X$ be an arbitrary set. As in the proof of Theorem \ref{xembed} we can determine an $\alpha$ (uniformly depending on the rank of $TC(\{X\})$ such that
\begin{eqnarray}\label{9.1.1}
&&\CH_\gamma\;\prov{\al}{0}{(\exists y\in\mathbb{L}_{\al})A_f(X,y)^{\mathbb{L}_\al}.}\quad\text{}
\end{eqnarray} To see that we can do this inside $T$ note that the $m$ in Remark \ref{rem8.3} does not depend on $\theta$.
Since (\ref{9.1.1}) contains no instances of $\Cut$ or $\SRX$, it follows by induction on $\al$ that
\begin{equation*}
L_\al(X)\models\exists yA(X,y).
\end{equation*}
Excluded middle for atomic formulas is required at several points. For instance it is needed in Lemma \ref{xsetuplemma}i), Case 1.
Also when showing that all sequents $\Lambda$ occurring in the derivation (\ref{9.1.1}) are true in $L_{\al}(X)$\footnote{This means that the disjunction over all formulas in $\Lambda$ is true in $L_{\al}(X)$.} one needs to invoke
the law of excluded middle for $\Delta_0$-formulas. The latter follows from (\ref{AEM}) with the help of $\Delta_0$-Separation.
\end{proof}

\section{A relativised ordinal analysis of $\KPP$}
With the help of \cite{powerKP} and the foregoing machinery one can also characterize the provable power recursive set
functions of  Power Kripke-Platek set theory, $\KPP$.
For background on $\KPP$ see \cite{powerKP}. To introduce its axioms we need the notion of subset bounded formula.
\begin{deff}\label{deltap} {\em
We use subset bounded quantifiers $\exists x\subseteq y \;\ldots$ and $\forall x\subseteq y\;\ldots$ as abbreviations
for $\exists x(x\subseteq y\und \ldots)$ and $\forall x(x\subseteq y\to \ldots)$, respectively.

The $\Delta_0^{\mathcal P}$-formulae are the smallest class of formulae containing the atomic formulae closed under $\wedge,\vee,\to,\neg$ and the quantifiers
  $$\forall x\in a,\;\exists x\in a,\; \forall x\subseteq a,\; \exists x\subseteq a.$$
  A formula is in $\Sigma^{\mathcal P}$ if belongs to the smallest collection of formulae which contains the $\Delta_0^{\mathcal P}$-formulae and is closed under $\wedge,\vee$ and the quantifiers
  $\forall x\in a,\;\exists x\in a,\; \forall x\subseteq a$ and  $\exists x$.  A formula is $\Pi^{\mathcal P}$ if belongs to the smallest collection of formulae which contains the $\Delta_0^{\mathcal P}$-formulae and is closed under $\wedge,\vee$, the quantifiers
  $\forall x\in a,\;\exists x\in a,\; \forall x\subseteq a$ and  $\forall x$.
}\end{deff}
\begin{deff}\label{KPP}{\em  $\KPP$ has the same language as $\ZF$. Its axioms are the following: Extensionality, Pairing, Union, Infinity, Powerset,
       $\Delta_0^{\mathcal P}$-Separation,    $\Delta_0^{\mathcal P}$-Collection and
       Set Induction (or Class Foundation).
       \\[1ex]
  The  transitive models of $\KPP$ have been termed {\bf power admissible} sets in \cite{friedaCount}.
}\end{deff}

\begin{rem}{\em
 Alternatively, $\KPP$ can be obtained from $\KP$ by adding  a function symbol $\PC$ for the  powerset function as a primitive symbol to the language and the axiom
$$\forall y\,[y\in \PC(x)\leftrightarrow y\subseteq x]$$
and extending the schemes of $\Delta_0$ Separation and Collection to the $\Delta_0$-formulae of this new language.}
\end{rem}

\begin{lem}\label{weaker}
$\KPP$ is {\bf not} the same theory as $\KP+\POW$, where $\POW$ denotes the Powerset Axiom. Indeed, $\KP+\POW$ is a much weaker theory than $\KPP$
 in which one cannot prove the existence of $V_{\omega+\omega}$.
 \end{lem}
 \begin{proof} \cite[Lemma 2.4]{powerKP}.\end{proof}

 \subsection{The infinitary proof system $\RSOPX$ }
 The infinitary proof system $\RSOP$ of \cite{powerKP} is based on a formal analogue of the von Neumann hierarchy along the Bachmann-Howard ordinal. For our purposes both have to be relativised to a given set $X$.
 \begin{definition}{\em
Let $X$ be any set. We may relativise the von Neumann hierarchy to $X$ as follows:
\begin{align*}
V_0(X)&:=TC(\{X\})\quad\text{the \textit{transitive closure} of $\{X\}$}\\
V_{\alpha+1}(X)&:=\{B\::B\subseteq V_\alpha (X)\}\\
V_\theta (X) &:=\displaystyle\bigcup_{\xi<\theta} V_\xi (X)\quad\text{when $\theta$ is a limit.}
\end{align*}
}\end{definition}
Let $X$ be an arbitrary (well founded) set and let $\theta$ be the set-theoretic rank of $X$ (hereby referred to as the $\in$-rank). Henceforth all ordinals are assumed to belong to the ordinal notation system  $T(\theta)$ developed in section 3. The system $\RSOPX$ will be the relativised version of the infinitary proof system $\RSOP$ from \cite{powerKP}.
\begin{definition}\label{level2}{\em
We give an inductive definition of the set $\mathcal{T}^{\mathcal P}$ of $\RSOPX$ terms. To each term $t\in\mathcal{T}^{\mathcal P}$ we assign an ordinal level $\lev t$.
\begin{description}
\item[(i)] For every $u\in TC(\{X\})$, $\bar{u}\in\mathcal{T}^{\mathcal P}$ and $\lev{\bar{u}}:=\Gamma_{\text{rank}(u)}$.
\item[(ii)]For every $\alpha<\Omega$, $\mathbb{V}_\alpha (X)\in\mathcal{T}^{\mathcal P}$ and $\lev{\mathbb{V}_\al(X)}:=\Gamma_{\theta+1} + \alpha$.
\item[(iii)]  For each { $\alpha<\Omega$}, we have infinitely many free variables $a_1^{\alpha},a_2^{\alpha},a_3^{\alpha},\ldots$ which  are
terms of level $\Gamma_{\theta+1}+\alpha$.
\item[(iv)]If  $\alpha<\Omega$, $A(a,b_1,\ldots,b_n)$ is a $\Delta_0^{\mathcal P}$ formula of $\KPP$ with all free variables displayed and $s_1,\ldots,s_n$ are terms in $\mathcal{T}^{\mathcal P}$ then
\begin{equation*}
[x\in\mathbb{V}_\alpha (X) | A(x,s_1,\ldots,s_n)]
\end{equation*}
is a term of level $\Gamma_{\theta+1} + \alpha$.
\end{description}
The {\it $\RSOPX$--formulae} are the expressions of the form
$F(s_1,\ldots,s_n)$, where $F(a_1,\ldots,a_n)$ is a
formula of $\KPP$ with all free variables exhibited and $s_1,\ldots,s_n$ are $\RSOPX$-terms.
We set \begin{eqnarray*}
\lev{F(s_1,\ldots,s_n)}&=&\{\lev{s_1},\ldots,\lev{s_n}\}.
%,\\
%\ko(F(s_1,\ldots,s_n))&=&\bigcup_{1\leq i\leq n}\ko(s_i).
\end{eqnarray*}
For a sequent $\Gamma=\{A_1,\ldots,A_n\}$ we define
\begin{eqnarray*}\lev{\Gamma} & :=& \lev{A_1}\cup\ldots\cup\lev{A_n}\,.
%\\ \ko(\Gamma) & :=& \ko(A_1)\cup\ldots\cup\ko(A_n).
\end{eqnarray*}
%If $F[a_1,\ldots,a_n]$ is in $\Deltaop$, $F(s_1,\ldots,s_n)$ is also called a
A formula is a {\bf $\Deltaop$-formula} of $\RSOPX$ if it is of the form $F(s_1,\ldots,s_n)$ with
$F(a_1,\ldots,a_n)$ being a $\Deltaop$-formula of $\KPP$ and $s_1,\ldots,s_n$ $\RSOPX$-terms.

As in the case of the Tait-style version of $\KPP$ in \cite[Sec. 3]{powerKP}, we let $\neg A$ be the formula which arises from $A$
by (i)  putting $\neg$ in front of each atomic formula, (ii) replacing
$\wedge, \vee,(\forall x \!\in\! s), (\exists x \!\in\! s),(\forall x\subseteq s),(\exists x\subseteq s), \forall x,\exists x$ by
$\vee, \wedge, (\exists x \!\in\! s), (\forall x \!\in\! s),(\exists x\subseteq s),(\forall x\subseteq s),\exists x,\forall x$, respectively, and
(iii) dropping double negations. $A\to B$ stands for $\neg A\,\vee\,B$.

}\end{definition}

\begin{remark}{\em There is a crucial difference between Definition \ref{level1} and Definition \ref{level2} when it comes to measuring the level of a comprehension term. The level of $[x\in\mathbb{V}_\alpha (X) | A(x,s_1,\ldots,s_n)]$ does not take the terms $s_1,\ldots,s_n$ into account. They may be of arbitrary (especially higher) level.
}\end{remark}

Since we also want to keep track of the complexity of cuts appearing in
derivations, we endow each formula with  an ordinal rank.
\begin{deff}{\em The {\it rank } of a term or formula is determined as
follows.}
\begin{enumerate}
%\item $rk(\Va):=\om\cdot\al$.
%\item $rk([x\In\Va:F(x)]):=max\{\om\cdot\al+1,rk(F(\Vb{0}))+2\}$.
\item $rk(\bar{u}):=\Gamma_{rank(u)}$ for $u$ in the transitive closure of $X$.
%\\[1ex]
\item $rk({\mathbb V}_{\alpha}(X)):=\Gamma_{\theta+1}+\omega\cdot\alpha$.
%\\[1ex]
\item $rk([x\In{\mathbb V}_{\alpha}(X)\mid F(x)]):=\max\{\Gamma_{\theta+1}+\om\cdot\al+1,rk(F(\bar{0}))+2\}$.
%\\[1ex]
\item $rk(s\In t):=rk(s\nIn t):=\max\{\lev s +6,\lev t +1\}$.
%\\[1ex]
%\item $rk((\exists x\In t)F(x)):=rk((\forall x\In t)F(x)):=
%max\{ rk( t),rk(F(\bar{0}))+2\}$ if $t$ is not of the form $\bar{u}$.
\item $rk((\exists x\In t)F(x)):=rk((\forall x\In t)F(x)):=
max\{ rk( t)+3,rk(F(\bar{0}))+2\}$. % if $t$ is not of the form $\bar{u}$.
%\\[1ex]
\item $rk((\exists x\subb t)F(x)):=rk((\forall x\subb t)F(x)):=
\max\{rk( t) +3,rk(F(\bar{0}))+2\}.$
%\\[1ex]
\item $rk(\exists x\,F(x)):=rk(\forall x\,F(x)):=
\max\{\Omega,rk(F(\bar{0}))+2\}.$
\\[1ex]
\item $rk(A\wedge B):=rk(A\vee B):=max\{rk(A),rk(B)\}+1.$\end{enumerate}
\end{deff}

\begin{deff}{\em
The {\em axioms} of $\RSOPX$ are:
\begin{itemize}
\item[(X1)] $\Gamma,\, \bar{u}\in\bar{v}$ $\;$ if $u,v\in TC(X)$ and $u\in v$.
\item[(X2)] $\Gamma,\, \bar{u}\notin\bar{v}$ $\;$ if $u,v\in TC(X)$ and $u\notin v$.
 \item[(A1)] $\Gamma,\,A,\,\neg A$ for $A$ in $\Deltaop$.
%\\[1ex]
\item[(A2)] $\Gamma,\,t=t$.
%\\[1ex]
\item[(A3)] $\Gamma,\, s_1\ne t_1,\ldots,s_n\ne t_n,\neg A(s_1,\ldots,s_n),\,A(t_1,\ldots,t_n)$\\[1ex]
 for $A(s_1,\ldots,s_n)$ in $\Deltaop$.
 % \\[1ex]
\item[(A4)]  $\Gamma,\, s\in \Va(X)$ if $\lev{s}<\lev{{\mathbb V}_{\alpha}(X)}$.
%\\[1ex]
\item[(A5)]  $\Gamma,\, s\subseteq \Va(X)$ if $\lev{s}\leq\lev{{\mathbb V}_{\alpha}(X)}$.
%\\[1ex]
\item[(A6)] $\Gamma, t\notin [x\in \Va(X)\mid F(x,\vec s\,)],\, F(t, \vec s\,)$ \\[1ex] whenever
$F(t,\vec s\,)$ is $\Deltaop$ and
$\lev{t}<\lev{{\mathbb V}_{\alpha}(X)}$.
%\\[1ex]
\item[(A7)] $\Gamma, \neg F(t, \vec s\,),\, t\in [x\in \Va(X)\mid F(x,\vec s\,) ] $\\[1ex]  whenever  $F(t,\vec s\,)$ is $\Deltaop$ and $\lev{t}<\lev{{\mathbb V}_{\alpha}(X)}$.%\\[1ex]
  %  \item[] The {\em inference rules} of $\RSOP$ are:
 \end{itemize}
 }\end{deff}

We adopt the notion of operator from Definition \ref{operator}. If $s$ is an $\RSOPX$-term, the operator { ${\mathcal H}[s]$}
is defined by
\begin{eqnarray*} {\mathcal H}[s](X) &=& {\mathcal H}(X\,\cup\,\{\lev{s}\}).\end{eqnarray*}
Likewise, if $\mathfrak X$ is a formula or a sequent we define
\begin{eqnarray*} {\mathcal H}[{\mathfrak X}](X) &=& {\mathcal H}(X\,\cup\,\lev{\mathfrak X}\,).\end{eqnarray*}

\begin{deff}{\em Let $\CH$ be an operator and let $\Lambda$ be a finite
set of $\RSOPX$--formulae.
$\provx{\CH}{\al}{\rho}{\Lambda}$ is defined by recursion on $\al$.

If $\Lambda$ is an {\bf axiom} and $\lev\Lambda\,\cup\,\{\alpha\}\subseteq {\mathcal H}(\emptyset)$, then
 $\provx{\CH}{\al}{\rho}{\Lambda}$.

 Moreover, we have inductive clauses pertaining to the inference rules of $\RSOPX$,
 which all come with the additional requirement that
$$\lev \Lambda\,\cup\,\{\alpha\}\subseteq {\mathcal H}(\emptyset)$$ where $\Lambda$ is the sequent of the conclusion.
We shall not repeat this requirement below.

Below the third column gives the requirements that the ordinals have to satisfy for each of the inferences.
For instance in the case of $(\forall)_{\infty}$, to be able to conclude that $\provx{\mathcal H} {\alpha}
{\rho} {\GA, \forall x F(x)}$, it is required that for all terms $s$ there exists $\alpha_s$
such that $\provx{{\mathcal H}[s]} {\alpha_{s}}
{\rho} {\GA,  F(s)}$ and $\lev s<\alpha_s+1<\alpha$. The side conditions for the rules $(b\forall)_{\infty},(pb\forall)_{\infty},
(\not\in)_{\infty} ,(\not\subb)_{\infty}$ below have to be read in the same vein.

Below we shall write $\lev s \kld \lev t$ and $\lev s \kldg \lev t$ for $\lev s < \max(\Gamma_{\theta+1},\lev t)$ and
$\lev s \leq \max(\Gamma_{\theta+1},\lev t)$, respectively.

The clauses are the following:

$$ \begin{array}{lcr}
(\wedge) & \ifthree{\provx{\mathcal H} {\alpha_0}{\rho}  {\GA,
A_0}}{\provx{\mathcal H} {\alpha_0} {\rho}{\GA, A_1}}{\provx
{\mathcal H} {\alpha}{\rho} {\GA,A_0\wedge A_1}}
&\begin{array}{r}\alpha_{0} < \alpha
\end{array} \\[0.6cm]
(\vee) & \infone{\mathcal H} {\alpha_0} {\rho}
{\Lambda,A_{i}} {\mathcal H} {\alpha} {\rho} {\GA,A_0\vee A_1}
&\begin{array}{r}\alpha_{0} < \alpha\\ i\in\{0,1\}
\end{array} \\[0.6cm]
\Cut & \ifthree{\provx{\mathcal H} {\alpha_0}{\rho}  {\Lambda,
B}}{\provx{\mathcal H} {\alpha_0} {\rho}{\Lambda, \neg B}}{\provx
{\mathcal H} {\alpha}{\rho} {\Lambda}}
&\begin{array}{r}\alpha_{0} < \alpha\\
rk(B)<\rho\end{array} \\[0.6cm]
(b\forall)_{\infty} & \infone{ {\mathcal H}[s]} {\alpha_{s}}
{\rho} {\GA, s\In t\to F(s)\mbox{ for all }\lev s<\lev t}
 {\mathcal H} {\alpha} {\rho} {\Gamma,(\forall x\In t)F(x)}
&\lev{s}\leq\alpha_{s} < \alpha
\\[0.6cm]
(b\exists) & \infone{\mathcal H} {\alpha_0} {\rho}
{\GA,s\In t\wedge F(s)} {\mathcal H} {\alpha} {\rho} {\GA,(\exists x\In t)F(x)}
&\begin{array}{r}\alpha_{0} < \alpha\\ \lev s <\lev t \\ \lev{s}<\al
\end{array} \\[0.6cm]
(pb\forall)_{\infty} & \infone{ {\mathcal H}[s]} {\alpha_{s}}
{\rho} {\GA, s\subb t\to F(s)\mbox{ for all }\lev s\kldg\lev t}
 {\mathcal H} {\alpha} {\rho} {\Gamma,(\forall x\subb t)F(x)}
&\lev{s}\leq\alpha_{s} < \alpha \\[0.6cm]
(pb\exists) & \infone{\mathcal H} {\alpha_0} {\rho}
{\GA,s\subb t\wedge F(s)} {\mathcal H} {\alpha} {\rho} {\GA,(\exists x\subb t)F(x)}
&\begin{array}{r}\alpha_{0} < \alpha\\ \lev s\kldg \lev t \\ \lev{s}<\al
\end{array}
\\[0.6cm]
%\end{array}$$
%$$ \begin{array}{lcr}
(\forall)_{\infty} & \infone{ {\mathcal H}[s]} {\alpha_{s}}
{\rho} {\GA,  F(s)\mbox{ for all } s}
 {\mathcal H} {\alpha} {\rho} {\Gamma,\forall x F(x)}
&\lev{s}<\alpha_{s}+1 < \alpha \\[0.6cm]
(\exists) & \infone{\mathcal H} {\alpha_0} {\rho}
{\GA, F(s)} {\mathcal H} {\alpha} {\rho} {\GA,\exists x F(x)}
&\begin{array}{r}\alpha_{0}+1 < \alpha\\  \lev{s}<\al
\end{array} %\\[0.6cm]
\end{array}$$
$$ \begin{array}{lcr}
(\not\in)_{\infty} & \infone{ {\mathcal H}[r]} {\alpha_{r}}
{\rho} {\GA, r\In t\to r\ne s\mbox{ for all }\lev r<\lev t}
 {\mathcal H} {\alpha} {\rho} {\Gamma,s\not\in t}
&\lev{r}\leq\alpha_{r} < \alpha \\[0.6cm]
(\in) & \infone{\mathcal H} {\alpha_0} {\rho}
{\GA,r\In t\wedge r=s} {\mathcal H} {\alpha} {\rho} {\GA,s\In t}
&\begin{array}{r}\alpha_{0} < \alpha\\ \lev r <\lev t \\ \lev{r}<\al
\end{array} \\[0.6cm]
(\not\subb)_{\infty} & \infone{ {\mathcal H}[r]} {\alpha_{r}}
{\rho} {\GA, r\subb t\to r\ne s\mbox{ for all }\lev r\kldg\lev t}
 {\mathcal H} {\alpha} {\rho} {\Gamma,s\not\subb t}
&\lev{r}\leq\alpha_{r} < \alpha \\[0.6cm]
(\subb) & \infone{\mathcal H} {\alpha_0} {\rho}
{\GA,r\subb t\wedge r=s} {\mathcal H} {\alpha} {\rho} {\GA,s\subb t}
&\begin{array}{r}\alpha_{0} < \alpha\\ \lev r\kldg \lev t \\ \lev{r}<\al
\end{array} \\[0.6cm]

\SRP &
\infone{\CH}{\al_0}{\rho}{\Gamma,A}{\CH}{\al}{\rho}{\Gamma,
\exists z\,A^z}
&\begin{array}{r} \al_0+1,\Omega<\al\\
 A\In\Sigmap\end{array}
\end{array}$$
}\end{deff}

\begin{rem}{\em Suppose $\provx{\CH}{\alpha}{\rho}{\Gamma(s_1,\ldots,s_n)}$ where $\Gamma(a_1,\ldots,a_n)$ is a
sequent of $\KPP$ such that all variables $a_1,\ldots,a_n$ do occur in $\Gamma(a_1,\ldots,a_n)$ and $s_1,\ldots,s_n$ are $\RSOPX$-terms.
Then we have that $\lev{s_1},\ldots,\lev{s_n}\in \CH(\emptyset)$. Standing in sharp contrast to the ordinal analysis of $\KP$, however,
 the terms $s_i$ may and often will contain subterms that the operator $\CH$ does {\bf not} control, that is, subterms $t$ with $\lev t\not\in\CH(\emptyset)$.
}\end{rem}

The embedding of $\KPP$ into $\RSOPX$ and the ordinal analysis of $\RSOPX$ can be carried in much the same way
as for $\RSOP$ in \cite{powerKP} with only minor amendments necessary to deal with terms and axioms pertaining to the given set $X$.
Below we list the main steps.

\begin{theorem}\label{xembedp}{\em
If $\KPP\vdash\Gamma(a_1,\ldots,a_n)$ where $\Gamma(a_1,\ldots,a_n)$ is a finite set of formulae whose free variables are amongst $a_1,\ldots,a_n$, then there is some $m<\omega$ (which we may compute from the derivation) such that
\begin{equation*}
\provx{\mathcal{H}[s_1,\ldots,s_n]}{\Omega\cdot\omega^m}{\Omega+m}{\Gamma(s_1,\ldots,s_n)}
\end{equation*}
for any operator $\CH$ and any $\RSOPX$ terms $s_1,\ldots,s_n$.
}\end{theorem}
\begin{proof} This can be proved in the same way as \cite[Theorem 6.9]{powerKP}.
\end{proof}

\begin{thm}[Cut elimination I]\label{xpredcep}
\begin{eqnarray*} \provx{{\mathcal H}}{\alpha}{\Omega+n+1}{\Gamma} &\Rightarrow &
\provx{{\mathcal H}}{\omega_n(\alpha)}{\Omega+1}{\Gamma}
\end{eqnarray*} where $\omega_0(\beta):=\beta$ and
$\omega_{k+1}(\beta):=\omega^{\omega_k (\beta)}$.
\end{thm}
\prf The proof is the special case of Theorem \ref{xpredce} when $\rho=\Omega+n$ and $\alpha=0$. See also
\cite[Theorem 7.1]{powerKP}. \qed
For a formula $C$ of $\RSOPX$, $C^{\mathbb{V}_{\delta}(X)}$ is obtained from $C$ by replacing all unbounded quantifiers $Qz$ in $C$ by $(Qz\in \mathbb{V}_{\delta}(X))$.
\begin{lemma}[Boundedness for $\RSOPX$]\label{xboundednessp}{\em
If $C$ is a $\Sigma^{\mathcal P} $ formula, $\alpha\leq\beta<\Omega$, $\beta\in\mathcal{H}$ and $\provx{\mathcal{H}}{\alpha}{\rho}{\Gamma,C}$ then $\provx{\mathcal{H}}{\alpha}{\rho}{\Gamma,C^{\mathbb{V}_\beta (X)}}$.
}\end{lemma}
\begin{proof} Similar to Lemma \ref{xboundednessp}.\end{proof}

\begin{theorem}[Collapsing for $\RSOPX$]\label{ximpredcep}{\em Suppose $\Gamma$ is a set of $\Sigma^{\mathcal P}$ formulae such that $\lev{\Gamma}\subseteq B(\eta)$ and $\eta\in B(\eta)$.
\begin{equation*}
\text{If}\quad\provx{\mathcal{H}_\eta}{\alpha}{\Omega+1}{\Gamma}\quad\text{then}\quad\provx{\mathcal{H}_{\hat\alpha}}{\psi\hat\alpha}{\psi\hat\alpha}{\Gamma}
\end{equation*} where $\hat{\alpha}=\eta+\omega^{\Omega+\alpha}$.
}\end{theorem}
\begin{proof} The proof is essentially the same as that of \cite[Theorem 7.4]{powerKP}.\end{proof}

For the characterisation theorem for $\KPP$, we need to show that derivability in $\RSOPX$ entails
truth for $\Sigmaop$-formulae. Since $\RSOPX$-formulae contain variables we need the notion of assignment.
Let $VAR$ be the set of free variables of $\RSOPX$.
A variable assignment $\su$ is a function
$${\su}:VAR\longrightarrow V_{\psi(\varepsilon_{\Omega+1})}$$
satisfying ${\su}(a^{\alpha})\in V_{\alpha+1}(X)$.
$\su$ can be canonically lifted to all $\RSOPX$-terms as follows:
\begin{eqnarray*} \su(\bar{u}) &=& u\mbox{ for $u$ in $TC(\{X\})$} \\
\su(\Vb{\alpha}(X)) &=& V_{\alpha}(X)\\
  \su([x\in \Vb{\alpha}(X)\mid F(x,s_1,\ldots,s_n)]) &=& \{x\in V_{\alpha}(X)\,:\, F(x,\su(s_1),\ldots,\su(s_n))\}\,.
  \end{eqnarray*}
  Note that $\su(s)\in V_{\psi(\varepsilon_{\Omega+1})}(X)$ holds for all $\RSOPX$-terms $s$.
  Moreover, we have $\su(s)\in V_{\lev s+1}(X)$.

  \begin{thm}[Soundness]\label{sound2} Let $\CH$ be an operator with $\CH(\emptyset)\subseteq B(\varepsilon_{\Omega+1})$ and $\alpha,\rho<\psi(\varepsilon_{\Omega+1})$.
  Let $\Gamma(s_1,\ldots,s_n)$ be a sequent consisting only of $\Sigmaop$-formulae with constants from $TC(\{X\})$. Suppose
  $$\provx{\CH}{\alpha}{\rho}{\Gamma(s_1,\ldots,s_n)}\,.$$
  Then, for all   variable assignments $\su$,
  $$ V_{\psi(\varepsilon_{\Omega+1})}(X)\models
  \Gamma(\su(s_1),\ldots,\su(s_n))\,,$$
  where the latter, of course, means that $V_{\psi(\varepsilon_{\Omega+1})}$ is a model of the disjunction
  of the formulae in $\Gamma(\su(s_1),\ldots,\su(s_n))$.

  \end{thm}
  \prf The proof is basically the same as for \cite[Theorem 8.1]{powerKP}. It proceeds by induction on $\alpha$. Note that, owing to $\alpha,\rho<\Omega$, the proof tree pertaining to
  $ \provx{\CH}{\alpha}{\rho}{\Gamma(s_1,\ldots,s_n)}$ neither contains any instances of $\SRP$ nor of $(\forall)_{\infty}$, and that all cuts are performed with
  $\Deltaop$-formulae. The proof is straightforward as all the axioms of $\RSOP$ are true under the interpretation and all other rules are
  truth preserving with respect to this interpretation. Observe that we make essential use of the free variables when showing
  the soundness of $(b\forall)_{\infty}$, $(pb\forall)_{\infty}$, $(\not\in)_{\infty}$ and $(\not\subb)_{\infty}$.
  We treat $(pb\forall)_{\infty}$ as an example. So assume $(\forall x\subseteq s_i)F(x,\vec s\,)\in\Gamma(\vec s\,)$ and $$\provx{\CH[r]}{\alpha_r}{\rho}{\Gamma(s_1,\ldots,s_n),r\subseteq s_i\to F(r,\vec s\,)}$$
  holds for all terms $r$ with $\lev r\leq \lev {s_i}$ for some $\alpha_r<\alpha$.
   In particular we have $$\provx{\CH[a^{\beta}]}{\alpha'}{\rho}{\Gamma(s_1,\ldots,s_n),a^{\beta}\subseteq s_i\to F(a^{\beta},\vec s\,)}$$ where $\beta=\lev {s_i}$ and $a^{\beta}$ is a free variable not occurring
   in $\Gamma(s_1,\ldots,s_n)$  and $\alpha'=\alpha_{a^{\beta}}$.
    By the induction hypothesis we have
    $$ V_{\psioi(\varepsilon_{\Omega+1})}\models
  \Gamma(\su(s_1),\ldots,\su(s_n)), \su'(a^{\beta})\subseteq \su(s_i)\to F(\su'(a^{\beta}),\su(s_1),\ldots,\su(s_n)\,) $$ where $\su'$ is an arbitrary variable assignment. This entails that either
  $$ V_{\psioi(\varepsilon_{\Omega+1})}\models
  \Gamma(\su(s_1),\ldots,\su(s_n))$$ or  $$ V_{\psioi(\varepsilon_{\Omega+1})}\models\su'(a^{\beta})\subseteq \su(s_i)\to F(\su'(a^{\beta}),\su(s_1),\ldots,\su(s_n)\,)$$ for all assignments $\su'$. In the former case we have found what we want and in the latter case we arrive at $ V_{\psioi(\varepsilon_{\Omega+1})}\models (\forall x\subseteq \su(s_i)) F(x,\su(s_1),\ldots,\su(s_n)\,)$ and therefore also have $ V_{\psioi(\varepsilon_{\Omega+1})}\models
  \Gamma(\su(s_1),\ldots,\su(s_n))$.
     \qed

\subsection{The provably total set functions of $\KPP$}
For each $n<\omega$ we define the following recursive set function
\begin{equation*}
G_n^{\mathcal P}(X):=V_{\psi_{\theta}(e_n)}(X)
\end{equation*}
where $e_n$ was defined in (\ref{en}) and $\theta$ stands for the rank of the transitive closure of $X$.

\begin{theorem}\label{xconcp}{\em Suppose $f$ is a set function that is $\Sigma^{\mathcal P}$ definable in $\KPP$, then there is some $n$ (which we may compute from the finite derivation) such that
\begin{equation*}
V\models\forall x(f(x)\in G_n^{\mathcal P}(x)).
\end{equation*}
Moreover $G_m^{\mathcal P}$ is $\Sigma^{\mathcal P}$ definable in $\KPP$ for each $m<\omega$.
}\end{theorem}
\begin{proof}
Let $A_f(a,b)$ be the $\Sigma^{\mathcal P}$ formula expressing $f$ such that $\KPP\vdash\forall x\exists!yA_f(x,y)$ and fix an arbitrary set $X$. Let $\theta$ be the rank of $X$. Applying Theorem \ref{xembedp} we can compute some $k<\omega$ such that
\begin{equation*}
\CH_0\;\prov{\OO\cdot\omega^k}{\OO+k}{\forall x\exists!yA_f(x,y)}.
\end{equation*}
Applying inversion as in Lemma \ref{xinversion} iv) twice we get
\begin{equation*}
\CH_0\;\prov{\OO\cdot\omega^k}{\OO+k}{\exists yA_f(X,y)}.
\end{equation*}
Applying Theorem \ref{xpredcep}  we get
\begin{equation*}
\CH_0\;\prov{e_{k+1}}{\OO+1}{\exists yA_f(X,y)}.
\end{equation*}
Now by Theorem \ref{ximpredcep} (collapsing) we have
\begin{equation*}
\CH_{e_{k+2}}\;\prov{\psi_{\theta}(e_{k+2})}{\psi_{\theta}(e_{k+2})}{\exists yA_f(X,y)}.
\end{equation*}
Now by Lemma \ref{xboundednessp} (boundedness) we obtain
\begin{eqnarray}\label{M1}
&&\CH_\gamma\;\prov{\psi_{\theta} (\gamma)}{\psi_{\theta} (\gamma)}{(\exists y\in\mathbb{V}_{\psi_{\theta}(\gamma)}(X))A_f(X,y)^{\mathbb{V}_{\psi_{\theta}(\gamma)}}}\quad\text{where $\gamma:=e_{k+2}$.}
\end{eqnarray}
The Soundness Theorem \ref{sound2} applied to (\ref{M1}) now yields that
$$\mathbb{V}_{\psi_{\theta}(\gamma)}\models \exists y\,A_f(X,y).$$
It remains to note that $V_\al(X)\subseteq G_{k+3}^{\mathcal P}(X)$ to complete this direction of the proof.\\

\noindent For the other direction we argue informally in $\KPP$. Let $X$ be an arbitrary set.
 By Theorem \ref{WO} we can find an ordinal of the same order type as $\psi_{\theta}(e_n)$. We can now generate $V_{\psi_{\theta}(e_n)}(X)$ by $\Sigma^{\mathcal P}$-recursion (similar to \cite{ba} p. 26 Theorem 6.4).
\end{proof}

\begin{remark}{\em As was the case for $\KP$, the first part of \ref{xconc} can be carried out {\em inside $\KPP$}, i.e. If $f$ is $\Sigma^{\mathcal P}$ definable in $\KPP$ then we can compute some $n$ such that $$\KPP\vdash\forall x(\exists!y\in G_n^{\mathcal P}(x))A_f(x,y)\,.$$ This is not immediately obvious since it appears we need induction up to $\psi_{\theta}(\varepsilon_{\OO +1})$, which we do not have access to in $\KPP$. The way to get around this is to note that we could, in fact, have managed with an infinitary system based on an ordinal representation built out of $B_\theta(e_m)$, provided $m$ is high enough, and we may compute how high $m$ needs to be from the finite derivation. We do have access to induction up to $\psi_{\theta}(e_m)$ in $\KPP$ by Theorem \ref{WO}.
}\end{remark}

\section{Adding global choice: $\KPP+\GAC$}
Here we extend the relativised ordinal analysis to $\KPP$ with global choice.
Since the global axiom of choice, $\GAC$, is less familiar, let us spell out the details. By $\KPP+\GAC$ we mean an extension of $\KPP$ where the language contains a new binary relation symbol $\RR$ and  the axiom schemes of
$\KPP$ are extended to this richer language and
the following axioms pertaining to $\RR$ are added:
\begin{eqnarray}\label{RR} (i) && \forall x\forall y\forall z[\RR(x,y)\wedge\RR(x,z)\to y=z]\\
    (ii) && \forall x[ x\ne \emptyset \to \exists y\in x\,\RR(x,y)].
    \end{eqnarray}
 Section 3 of \cite{powerKPAC} describes an extension of $\RSOP$ that incorporates the new symbol $\RR$. We can now relativise this system to a given set $X$ as we did with $\RSOP$ in the previous section. Let us call the relativized version
 $\RSOPXR$.    The ordinal analysis of  $\RSOPXR$ can be performed with almost no changes as for $\RSOPX$ in the foregoing section.
 On account of the relativization we arrive at stronger versions of \cite[Corollary 3.1]{powerKPAC} and \cite[Theorem]{powerKPAC}
 which incorporate the parameter $X$. A $\Pi_2^{\mathcal P}$-formula is a formula of the form $\forall y\,A(y)$ with $A(y)$  in $\Sigma^{\mathcal P}$.
 \begin{thm} Let $B$ be $\Pi_2^{\mathcal P}$-sentence of the language without the predicate $\RR$. If $\KPP+\GAC\vdash B$, then $\KPP+\AC\vdash B$.
 \end{thm}
 \begin{proof} Basically as in \cite[Theorem 3.2]{powerKPAC}. \end{proof}
 The acronym $\CZF$ stands for Constructive Zermelo-Fraenkel set theory. For details see \cite{mar,book}.

 \begin{cor}\label{endbericht}\begin{itemize}
 \item[(i)] $\KPP+\GAC$, $\KPP+\AC$, and $\CZF+\AC$ prove the same $\Pi_2^{\mathcal P}$-sentences.

\item[(ii)] The three theories are of the same proof-theoretic strength as
 $\KPP$. More precisely, they prove the same $\Pi^1_4$-sentences of the language of second order arithmetic when
 identified with their canonical translation into the language of set theory.
 \end{itemize}
 \end{cor}
 \begin{proof} (i) For $\KPP+\GAC$ and $\KPP+\AC$ this follows from the foregoing Theorem.
 A question left open in \cite{acend} was that of the strength of constructive Zermelo-Fraenkel set theory with the axiom of choice. There $\CZF+\AC$ was interpreted in $\KPP+V=L$ (\cite[Theorem 3.5]{acend}). However, the realizability interpretation
works with $\GAC$ as well. Moreover, for this notion of realizability, realizability of a $\Pi_2^{\mathcal P}$-sentence $B$ entails its truth. Therefore if $\CZF+\AC\vdash B$, then $\KPP+\GAC\vdash B$.

 Conversely note that $\CZF+\AC$ proves  the law of excluded middle for $\Delta_0^{\mathcal P}$-formulae. This amount of classical
  logic suffices to prove the power set axiom from the subset collection axiom.
The proof-theoretic ordinal of $\CZF$ is also the Bachmann-Howard ordinal.  Moreover, in Theorem \ref{WO}, $\KP$ can be replaced by
$\CZF+\AC$.   As a result, the ordinal analysis for $B$ utilizing $\RSOPXR$, can be carried out in $\CZF+\AC$ itself and the proof of the pertaining soundness is also formalisable in $\CZF+\AC$,
 whence the latter theory proves $B$.

(ii) follows from (i) viewed in conjunction with \cite[Corollary 3.5]{repl}.
  \end{proof}

Finally, we remark that  the three theories of Corollary \ref{endbericht}  can be added to the list of proof-theoretically equivalent theories  presented in \cite[Theorem 15.1]{ml}.

\section{The provably total set functions of other theories}
Part of the machinery developed here could also be used to give a characterization of the total set functions of extensions of $\KP$ such as the theories  $\mathbf{KPi}$ and $\mathbf{KPM}$ that are describing a recursively inaccessible and a recursively Mahlo universe of sets, respectively (see \cite{jp,r91,mi99realm}). This however would also require an interpretation of collapsing functions as acting on set-theoretic ordinals along the lines of \cite{mi93,r94a}.

\paragraph{Acknowledgement}
Part of the material is based upon research supported by the EPSRC of the UK through grant No. EP/K023128/1. This research was also supported by a Leverhulme Research Fellowship and a Marie Curie International Research Staff Exchange
 Scheme Fellowship within the 7th European Community Framework Programme.
 This publication was also made possible through the support of a
grant from the John Templeton Foundation.
% The opinions expressed in this
%publication are those of the authors and do not necessarily reflect the
%views of the John Templeton Foundation.


\begin{thebibliography}{7}
\bibitem{mar} P. Aczel, M. Rathjen: {\em Notes on constructive set theory}. Technical Report 40,
Institut Mittag-Leffler (The Royal Swedish Academy of Sciences,Stockholm,2001).
{\em http://www.ml.kva.se/preprints/archive2000-2001.php}

\bibitem{book} P. Aczel, M. Rathjen: {\em Constructive set theory}. book draft,
August 2010.

\bibitem{ba} J Barwise: {\em Admissible Sets and Structures.}
(Springer, Berlin 1975).

\bibitem{beeson} M.~Beeson: {\em Foundations of Constructive
Mathematics.}  Springer, Berlin (1985).

\bibitem{blankertz-w} B. Blankertz, W. Weiermann: {\em A uniform approach for characterizing the provably total number-theoretic functions of $KPM$ and (some of) its subsystems.} Preprint 1994.

    \bibitem{blankertz} B. Blankertz: {\em Beweistheoretische Techniken zur Bestimmung der $\Pi^0_2$-Skolem Funktionen}. Dissertation, Westf\"alische Wilhelms-Universit\"at (M\"unster, 1997).

        \bibitem{blankertz-w2} B. Blankertz, W. Weiermann: {\em How to chracterize provably total functions by the Buchholz operator method}. Lecture Notes in Logic 6 (Springer, Heidelberg, 1996).


\bibitem{bu81} W. Buchholz, S. Feferman, W. Pohlers, W. Sieg: {\em Iterated inductive definitions and subsystems of analysis.}
(Springer, Berlin, 1981).

\bibitem{bu86} W. Buchholz: {\em A new system of proof-theoretic ordinal functions.} Annals of
Pure and Applied Logic
32 (1986) 195-207

\bibitem{stan-PA} W. Buchholz, S. Wainer {\em Provably computable functions and the fast growing hierarchy.}
in: Logic and Combinatorics, Contemporary Mathematics 65 (American Mathematical Society, Providence,1987) 179-198.

\bibitem{bu93} W. Buchholz: {\em A simplified version
of local predicativity}. in: P. Aczel, H. Simmons, S. Wainer (eds.), {\em Leeds
Proof Theory 90} (Cambridge University Press, Cambridge, 1993)
115-147.

\bibitem{buch-not} W. Buchholz: {\em Notation systems for infinitary derivations.}
Arch. Math. Logic 30 (1991) 277--296.


\bibitem{fef64} S. Feferman: {\em Systems of predicative analysis.} Journal of Symbolic Logic 29 (1964) 1-30.

\bibitem{fef66} S. Feferman: {\em predicative provability is set theory.} American Mathematical Society, Volume 72, Number 3 (1966), 486-489.

\bibitem{fef68} S. Feferman: {\em Systems of predicative analysis II. Representations of ordinals}. Journal of
Symbolic Logic 33 (1968) 193-220.

\bibitem{fef87} S. Feferman: {\em Proof theory: a personal report.} in: G. Takeuti {\em Proof Theory}, 2nd edition (North-Holland, Amsterdam, 1987) 445-485.

    \bibitem{F1} S. Feferman: {\em On the strength of some semi-constructive theories}.
In: U.Berger, P. Schuster, M. Seisenberger (Eds.): {\em Logic, Construction, Computation} (Ontos Verlag, Frankfurt, 2012) 201--225.

\bibitem{F2} S. Feferman: {\em
Is the continuum hypothesis a definite mathematical
problem?} Draft of paper for the lecture to the Philosophy Dept., Harvard University, Oct. 5, 2011 in the {\em Exploring the Frontiers of Incompleteness} project series,
Havard 2011--2012. %Conjecture S. 26 Slides 55-56

\bibitem{F3} S. Feferman: {\em
Three Problems for Mathematics: Lecture 2: Is the Continuum Hypothesis a definite mathematical problem?} Slides for inaugural Paul Bernays Lectures, ETH, Z\"urich, Sept. 12, 2012. %Conjecture S. 39

\bibitem{F4} S. Feferman: {\em Why isn't the Continuum Problem on the Millennium ($\$$1,000,000) Prize list?} Slides for CSLI Workshop on Logic, Rationality and Intelligent Interaction, Stanford, June 1, 2013.

    \bibitem{frege} G. Frege: {\em Die Grundlagen der Arithmetik}. (Verlag Wilhelm Koebner, Breslau, 1884).



%\bibitem{frsc} H.~Friedman, S.~\v{S}\v{c}edrov: {\em Set existence property for intuitionistic
%theories with dependent  choice.} Annals of Pure and Applied Logic 25 (1983) 129--140.

\bibitem{friedaCount} H. Friedman: {\em Countable models of set theories.}
In: A. Mathias and H. Rogers (eds.): {\em Cambridge Summer School in Mathematical Logic},
volume 337, Lectures Notes in Mathematics (Springer, Berlin, 1973) 539-573.

%\bibitem{fs} H.Friedman, S.~\v{S}\v{c}edrov: {\em The lack of definable witnesses and
%provably recursive functions in intuitionistic set theory}.
%Advances in Mathematics 57 (1985) 1-13.

\bibitem{fs} H. Friedman and S. Sheard: {\sl Elementary descent recursion
and proof theory}, Annals of Pure and Applied Logic 71 (1995) 1--45.

\bibitem{gentzen} G. Gentzen: {\em Die Widerspruchsfreiheit der reinen Zahlentheorie.}
Mathematische Annalen 112 (1936) 493-565

\bibitem{gentzen38} G. Gentzen: {\em Neue Fassung des Widerspruchsfreiheitsbeweises f\"ur die reine Zahlentheorie}, Forschungen zur Logik und zur Grundlegung der exacten Wissenschaften, Neue Folge 4 (Hirzel, Leipzig, 1938) 19--44.

    \bibitem{gentzen43} G. Gentzen: {\em Beweisbarkeit und Unbeweisbarkeit von Anfangsf\"allen der transfiniten Induktion in der reinen Zahlentheorie}. Mathematische Annalen 119 (1943) 140--161.




\bibitem{j80}G. J\"ager: {\em  Beweistheorie von KPN}. Archiv
f. Math. Logik 2 (1980) 53-64.

\bibitem{j82}G. J\"ager: {\sl Zur Beweistheorie der
Kripke--Platek Mengenlehre \"uber den nat\"ur\-li\-chen Zah\-len}.
Archiv f. Math. Logik  22 (1982) 121-139.

\bibitem{j86} G. J\"ager:
  {\em Theories for admissible sets: a unifying approach to proof
theory.} Bibliopolis, Naples, 1986

\bibitem{jp} G. J\"ager and W. Pohlers: {\em Eine beweistheoretische
Untersuchung von $\dbi$  und verwandter Systeme}, Sitzungsberichte
der Bayerischen Akademie der Wissenschaften,
Mathematisch--Naturwissenschaftliche Klasse (1982).

\bibitem{kreisel1} G. Kreisel: {\em On the interpretation of non-finitist proofs I.} Journal of Symbolic Logic 16 (1951) 241--267.


\bibitem{kreisel2} G. Kreisel: {\em On the interpretation of non-finitist proofs II.} Journal of Symbolic Logic 17 (1952) 43--58.

\bibitem{kreisel68} G. Kreisel: {\em A survey of proof theory}.
Journal of Symbolic Logic 33 (1968) 321-388.

\bibitem{kreisel-PA} G. Kreisel, G. Mints, S. Simpson: {\em The use of abstract language in elementary meta- mathematics: Some pedagogic examples.}
in: Lecture Notes in Mathematics, vol. 453 (Springer, Berlin, 1975) 38-131.

\bibitem{kripke} S. Kripke: {\em Transfinite recursion on admissible ordinals.} Journal of Symbolic logic 29 (1964) 161-162

\bibitem{lopes-pa} E. Lopez-Escobar: {\em An extremely restricted $\omega$-rule},
Fundamenta Mathematicae 90 (1976) 159-172.

%\bibitem{lub} Lubarsky, R {\em Independence results around constructive ZF.} Annals of Pure and Applied Logic, 132 (2005) %209-225.

\bibitem{michelbrink} M. Michelbrink: {\em  A Buchholz derivation system for the ordinal analysis of $KP+\Pi_3$-reflection}.
Journal of Symbolic Logic 71 (2006) 1237--1283.

\bibitem{mosc} Y.N. Moschovakis: {\em Recursion in the universe of sets},
 mimeographed note, 1976.

%\bibitem{myhill} J. Myhill: {\em Some properties of Intuitionistic Zermelo-Fraenkel set theory.} In: A. Mathiasand H. Rogers %(eds.):
%{\em Cambridge Summer School in Mathematical Logic}
%, volume 337 of
%Lectures Notes in Mathematics
%3(Springer, Berlin, 1973) 206- 231.

\bibitem{normann} D. Normann: {\em Set recursion}, in: Fenstad et al. (eds.):
{\em Generalized recursion theory II} (North-Holland, Amsterdam, 1978) 303-320.

\bibitem{platek} R. A. Platek: {\em Foundations of recursion theory} PhD Thesis, Stanford University, 1966 (219pp).

\bibitem{po} W. Pohlers: {\em Proof Theory}, Unversitext (Springer 2009).

\bibitem{po-ste} W. Pohlers, J.-C. Stegert: {\em Provably recursive functions of reflection}.
In: U.Berger, P. Schuster, M. Seisenberger (Eds.): {\em Logic, Construction, Computation} (Ontos Verlag, Frankfurt, 2012) 381--474.

\bibitem{r91}M. Rathjen: {\em Proof-Theoretic Analysis of KPM}, Arch. Math.
Logic 30 (1991) 377--403.

\bibitem{r94a}M. Rathjen: {\em Collapsing functions based on
recursively large ordinals: A well--ordering proof for KPM.} Archive for
Mathematical Logic 33  (1994) 35--55.
%\bibitem{rt}  M.~Rathjen, S.~Tupailo:  {\em Characterizing the interpretation
% of set theory in Martin-L\"of type theory.}
% Annals of Pure and Applied Logic 141 (2006) 442-471.

\bibitem{PRST} Michael Rathjen: {\em A Proof-Theoretic Characterization of the Primitive Recursive Set Functions.}
The Journal of Symbolic Logic
Vol. 57, No. 3 (1992), pp. 954-969

\bibitem{mi93} M. Rathjen: {\em How to develop proof--theoretic ordinal functions on the basis of admissible ordinals.} Mathematical Logic
Quarterly 39 (1993) 47-54.

\bibitem{mi94} M. Rathjen: {\em Proof theory of reflection.} Annals of Pure and Applied Logic 68 (1994) 181-224

\bibitem{mi95} M. Rathjen {\em Recent advances in ordinal analysis: $\Pi^1_2-CA$ and related systems.}
Bulletin of Symbolic Logic 1, (1995) 468-485

\bibitem{mi99realm} M. Rathjen: {\em The realm of ordinal analysis.} In:
 S.B. Cooper and J.K. Truss (eds.):
{\em Sets and Proofs.}
(Cambridge University Press, 1999) 219-279.

\bibitem{mi99} M. Rathjen: {\em Unpublished Lecture Notes on Proof Theory.} (1999)

\bibitem{R-CH} M. Rathjen: {\em Indefiniteness in semi-intuitionistic set theories: On a conjecture of Feferman}.
Journal of Symbolic Logic 81 %issue 02, pp.
(2016) 742--754.



%\bibitem{tklr} M. Rathjen: {\em The disjunction and other properties for constructive
%Zermelo-Fraenkel set theory.} Journal of Symbolic Logic
%70 (2005) 1233-1254.

\bibitem{mi2005} M. Rathjen: {\em An ordinal analysis of stability.} Archive for Mathematical Logic 44 (2005) 1-62

\bibitem{r2005} M. Rathjen: {\em An ordinal analysis of parameter-free $\Pi^1_2$
comprehension.} Archive for Mathematical Logic 44 (2005) 263-362.

\bibitem{repl} M. Rathjen: {\em Replacement versus collection in constructive
Zermelo-Fraenkel set  theory}. Annals of Pure and Applied Logic
136 %Issues 1-2, October
(2005) Pages 156--174.


%\bibitem{mi2008} M. Rathjen: {\em Metamathematical Properties of Intuitionistic  Set Theories
%with Choice Principles.}
%  In:  S. B. Cooper, B. L\"owe, A. Sorbi (eds.): {\em New Computational
%Paradigms: Changing Conceptions of What is Computable} (Springer, New
%York, 2008) 287-312.

%\bibitem{weak} M. Rathjen: {\em From the weak to the strong existence property}.
%Annals of Pure and Applied Logic 163 (2012) 1400–-1418.

%\bibitem{ML} M. Rathjen: {\em Constructive Zermelo-Fraenkel Set Theory, Power Set, and the
%Calculus of Constructions}. In:  Peter Dybjer, Sten Lindstr\"om, Erik Palmgren and Gö\"oran Sundholm: {\em Epistemology versus %ontology:} {\small Essays on the philosophy and foundations of mathematics in honour of Per Martin-L\"of} (Springer, Dordrecht, %Heidelberg, 2012) 313-349.

\bibitem{acend} M. Rathjen:  {\em Choice principles in constructive and classical set
theories}. In: Z. Chatzidakis, P. Koepke, W. Pohlers (eds.): {\em
Logic Colloquium '02}, Lecture Notes in Logic 27 (A.K. Peters,
2006) 299--326.

\bibitem{powerKP} M. Rathjen: {\em Relativized ordinal analysis: The case of Power Kripke-Platek set theory.}
Annals of Pure and Applied Logic 165 (2014) 316-393.

%\bibitem{rathjen-EP} M. Rathjen: {\em The existence property for intuitionistic set theories with collection}.
%In preparation.

\bibitem{ml} M.~Rathjen: {\em Constructive Zermelo-Fraenkel Set Theory, Power Set, and the Calculus of Constructions}.
In: P. Dybjer, S. Lindstr\"om, E. Palmgren and G. Sundholm: {\em Epistemology versus Ontology: Essays on the Philosophy and Foundations of Mathematics in Honour of Per Martin-L\"of}, (Springer, Dordrecht, Heidelberg, 2012) 313--349.

\bibitem{powerKPAC} M. Rathjen: {\em  Power Kripke-Platek set theory and the axiom of choice}. Submitted.
Also  published in the Isaac Newton Institute for Mathematical Sciences preprint series for the programme {\em `Mathematical, Foundational and Computational Aspects of the Higher Infinite'}.


\bibitem{sacks} G.E Sacks: {\em Higher recursion theory.}
(Springer, Berlin, 1990)

\bibitem{sch64} K. Sch\"utte: {\em Eine Grenze f\"ur die Beweisbarkeit der transfiniten Induktion in der verzweigten Typenlogik}.
Archiv f\"ur Mathematische Logik und Grundlagenforschung 67 (1964) 45-60.

\bibitem{sch65} K. Sch\"utte: {\em Predicative well-orderings}, in: Crossley, Dummett (eds.), {\em Formal systems and recursive functions} (North Holland, 1965) 176?184.

\bibitem{sch77} K. Sch\"utte: {\em Proof Theory.} (Springer, Berlin 1977).

\bibitem{schwichtenberg-PA} H. Schwichtenberg: {\em Proof theory: Some applications of cut-elimination.} In: J. Barwise
(ed.): {\em Handbook of Mathematical Logic.} (North Holland, Amsterdam, 1977) 867-895.

\bibitem{schwicht-wainer} H. Schwichtenberg, S.S. Wainer: {\em Proofs and Computations}. (Cambridge University Press, Cambridge, 2012).

%\bibitem{swan} A.W. Swan: {\em CZF does not have the existence property}.
%Annals of Pure and Applied Logic 165 (2014) 1115-1147.

%\bibitem{tait} W. W. Tait {\em Finitism.}
%Journal of Philosophy 78 (1981)  524-546

%\bibitem{tak67} G. Takeuti: {\em Consistency proofs of subsystems of classical analysis.} Annals of Mathematics 86 no.2, (1967)
%299-348.

%\bibitem{tak73} G. Takeuti, M. Yasugi: {\em The ordinals of the systems of second order arithmetic with the provably %$\Delta_2^1$-comprehension and the $\Delta_2^1$-comprehension axiom respectively.} Japan Journal of Mathematics. 41 (1973) 1-67.

\bibitem{tak87} G. Takeuti: {\em Proof theory, second edition.} (North Holland, Amsterdam, 1987).

\bibitem{weiermann-PA} A. Weiermann: {\em How to characterize provably total functions by local predicativity}.
Journal of Symbolic Logic 61 (1996) 52-69.

\end{thebibliography}
\end{document}